\documentclass[a4paper]{article}
\usepackage{algorithmicx,algorithm}
\usepackage[noend]{algpseudocode}
\usepackage{amsfonts,amsmath,amsopn,amssymb}
\usepackage{braket}
\usepackage{enumitem}
\usepackage{graphicx}
\usepackage{epstopdf}
\usepackage[caption=false]{subfig}
\usepackage[colorlinks,linkcolor=blue,anchorcolor=blue,citecolor=blue]{hyperref}
\usepackage{cleveref}
\usepackage{multirow}
\usepackage{ntheorem}
\usepackage{url}
\usepackage{xcolor}

\colorlet{shadecolor}{blue!20}

\numberwithin{equation}{section}

\theoremheaderfont{\normalfont\bf}
\theorembodyfont{\normalfont\rm}

\newtheorem{definition}{Definition}[section]
\newtheorem{remark}{Remark}[section]
\newtheorem{theorem}{Theorem}[section]
\newtheorem{corollary}{Corollary}[section]
\newtheorem{lemma}{Lemma}[section]
\newtheorem{proposition}{Proposition}[section]

\newenvironment{keywords}{{\noindent\it {\bf Key~words.}}\quad}{}
\newenvironment{proof}{{\noindent\it Proof.}\quad}{\hfill $\square$\par}

\DeclareMathOperator{\diag}{diag}

\DeclareMathOperator{\A}{\mathcal{A}}
\DeclareMathOperator{\B}{\mathcal{B}}

\begin{document}

\title{Noda Iteration for Computing Generalized Tensor Eigenpairs}

\author{
Wanli Ma\thanks{Email: 18110840010@fudan.edu.cn. School of Mathematical Sciences, Fudan University, Shanghai, 200433, P. R. of China. This author is supported by the National Natural Science Foundation of China under grant 12271108.}
\and
Weiyang Ding\thanks{Email: dingwy@fudan.edu.cn. Institute of Science and Technology for Brain-Inspired Intelligence, Fudan University, Shanghai, China; Shanghai Center for Brain Science and Brain-Inspired Technology, Shanghai, China; Key Laboratory of Computational Neuroscience and Brain-Inspired Intelligence (Fudan University), Ministry of Education, China; MOE Frontiers Center for Brain Science, Fudan University, Shanghai, China; Zhangjiang Fudan International Innovation Center. W. Ding’s research is supported by Shanghai Municipal Science and Technology Major Project (No. 20JC1419500, 2018SHZDZX01).}
\and
Yimin Wei\thanks{
Corresponding author (Y. Wei). Email: ymwei@fudan.edu.cn and yimin.wei@gmail.com. School of Mathematical Sciences and Shanghai Key Laboratory of Contemporary Applied Mathematics, Fudan University, Shanghai, 200433, P. R. of China. This author is supported by the National Natural Science Foundation of China under grant 12271108, the Innovation Program of Shanghai Municipal Education Committee and Shanghai Municipal Science and Technology Commission under grant 22WZ2501900.}
}

\date{}

\maketitle

\begin{abstract}
In this paper, we propose the tensor Noda iteration (NI) and its inexact version for solving the eigenvalue problem of a particular class of tensor pairs called generalized $\mathcal{M}$-tensor pairs. A generalized $\mathcal{M}$-tensor pair consists of a weakly irreducible nonnegative tensor and a nonsingular $\mathcal{M}$-tensor within a linear combination. It is shown that any generalized $\mathcal{M}$-tensor pair admits a unique positive generalized eigenvalue with a positive eigenvector. A modified tensor Noda iteration(MTNI) is developed for extending the Noda iteration for nonnegative matrix eigenproblems. In addition, the inexact generalized tensor Noda iteration method (IGTNI) and the generalized Newton-Noda iteration method (GNNI) are also introduced for more efficient implementations and faster convergence. Under a mild assumption on the initial values, the convergence of these algorithms is guaranteed. The efficiency of these algorithms is illustrated by numerical experiments.
\end{abstract}



\section{Introduction}
\
Tensor spectral theory and eigenvalue problems with a vast range of applications are widely investigated \cite{MR4333576,DW,QCC,QL18}. Variant versions of tensor eigenvalues are introduced from different aspects of generalizing from the matrix counterpart. Some recent papers \cite{CPZ,CDN,Kolda} point out that this generalized eigenvalue framework unifies several definitions of tensor eigenvalues. Generalized tensor eigenvalue problems have been extensively studied due to their wide applications such as higher-order Markov chain \cite{CHNS}, quantum information processing \cite{NQB}, and multilabel learning \cite{SJY}.

A tensor $\A=(a_{i_1\ldots i_m})$ is a multi-array of entries $a_{i_1\ldots i_m}\in \mathbf{F}$, where $i_j=1,\ldots,n_j$ for $j=1,\ldots,m$ and $\mathbf{F}$ is a field. In this paper, we only consider real tensors, i.e., $\mathbf{F}=\mathbb{R}$. When $n=n_1=\cdots=n_m$, $\A$ is called an $m$th order $n$-dimensional tensor. Denote the set of all $m$th order $n$-dimensional real tensors as $T_{m,n}$. For any tensor $\A\in T_{m,n}$ and any vector $\mathbf{x}\in\mathbb{R}^n$, the tensor-vector multiplication $\A\mathbf{x}^{m-1}$ is defined by
\begin{equation*}
\A\mathbf{x}^{m-1}=\left(\sum\limits_{i_2,\ldots,i_m=1}^na_{ii_2\cdots i_m}x_{i_2}\cdots x_{i_m}\right)\in\mathbb{R}^n.
\end{equation*}
The definition of tensor eigenvalues was proposed by Qi \cite{Q05} and Lim \cite{L05} independently in 2005. Let $\A=(a_{i_1\cdots i_m})\in T_{m,n}$. We call a number $\lambda\in\mathbb{C}$ an \emph{eigenvalue} of $\A$ if there exists a nonzero vector $\mathbf{x}\in\mathbb{C}^n$ satisfying the homogeneous polynomial equations:
\begin{equation}\label{equ:eigenvalue of tensor}
\A\mathbf{x}^{m-1}=\lambda\mathbf{x}^{[m-1]},
\end{equation}
where the notation $\mathbf{x}^{[m-1]}$ for $\mathbf{x}=(x_1,\ldots,x_n)^\top\in\mathbb{C}^n$ is defined by $\mathbf{x}^{[m-1]}=(x_1^{m-1},\ldots,x_n^{m-1})^\top$. Then we call the nonzero vector $\mathbf{x}$ an \emph{eigenvector} of $\A$ associated with the eigenvalue $\lambda$ and the pair $(\mathbf{x},\lambda)$ an \emph{eigenpair} of $\A$. The set of all eigenvalues of a tensor is called its \emph{spectrum}. The largest modulus of the elements in the spectrum of $\A$ is denoted as $\rho(\A)$.

Chang, Pearson, and Zhang \cite{CPZ} first introduced the generalized eigenvalues, called the $\mathcal{B}$-eigenvalues for a tensor $\mathcal{A}$ in their paper. Let $\mathcal{A}$ and $\mathcal{B}$ be two square tensors of the same size. Supposing that $\lambda\in\mathbb{C}$ and $\mathbf{x}\in \mathbb{C}^{n}$ satisfy
\begin{equation}
\mathcal{A}\mathbf{x}^{m-1}=\lambda\mathcal{B}\mathbf{x}^{m-1},~\mathbf{x}\neq\mathbf{0},
\end{equation}
we call $\lambda$ a \emph{$\mathcal{B}$--eigenvalue} of $\mathcal{A}$ and $\mathbf{x}$ the associated \emph{$\mathcal{B}$--eigenvector}. Ding and Wei \cite{DW15} further investigated the perturbation and error analysis of the generalized eigenvalue problem systematically.

Some frequently used notations are introduced as follows. For any real tensor $\mathcal{A}=(a_{i_1\ldots i_m})\in T_{m,n}$, we say that $\mathcal{A}$ is nonnegative, and write $\mathcal{A}\geq0$, if $a_{i_1\ldots i_m}\geq0$ for all $i_1,\ldots,i_m$. The tensor $\mathcal{A}$ is called positive, $\mathcal{A}>0$, if $a_{i_1\ldots i_m}>0$ for all $i_1,\ldots,i_m$. If $\mathcal{A},\mathcal{B}\in T_{m,n}$, then $\mathcal{A}\geq\mathcal{B}$ $(\mathcal{A}>\mathcal{B})$ means that $a_{i_1\ldots i_m}\geq b_{i_1\ldots i_m}(a_{i_1\ldots i_m}>b_{i_1\ldots i_m})$ for all $i_1,\ldots,i_m$. A nonnegative (positive) vector or matrix is defined in the same way.

For real vectors $\mathbf{x}=(x_1,x_2,\ldots,x_n)^\top$ and $\mathbf{y}=(y_1,y_2,\ldots,y_n)^\top$ with $y_i\neq0$ for all $i$, we use $\frac{\mathbf{x}}{\mathbf{y}}$ to denote the column vector whose $i$-th component is $\frac{x_i}{y_i}$. We also define $\max \mathbf{x}=\max\limits_i x_i$ and $\min \mathbf{x}=\min\limits_i x_i$. We denote $\mathbf{x}^{[m-1]}=(x_1^{m-1},\ldots,x_n^{m-1})^\top$ and $[n]=\{1,2,\ldots,n\}$. For simplicity, we use $\|\cdot\|$ to denote the 2--norm for vectors and matrices in this paper.

Several numerical methods have been proposed in the literature for computing generalized eigenpairs for different classes of tensor pairs. Kolda and Mayo \cite{Kolda} proposed a power method for computing the generalized eigenpairs for symmetric tensor pair. In Cui, Dai, and Nie \cite{CDN}, a semidefinite relaxation method was developed to find all real eigenvalues of symmetric tensor pairs. In \cite{CHZ,CHZ17}, Chen, Han and Zhou presented the homotopy methods for computing the (generalized) tensor eigenpair. Yu, Yu, Xu, Song, and Zhou \cite{YYXSZ} gave an adaptive gradient method for computing generalized tensor eigenpairs. Zhao, Yang and Liu \cite{ZYL17} computed the generalized eigenvalues of weakly symmetric tensors. Che, Cichocki and Wei \cite{che2017neural} applied the neural dynamical network to compute a best rank-one approximation of a real- valued tensor and solve the tensor eigenvalue problems. Mo, Wang and Wei \cite{mo2020time} explored time-varying generalized tensor eigenanalysis via Zhang neural networks.

In this paper, inspired by the work of Chen, Vong, Li, and Xu \cite{CVLX}, we will prove an extension of the Perron-Frobenius theory for a special kind of tensor pair and present iteration methods for finding the Perron pair of this special kind of tensor pair. Chen, Vong, Li, and Xu \cite{CVLX} considered the generalized eigenvalue problem of a special type of matrix pairs that exhibits the same kind of properties of nonnegative matrices provided by the Perron-Frobenius theory. Motivated by the nonlinear extension of the Perron-Frobenius theory in \cite{MoFu}, Fujimoto \cite{Fuji} considered the matrix generalized eigenproblem $A\mathbf{x}=\lambda B\mathbf{x}$ satisfying the following conditions:
\begin{enumerate}[label=(C\arabic*)]
\item $A\geq0.$

\item $A$ is irreducible.

\item There exists a vector $\mathbf{v}>0$ such that $B\mathbf{v}>A\mathbf{v}$.

\item For all $i\neq j$, $b_{ij}\leq a_{ij}$.
\end{enumerate}
Economic interpretation of these conditions is given in \cite{Fuji}. In \cite{BOD95}, the following extension of Perron-Frobenius theory was proved.

\begin{theorem}{\rm\cite{BOD95}}\label{thm:Perron Frobenius for matrix}
Let $A$ and $B$ be $n\times n$ matrices satisfying the condition $(\mathrm{C1})-(\mathrm{C4})$. Then there exist $\lambda\in(0,1)$ and a vector $\mathbf{x}_\ast$ such that $A\mathbf{x}_\ast=\lambda B\mathbf{x}_\ast$.

Furthermore, if $A\mathbf{v}=\lambda'B\mathbf{v}$ with a nonnegative $\lambda'$ and a nonzero nonnegative $\mathbf{v}$, then $\lambda=\lambda'$ and $\mathbf{v}=\alpha\mathbf{x}_\ast$ for some $\alpha>0$.
\end{theorem}

Chang, Pearson, and Zhang \cite{CPZ08} extended the Perron-Frobenius theory to the nonnegative tensor case. The spectral radius of an irreducible nonnegative tensor $\A$ is actually a positive eigenvalue with a positive eigenvector. This eigenpair is called the Perron pair of $\A$. It is related to the higher-order connectivity in hypergraphs \cite{HQ,HQX} and the stationary probability distribution of higher-order Markov chains \cite{NQZ09}.

In 1971, Noda \cite{Noda} introduced a positivity-preserving method--{\em Noda Iteration (NI)}--for computing the Perron pair of a nonnegative matrix. In \cite{CVLX}, NI was modified for a matrix pair $(A,B)$ satisfying the conditions $(\mathrm{C}1)-(\mathrm{C}4)$, including a modified Noda iteration (MNI) and a generalized Noda iteration (GNI). It is guaranteed that the associated generalized eigenvector is always positive. Furthermore, Noda iteration was also considered for finding the Perron pair for weakly irreducible nonnegative tensors \cite{LGL}.

In this paper, we consider the tensor pairs satisfying conditions analogous to $(\mathrm{C}1)-(\mathrm{C}4)$, which is referred to as the generalized $\mathcal{M}$-tensor pairs. We show that any generalized $\mathcal{M}$-tensor pair has a unique positive eigenvalue with a positive eigenvector, which is an extension of the matrix Perron-Frobenius theorem. The tensor Noda iteration is also designed for finding the Perron pair for this kind of tensor pair $(\mathcal{A},\mathcal{B})$.

The rest of this paper is organized as follows. First, we give the assumptions analogous to $(\mathrm{C1})-(\mathrm{C4})$ for tensor pairs and prove Perron-Frobenius-type theory for the tensor pairs in Section 2. Based on this extended theory, we propose the tensor Noda iteration with practical modifications for finding the Perron pair of this kind of tensor pair in Section 3. Next, we analyze the convergence of these algorithms in Section 4. Finally, we present numerical examples to demonstrate the effectiveness and convergence behavior of our methods in Section 5.

\section{Tensor eigenproblem for generalized $\mathcal{M}$-tensor pair}

Similar to the matrix case in \cite{CVLX}, we investigate the tensor generalized eigenproblems with some special structures. To present the conditions in tensor case, we need to introduce several concepts of tensor irreducibility.

\begin{definition}\label{def:irreducible tensor}
A tensor $\mathcal{A}\in T_{m,n}$ is said to be {\em reducible} if there is a nonempty proper index subset $J\subset\{1,2,\ldots,n\}$ such that
$$a_{i_1\ldots i_m}=0,~\forall i_1\in J,~\forall i_2,\ldots,i_m\notin J.$$

$\mathcal{A}$ is called {\em irreducible} if it is not reducible. In addition, a tensor $\A\in T_{m,n}$ is called \emph{weakly irreducible} if for every nonempty proper index subset $S\subset\{1,2,\ldots,n\}$ there exist $i_1\in S$ and $i_2,\ldots,i_m$ with at least one $i_q\notin S$, $q=2,\ldots,m$, such that $a_{i_1i_2\ldots i_m}\neq0$.
\end{definition}

We consider the generalized eigenproblem of real tensor pair $(\mathcal{A,B})$ under the following conditions:
\begin{enumerate}[label=(C\arabic*')]
\item $\mathcal{A}\geq0$.

\item $\A$ is weakly irreducible.

\item There exists a vector $\mathbf{v}\in\mathbb{R}^n,~\mathbf{v}>0$ such that $\mathcal{B}\mathbf{v}^{m-1}>\mathcal{A}\mathbf{v}^{m-1}$.

\item For all $(i_2,\ldots,i_m)\neq(i,\ldots,i)$, $b_{ii_2\ldots i_m}\leq a_{ii_2\ldots i_m}$.
\end{enumerate}

Considering third order tensors for example, the economics interpretations of the above assumptions can be made as follows. Suppose that there are $n$ kinds of goods, $n$ industries, and $n$ kinds of techniques available to produce these goods. The elements $a_{ijk}$ of $\A$ and $b_{ijk}$ of $\B$ stand for input and output quantity of the $i$-th goods used by the $k$-th technique of the $j$-th industry. Thus, $(\mathrm{C3'})$ tells that the technology is productive enough to produce a surplus in each goods. In addition, $(\mathrm{C1'})$ and $(\mathrm{C2'})$ implies that every technique used by every industry needs every goods directly or indirectly. Besides, $(\mathrm{C4'})$ means that there are no net joint products.

For simplicity, we will call the tensor pair $(\A,\B)$ satisfying conditions $(\mathrm{C1'})-(\mathrm{C4'})$ as ``\emph{generalized $\mathcal{M}$-tensor pair}'' in the following contents.

In order to prove the extension of Perron-Frobenius theory for the generalized $\mathcal{M}$-tensor pair $(\A,\B)$, we further present some properties of $\mathcal{M}$-tensors.

First, we introduce the definition of a $\mathcal{M}$-tensor. A tensor $\mathcal{D}=(d_{i_1\ldots i_m})\in T_{m,n}$ is called a \emph{diagonal tensor} if its entries are
\begin{equation}\label{equ:diagonal tensor}
d_{i_1\ldots i_m}=\left\{
\begin{aligned}
&d_{i\ldots i}, \quad&\text{if}~(i_1,\ldots,i_m)=(i,\ldots,i)\\
&0, \quad&\text{otherwise}
\end{aligned}\right.
\end{equation}
The entries $d_{i\ldots i}~(i\in[n])$ are called \emph{diagonal entries} and the others are called \emph{off-diagonal entries}. If $d_{i\ldots i}=1$ for $i\in[n]$, then $\mathcal{D}$ is called the \emph{unit tensor}.

A real tensor $\mathcal{A}$ is a $\mathcal{Z}$--{\em tensor} if all its off--diagonal entries are nonpositive, which is equivalent to $\mathcal{A}=s\mathcal{I}-\mathcal{B}$, where $\mathcal{I}$ is the unit tensor and $\mathcal{B}$ is a nonnegative tensor. A $\mathcal{Z}$--tensor $\mathcal{A}=s\mathcal{I}-\mathcal{B} (\mathcal{B}\geq0)$ is called a $\mathcal{M}$--{\em tensor} if $s\geq\rho(\mathcal{B})$, and we call it as a {\em nonsingular} $\mathcal{M}$--{\em tensor} if $s>\rho(\mathcal{B})$.

Combining the results in \cite{DW13,ZQZ} and \cite[pages 81-96]{DW}, there are dozens of equivalent definitions for nonsingular $\mathcal{M}$-tensors. We only mention a few that will be used in this work as follows.

\begin{proposition}\label{prop:nonsingular M-tensor}
If $\mathcal{A}$ is a $\mathcal{Z}$--tensor, then the following conditions are equivalent:

$\rm(1)$ $\mathcal{A}$ is a nonsingular $\mathcal{M}$--tensor.

$\rm(2)$ There exists $\mathbf{x}>0$ with $\mathcal{A}\mathbf{x}^{m-1}>0.$

$\rm(3)$ There exists $\mathbf{x}\geq0$ with $\mathcal{A}\mathbf{x}^{m-1}>0.$
\end{proposition}

Thus for a generalized $\mathcal{M}$-tensor pair $(\A,\B)$, $\B-\A$ is a $\mathcal{Z}$-tensor since $b_{ii_2\ldots i_m}\leq a_{ii_2\ldots i_m}$ for all $(i_2,\ldots,i_m)\neq(i,\ldots,i)$. Moreover, \Cref{prop:nonsingular M-tensor} indicates that $\B-\A$ is a nonsingular $\mathcal{M}$-tensor since $(\mathrm{C}3')$ can be rewritten to $(\B-\A)\mathbf{x}^{m-1}>0$ with a positive vector $\mathbf{x}$. Thus we can denote $\B-\A=s\mathcal{I}-\mathcal{C}$, where $\mathcal{I}$ is the unit tensor, $\mathcal{C}$ is nonnegative, and $s>\rho(\mathcal{C})$.

Let $\A$ and $\B$ be two $m$th-order tensors in $\mathbb{C}^{n\times\cdots\times n}$. We call the tensor pair $(\A,\B)$ a \emph{regular tensor pair}, if $\mathrm{det}(\beta\A-\alpha\B)\neq0$ for some $(\alpha,\beta)\in\mathbb{C}_{1,2}$. Reversely, we call $(\A,\B)$ a \emph{singular tensor pair}, if $\mathrm{det}(\beta\A-\alpha\B)=0$ for all $(\alpha,\beta)\in\mathbb{C}_{1,2}$. Here the determinant of tensor is defined by Qi \cite{HHLQ} \cite[page 23]{QL18} and $\mathbb{C}_{1,2}$ denotes a projective plane, in which $(\alpha_1,\beta_1),(\alpha_2,\beta_2)\in\mathbb{C}\times\mathbb{C}$ are regarded as the same point, if there is a nonzero scalar $\gamma\in\mathbb{C}$ such that $(\alpha_1,\beta_1)=(\gamma\alpha_2,\gamma\beta_2)$. If a tensor pair is singular, then any nonzero complex number will be its eigenvalue. Therefore, before proceeding our proof for the extension of Perron-Frobenius theory for the generalized $\mathcal{M}$-tensor pair $(\A,\B)$, we should mention that $(\A,\B)$ is a regular tensor pair. Actually, if $(\A,\B)$ is a singular tensor pair, then $\det(\B-\A)=0$. According to \cite[Theorem 3.1]{HHLQ}, there exists a vector $\mathbf{x}\in\mathbb{C}^n\backslash\{\mathbf{0}\}$ such that $(\B-\A)\mathbf{x}^{m-1}=0$. This implies that $(s\mathcal{I}-\mathcal{C})\mathbf{x}^{m-1}=0$ or equivalently, $\mathcal{C}\mathbf{x}^{m-1}=s\mathbf{x}^{[m-1]}$. Thus $s$ is a eigenvalue of $\mathcal{C}$, which contradicts the condition $s>\rho(\mathcal{C})$. So $(\A,\B)$ is a regular tensor pair.

For any tensor pair $(\mathcal{A},\mathcal{B})$, $(s,\mathbf{x})$ is an eigenpair of the tensor pair $(\mathcal{A},\mathcal{B-A})$ if and only if $(\frac{s}{1+s},\mathbf{x})$ is an eigenpair of $(\mathcal{A,B})$. Thus, we can consider the eigenproblem $\mathcal{A}\mathbf{x}^{m-1}=s(\mathcal{B}-\mathcal{A})\mathbf{x}^{m-1}$ instead of $\mathcal{A}\mathbf{x}^{m-1}=\lambda\mathcal{B}\mathbf{x}^{m-1}$. We will show later that the existence of the positive solution for the first problem is guaranteed under proper conditions. We summarize the relations between these two eigenproblems in the following lemma.

\begin{lemma}\label{lemma:relation between two eigenproblem}
For a tensor pair $(\mathcal{A,B})$, denote $\mathcal{C=B-A}$. Let $s$ be a generalized eigenvalue of tensor pair $(\mathcal{A,C})$, then $\lambda=\frac{s}{1+s}$ is a generalized eigenvalue of $\mathcal{(A,B)}$. The relation between $\lambda$ and $s$ can be characterized as follows.

$\rm(1)$ $s\geq0$ if and only if $\lambda\in[0,1)$;

$\rm(2)$$-1\neq s<0$ if and only if $\lambda\in(-\infty,0)\cup(1,+\infty)$;

$\rm(3)$ $s=-1$ if and only if $\lambda=\infty$;

$\rm(4)$ $s$ is complex and $\mathrm{Im}(s)\neq0$ if and only if $\lambda$ is complex and $\mathrm{Im}(\lambda)\neq0$, where $\mathrm{Im}(z)$ denotes the imaginary part of a complex number $z$.
\end{lemma}

Now we can extend the Perron-Frobenius theory to the generalized $\mathcal{M}$-tensor pair case. First, we prove the existence of a nonnegative solution $(s,\mathbf{x})$ for the generalized eigenproblem $\A\mathbf{x}^{m-1}=s(\B-\A)\mathbf{x}^{m-1}$. We need the following lemmas about $\mathcal{M}$-tensor equations.

\begin{lemma}\label{lemma:M-equations}{\rm\cite[Theorem 3.2]{DW16}\cite[Theorem 6.3]{DW}}
If $\mathcal{T}$ is a nonsingular $\mathcal{M}$--tensor, then for every positive vector $\mathbf{b}$, the multilinear system of equations $\mathcal{T}\mathbf{x}^{m-1}=\mathbf{b}$ has a unique positive solution.
\end{lemma}

\begin{lemma}\label{lemma:monotony of M equation}
Let $\mathcal{T}$ be a nonsingular $\mathcal{M}$-tensor. If $\mathcal{T}\mathbf{x}^{m-1}\geq\mathcal{T}\mathbf{y}^{m-1}>0$, then we have $\mathbf{x}\geq\mathbf{y}>0$.
\end{lemma}

\begin{proof}
By definition, we can denote $\mathcal{T}=\mu\mathcal{I}-\mathcal{N}$, where $\mathcal{N}$ is a nonnegative tensor and $\mu>\rho(\mathcal{N})$. Denote $\mathcal{T}\mathbf{x}^{m-1}=\mathbf{b}$ and $\mathcal{T}\mathbf{y}^{m-1}=\mathbf{c}$.
By \Cref{lemma:M-equations}, we have $\mathbf{x}>0$ and $\mathbf{y}>0$.
According to \cite[page 702]{DW16}, the iteration
\begin{displaymath}
\mathbf{x}_k=(\mu^{-1}\mathcal{N}\mathbf{x}_{k-1}^{m-1}+\mu^{-1}\mathbf{b})^{[1/(m-1)]},~k=1,2,\ldots
\end{displaymath}
converges to the unique positive solution of $(\mu\mathcal{I}-\mathcal{N})\mathbf{x}^{m-1}=\mathbf{b}>0$. The same is true for $(\mu\mathcal{I}-\mathcal{N})\mathbf{y}^{m-1}=\mathbf{c}>0$. If we set $\mathbf{x}_0=\mathbf{y}_0$, then
\begin{displaymath}
\mathbf{x}_1=(\mu^{-1}\mathcal{N}\mathbf{x}_0^{m-1}+\mu^{-1}\mathbf{b})^{[1/(m-1)]}\geq(\mu^{-1}\mathcal{N}\mathbf{y}_0^{m-1}+\mu^{-1}\mathbf{c})^{[1/(m-1)]}=y_1.
\end{displaymath}
By induction, we can see that $\mathbf{x}_k\geq\mathbf{y}_k$ holds for all $k=1,2,\ldots$. Therefore $\mathbf{x}=\lim\limits_{k\rightarrow\infty}\mathbf{x}_k\geq\lim\limits_{k\rightarrow\infty}\mathbf{y}_k=\mathbf{y}$.
\end{proof}

\begin{theorem}\label{thm:exist of nonnegative eigenpair}
Let $(\A,\B)$ be a generalized $\mathcal{M}$-tensor pair, then there exists a positive eigenpair $(s,\mathbf{x}_{\ast})$ for the tensor pair $(\A,\B-\A)$.
\end{theorem}

\begin{proof}
Denote $\mathbb{R}_{+}^{n}=\{\mathbf{x}\in\mathbb{R}^n|\mathbf{x}\geq0\}$ and $\mathbb{R}_{++}^{n}=\{\mathbf{x}\in\mathbb{R}^n|\mathbf{x}>0\}$. We define the operator $(\B-\A)_{++}^{-1}:\mathbb{R}_{++}^n\rightarrow\mathbb{R}_{++}^n$ to be $(\B-\A)_{++}^{-1}\mathbf{b}=\mathbf{x}$,
where $\mathbf{x}$ is the unique positive solution for the equation $(\B-\A)\mathbf{x}^{m-1}=\mathbf{b}$. According to \Cref{lemma:M-equations}, $\mathbf{x}$ exists and is unique since $(\B-\A)$ is a nonsingular $\mathcal{M}$--tensor, thus $(\B-\A)_{++}^{-1}$ is well-defined. The continuity of $(\B-\A)_{++}^{-1}$ is implied by the continuity of the roots of polynomials with respect to the coefficients.

Define the operator $F:\mathbb{R}_{++}^n\rightarrow\mathbb{R}_{++}^n$ by $F(\mathbf{x})=(\B-\A)_{++}^{-1}(\A\mathbf{x}^{m-1})$. Then it is easy to verify that $F$ is homogeneous, that is, for any $t>0$ and $\mathbf{x}\in\mathbb{R}_{++}^n$, $F(t\mathbf{x})=(\B-\A)_{++}^{-1}\A(t\mathbf{x})^{m-1}$
$=t(\B-\A)_{++}^{-1}(\A\mathbf{x}^{m-1})=tF(\mathbf{x})$. Besides, according to \Cref{lemma:monotony of M equation}, $F$ is a monotone function. Referring to \cite{Gaubert04}, we define the associated graph of $F$, $G(F)$, to be the directed graph with vertices $1,\ldots,n$ and an edge from $i$ to $j$ if and only if $\lim\limits_{u\rightarrow\infty}(F(\mathbf{u}_{\{j\}}))_i=\infty$, where $\mathbf{u}_{\{j\}}$ is defined by $(\mathbf{u}_{\{j\}})_i=u>0$ for $i=j$ and $(\mathbf{u}_{\{j\}})_i=1$ for $i\neq j$. Define the operator $F_{\A}:\mathbb{R}_{++}^n\rightarrow\mathbb{R}_{++}^n$ by $F_{\A}(\mathbf{x})=\left(\A\mathbf{x}^{m-1}\right)^{[1/(m-1)]}$, then $F_{\A}$ is homogeneous and monotone. The associated graph of $F_{\A}$, $G(F_{\A})$, is defined in the same way as $G(F)$. According to the proof of \cite[Theorem 3.2]{DW16}, $(\A\mathbf{u}_{\{j\}}^{m-1})_i\rightarrow\infty$ implies that $\big((\B-\A)_{++}^{-1}(\A\mathbf{u}_{\{j\}}^{m-1})\big)_i\rightarrow\infty$. Thus, the associated graph $G(F_{\A})$ of $F_{\A}$ is a spanning subgraph of the assoicated graph $G(F)$ of $F$. By \cite[Lemma 3.2]{Friedland13} and condition $(\mathrm{C2'})$, $G(F_{\A})$ is strongly connected and hence $G(F)$ is strongly connected. By using \cite[Theorem 2]{Gaubert04}, $F$ has an eigenvector in $\mathbb{R}_{++}^n$ with a positive eigenvalue. That is, there exist $\mathbf{x}_{\ast}\in\mathbb{R}_{++}^n$ and $s>0$ such that $(\B-\A)_{++}^{-1}(\A\mathbf{x}_{\ast}^{m-1})=s^{1/(m-1)}\mathbf{x}_{\ast}$, or equivalently, $\A\mathbf{x}_{\ast}^{m-1}=s(\B-\A)\mathbf{x}_{\ast}^{m-1}$.
\end{proof}

Based on \Cref{lemma:M-equations,lemma:monotony of M equation}, we can prove the uniqueness of the positive eigenvalue for any generalized $\mathcal{M}$-tensor pair.

\begin{remark}
We need to assume that tensors $\A$ and $\B$ are semisymmetric, i.e., $\mathcal{A}_{ii_2\ldots i_m}=\mathcal{A}_{ij_2\ldots j_m},\mathcal{B}_{ii_2\ldots i_m}$
$=\mathcal{B}_{ij_2\ldots j_m},1\leq i\leq n$, $j_2\ldots j_m$ is any permutation of $i_2\ldots i_m$, $1\leq i_2,\ldots,i_m\leq n$ so that we can compute the partial derivatives $D\A\mathbf{x}^{m-1},D\B\mathbf{x}^{m-1}$ and put the results in simple forms. For any tensor $\A$, we can get a semisymmetric tensor $\overline{\A}$ such that $\A\mathbf{x}^{m-1}=\overline{\A}\mathbf{x}^{m-1}$, by defining $\overline{\A}$ as
\begin{equation}\label{equ:definition of overline A}
\overline{a}_{i_1i_2\ldots i_m}=\frac{1}{(m-1)!}\sum\limits_{j_2\ldots j_m}a_{i_1j_2\ldots j_m},
\end{equation}
where $j_2\ldots j_m$ is any permutation of $i_2\ldots i_m$.
Note that the tensors $\A$ and $\B$ in this paper will always enter the discussion through terms in the form $\A\mathbf{x}^{m-1}$ and $\B\mathbf{x}^{m-1}$, directly or indirectly, so we can just assume that they are semisymmetric, unless specified otherwise.
\end{remark}

\begin{theorem}\label{thm:uniqueness of the eigenvalue}
Suppose that $(s,\mathbf{x}_{\ast})$ is the positive eigenpair in \Cref{thm:exist of nonnegative eigenpair}, then $s$ is the unique nonnegative generalized eigenvalue of tensor pair $(\A,\B-\A)$ with a nonnegative generalized eigenvector, and $\mathbf{x}_{\ast}$ is the unique nonnegative generalized eigenvector associated with $s$, up to a multiplicative constant.
\end{theorem}

\begin{proof}
First, we prove that $s$ is unique. Suppose that $\A\mathbf{y}_{\ast}^{m-1}=t(\B-\A)\mathbf{y}_{\ast}^{m-1}$ with $t>0$ and $\mathbf{y}_{\ast}>0$. Let $\sigma=\max\{\alpha:\mathbf{x}_{\ast}-\alpha\mathbf{y}_{\ast}\geq0\}$, then $\sigma>0$ and
\begin{displaymath}
s(\B-\A)\mathbf{x}_{\ast}^{m-1}=\A\mathbf{x}_{\ast}^{m-1}\geq\A(\sigma\mathbf{y}_{\ast})^{m-1}
=\sigma^{m-1}t(\B-\A)\mathbf{y}_{\ast}^{m-1}>0.
\end{displaymath}
By \Cref{lemma:monotony of M equation}, this implies that $\mathbf{x}_{\ast}\geq\sigma(t/s)^{1/(m-1)}\mathbf{y}_{\ast}$and hence $t\leq s$. If we exchange $s$ and $t$, $\mathbf{x}_{\ast}$ and $\mathbf{y}_{\ast}$, we have $s\leq t$. Thus, $s=t$, i.e., $s$ is unique.

Next, we show that $\mathbf{x}_{\ast}$ is unique, up to a multiplicative constant. Consider the matrix $D=DF(\mathbf{x}_{\ast})=[\frac{\partial F_i}{\partial x_j}(\mathbf{x}_{\ast})]_{i,j=1}^n$. Denote $(\B-\A)\mathbf{y}^{m-1}=\A\mathbf{x}_{\ast}^{m-1}$, by the inverse function theorem, we can deduce the expression $DF(\mathbf{x}_{\ast})=\big[(\B-\A)\mathbf{y}^{m-2}\big]^{-1}\A\mathbf{x}_{\ast}^{m-2}$. It is easy to see that $(\B-\A)\mathbf{y}^{m-2}$ is a nonsingular M-matrix. Similar to the analysis in \cite[Theorem 3.3]{Friedland13}, the di-graph $G(F)$ is a spanning subgraph of the di-graph $G(D)$, which is induced by the nonnegative entries of $D$. Thus $G(D)$ is strongly connected and according to \cite[Theorem 1]{You19}, $D$ is irreducible. Therefore, by \cite[Theorem 2.2]{Friedland13} \cite[Theorem 2.5]{Nussbaum88}, the eigenvector $\mathbf{x}_{\ast}$ is unique.
\end{proof}

\Cref{thm:exist of nonnegative eigenpair,thm:uniqueness of the eigenvalue} form a tensor pair version of the Perron-Frobenius theory, which locates at the core for constructing our algorithms and proving their convergence. Based on the above results, we can deduce another useful conclusion, which is similar to the Collatz-Wielandt theorem.

\begin{theorem}\label{thm: Collatz-Wielandt}
Suppose that $(\A,\B)$ is a generalized $\mathcal{M}$-tensor pair. Then for any $\mathbf{x}\in\mathbb{R}_+^n\backslash\{0\}$ such that $(\B-\A)\mathbf{x}^{m-1}>0$, we have
\begin{equation}\label{equ:Collatz-Wielandt s}
\min\limits_i\frac{(\A\mathbf{x}^{m-1})_i}{\big((\B-\A)\mathbf{x}^{m-1}\big)_i}\leq s\leq\max\limits_i\frac{(\A\mathbf{x}^{m-1})_i}{\big((\B-\A)\mathbf{x}^{m-1}\big)_i},
\end{equation}
where $s$ is the unique positive generalized eigenvalue of the tensor pair $(\A,\B-\A)$ corresponding to a positive generalized eigenvector.
\end{theorem}

\begin{proof}
Denote $\mathbf{x}_{\ast}$ as the unique positive eigenvector corresponding to $s$. First, we prove that if $t\geq0$ and nonzero vector $\mathbf{x}\in\mathbb{R}_+^n$ satisfy $\A\mathbf{x}^{m-1}\geq t(\B-\A)\mathbf{x}^{m-1}>0$, then $t\leq s$.

Let $\sigma=\max\{\alpha:\mathbf{x}_{\ast}-\alpha\mathbf{x}\geq0\}$, then $\sigma>0$. We have
\begin{displaymath}
s(\B-\A)\mathbf{x}_{\ast}^{m-1}=\A\mathbf{x}_{\ast}^{m-1}\geq\sigma^{m-1}\A\mathbf{x}^{m-1}\geq\sigma^{m-1}t(\B-\A)\mathbf{x}^{m-1}>0.
\end{displaymath}
This implies that $\mathbf{x}_{\ast}\geq(t/s)^{1/(m-1)}\sigma\mathbf{x}$ and hence $t\leq s$. We can immediately get the left side of \eqref{equ:Collatz-Wielandt s} by letting $t=\min\limits_i\frac{(\A\mathbf{x}^{m-1})_i}{\big((\B-\A)\mathbf{x}^{m-1}\big)_i}$ in the above result.

The proof for the other side is analogous.
\end{proof}

From the relationship $\lambda=\frac{s}{1+s}$, we can get the following corollary immediately.

\begin{corollary}\label{cor: Collatz-Wielandt}
Suppose that ($\A,\B$) is a generalized $\mathcal{M}$-tensor pair. Then for any $\mathbf{x}\in\mathbb{R}_+^n\backslash\{0\}$ such that $(\B-\A)\mathbf{x}^{m-1}>0$, we have
\begin{displaymath}
\min\limits_i\frac{(\A\mathbf{x}^{m-1})_i}{(\B\mathbf{x}^{m-1})_i}\leq\lambda\leq\max\limits_i\frac{(\A\mathbf{x}^{m-1})_i}{(\B\mathbf{x}^{m-1})_i},
\end{displaymath}
where $\lambda$ is the unique positive generalized eigenvalue of the tensor pair $(\A,\B)$ corresponding to a positive generalized eigenvector.
\end{corollary}




Before proposing the Noda iterative methods for computing the Perron pair of the generalized $\mathcal{M}$-tensor pair $(\A,\B)$, we need to prove that tensor $\mu\B-\A$ is a nonsingular $\mathcal{M}$-tensor for any $\mu\in(\lambda,1]$, where $\lambda$ is the unique positive generalized eigenvalue for $(\A,\B)$.

\begin{theorem}\label{thm:mu B-A is also M-tensor}
Let $(\A,\B)$ be a generalized $\mathcal{M}$-tensor pair. Suppose that $\lambda=\frac{s}{s+1}$, where $s$ is the unique positive generalized eigenvalue of tensor pair $(\A,\B-\A)$ corresponding to a positive generalized eigenvector. Then for any $\mu\in(\lambda,1]$, $\mu\B-\A$ is a nonsingular $\mathcal{M}$-tensor.
\end{theorem}

\begin{proof}
Since $s>0$, we have $0<\lambda<1$. Thus $\mu\B-\A$ is a $\mathcal{Z}$-tensor according to the condition $\mathrm{(C4')}$. Based on \Cref{thm:exist of nonnegative eigenpair,thm:uniqueness of the eigenvalue}, there exists a vector $\mathbf{x}>0$ such that $(\lambda\B-\A)\mathbf{x}^{m-1}=0$ and hence $\B\mathbf{x}^{m-1}=\frac{1}{\lambda}\A\mathbf{x}^{m-1}>0$ since $\A$ is a weakly irreducible nonnegative tensor. Furthermore, since $\mu>\lambda$, we have
\begin{displaymath}
(\mu\B-\A)\mathbf{x}^{m-1}=(\mu-\lambda)\B\mathbf{x}^{m-1}+(\lambda\B-\A)\mathbf{x}^{m-1}=(\mu-\lambda)\B\mathbf{x}^{m-1}>0.
\end{displaymath}
By \Cref{prop:nonsingular M-tensor}, $\mu\B-\A$ is a nonsingular $\mathcal{M}$-tensor.
\end{proof}

\begin{remark}
The tensor pair satisfying conditions $\mathrm{(C1')-(C4')}$ can be found in generalized eigenvalue problems of directed hypergraphs. Consider a strongly connected directed uniform hypergraph $H=(V,E)$. Suppose that $H$ is a $k$-graph, which means $|e_i|=k$ for every arc $e_i$ in the arc set $E=\{e_1,\ldots,e_m\}$. Denote $\A$ as the adjacency tensor of $H$ and $\B$ as the signless Laplacian tensor $\mathcal{D}+\A$. The \emph{adjacency tensor} $\A$ of the directed $k$-graph $H$ is defined as a $k$-order $n$-dimensional tensor whose $(i_1,\ldots,i_k)$ entry is 
\begin{displaymath}
a_{i_1\ldots i_k}=\left\{\begin{aligned}
&\frac{1}{(k-1)!} & {\rm if~} (i_1,\ldots,i_k)=e\in E {\rm~and~} i_1 {\rm~is~the~tail~of~} e,\\
&0 & {\rm otherwise}.\\
\end{aligned}\right.
\end{displaymath}
A vertex $i_1$ is called the \emph{tail} of an arc $e$ if it is in the first position of $e$, that is, $e=(i_1,i_2,\ldots,i_k)$. The diagonal tensor $\mathcal{D}$ is defined as $d_{i\ldots i}=d_i^+$, the out-degree of vertex $i$, for all $i\in[n]$. The \emph{out-degree} of a vertex $i\in V$ is defined as $d_i^+=|E_i^+|$, where $E_i^+=\{e\in E:i{\rm~is~the~tail~of~}e\}$. Interested readers can refer to \cite[page 163]{QL18} for more detailed definitions about directed uniform hypergraphs. According to \cite[Theorem 4.59]{QL18}, $\A$ is a weakly irreducible nonnegative tensor. Besides, $\B-\A=\mathcal{D}$ is a positive diagonal tensor and hence an $\mathcal{M}$-tensor. The generalized tensor eigenproblem $\A\mathbf{x}^{m-1}=\lambda\B\mathbf{x}^{m-1}$ can be regarded as a higher-order generalization of the graph embedding problems \cite{Cai07,Cai18}, similar to the matrix case discussed in \cite{SJY}. Besides, if $\B=\mathcal{E}$, the identity tensor such that $\mathcal{E}\mathbf{x}^{m-1}=\|\mathbf{x}\|^{m-2}\mathbf{x}$ for all $\mathbf{x}\in\mathbb{R}^n$ \cite{CPZ}, then the generalized eigenvalue problem $\A\mathbf{x}^{m-1}=\lambda\B\mathbf{x}^{m-1}$ is equivalent to a $Z-$eigenvalue problem \cite[page 26]{QL18}, which can be used for computing the Fieldler vector of a Laplacian tensor \cite{Chen17}.
\end{remark}

\section{Noda iteration for generalized tensor eigenproblem}

The original NI is designed for the Perron pair of an irreducible nonnegative matrix, which iterates \cite{Noda}:
\begin{enumerate}[label=(\arabic*)]
\item $\mathbf{y}_k=(\rho_{k-1}I-C)^{-1}\mathbf{x}_{k-1}$;
\item $\rho_k=\rho_{k-1}-\min\frac{\mathbf{x}_{k-1}}{\mathbf{y}_k}$;
\item $\mathbf{x}_k=\mathbf{y}_k/\|\mathbf{y}_k\|$.
\end{enumerate}

Based on \Cref{thm:Perron Frobenius for matrix}, a modified Noda iteration (MNI) for computing the Perron pair of a matrix pair $(A,B)$ satisfying conditions $\mathrm{(C1)-(C4)}$ has been proposed in \cite{CVLX}. Analogous to the step $(1)$ and step $(2)$ above, it updates $\mathbf{y}_{k+1}$ by $(\rho_kB-A)\mathbf{y}_{k+1}=(B-A)\mathbf{x}_k$ and updates $\rho_{k+1}$ by $\rho_{k+1}=\rho_k-(1-\rho_k)\frac{\tau_k}{1-\tau_k}$, where $\tau_k=\min\frac{\mathbf{x}_k}{\mathbf{y}_{k+1}}$. Based on the Collatz-Wielandt theorem, a generalized Noda iteration (GNI) has also been introduced in \cite{CVLX}. Compared with MNI, it updates $\mathbf{y}_{k+1}$ by $(\rho_kB-A)\mathbf{y}_{k+1}=A\mathbf{x}_k$ and updates $\rho_{k+1}$ by $\rho_{k+1}=\rho_k(1-\min\frac{A\mathbf{x}_k}{A\mathbf{y}_{k+1}+A\mathbf{x}_k})$. As mentioned in \cite{CVLX}, one advantage of GNI is that we can always simply select $\rho_0=1$, while in MNI $\rho_0$ must be smaller than 1. Otherwise, $\rho_k=1$ for all $k$ and MNI fails.

\subsection{Modified Tensor Noda Iteration (MTNI)}

For a generalized $\mathcal{M}$-tensor pair $\mathcal{(A,B)}$, based on \Cref{thm:exist of nonnegative eigenpair,thm:uniqueness of the eigenvalue}, its unique positive generalized eigenpair can be computed by solving the problem $\mathcal{A}\mathbf{x}^{m-1}=s(\mathcal{B-A})\mathbf{x}^{m-1}$. Analogous to MNI, we can construct the iteration scheme:
\begin{equation}\label{equ:iteration format 2}
\big[s_{k-1}(\B-\A)-\A\big]\mathbf{z}_k^{m-1}=(\B-\A)\mathbf{x}_{k-1}^{m-1}.
\end{equation}
Denote $\mathbf{y}_k=(1+s_{k-1})^{1/(m-1)}\mathbf{z}_k$ and $\rho_{k-1}=s_{k-1}/(1+s_{k-1})$, then \eqref{equ:iteration format 2} becomes
\begin{equation}\label{equ:iteration format 3}
(\rho_{k-1}\B-\A)\mathbf{y}_k^{m-1}=(\B-\A)\mathbf{x}_{k-1}^{m-1}.
\end{equation}

According to \Cref{lemma:M-equations} and \Cref{thm:mu B-A is also M-tensor}, if we can ensure that $(\B-\A)\mathbf{x}_{k-1}^{m-1}>0$ and $\rho_{k-1}>\lambda$, where $\lambda$ is the unique positive generalized eigenvalue for the tensor pair $(\A,\B)$, then we can find a unique $\mathbf{y}_k>0$ and $(\B-\A)\mathbf{y}_k^{m-1}=\big[(\B-\frac{1}{\rho_{k-1}}\A)+\frac{1-\rho_{k-1}}{\rho_{k-1}}\A\big]\mathbf{y}_k^{m-1}>0$.

According to \Cref{thm: Collatz-Wielandt}, we can update $s_k$ in a similar way to the step (2) of the original NI, that is, $s_k=\max\frac{\A\mathbf{x}_k^{m-1}}{(\B-\A)\mathbf{x}_k^{m-1}}=\max\frac{\A\mathbf{z}_k^{m-1}}{(\B-\A)\mathbf{z}_k^{m-1}}=s_{k-1}-\tau_{k-1}$, where $\tau_{k-1}=\min\frac{(\B-\A)\mathbf{x}_{k-1}^{m-1}}{(\B-\A)\mathbf{z}_k^{m-1}}$. Since $\rho_k=s_k/(1+s_k)$, we can update $\rho_k$ adaptively by
\begin{equation}\label{equ:update rho 1}
\rho_k=\frac{s_{k-1}-\tau_{k-1}}{1+(s_{k-1}-\tau_{k-1})}.
\end{equation}
\Cref{equ:update rho 1} can be rewritten to
\begin{displaymath}\label{equ:update rho 2}
\rho_k=\frac{s_{k-1}/(1+s_{k-1})-\tau_{k-1}/(1+s_{k-1})}{1-\tau_{k-1}/(1+s_{k-1})}=\frac{\rho_{k-1}-(1-\rho_{k-1})\tau_{k-1}}{1-(1-\rho_{k-1})\tau_{k-1}}=\rho_{k-1}-\frac{(1-\rho_{k-1})^2\tau_{k-1}}{1-(1-\rho_{k-1})\tau_{k-1}}.
\end{displaymath}
As $\mathbf{y}_k=(1+s_{k-1})^{1/(m-1)}\mathbf{z}_k$, if we denote
\begin{displaymath}\label{equ:tilde tau}
\tilde{\tau}_{k-1}=\min\frac{(\B-\A)\mathbf{x}_{k-1}^{m-1}}{(\B-\A)\mathbf{y}_k^{m-1}}=(1-\rho_{k-1})\min\frac{(\B-\A)\mathbf{x}_{k-1}^{m-1}}{(\B-\A)\mathbf{z}_k^{m-1}}=(1-\rho_{k-1})\tau_{k-1},
\end{displaymath}
then
\begin{equation}\label{equ:update rho 3}
\rho_k=\rho_{k-1}-\frac{(1-\rho_{k-1})\tilde{\tau}_{k-1}}{1-\tilde{\tau}_{k-1}}.
\end{equation}

Since $(\B-\A)\mathbf{x}_{k-1}^{m-1}>0$ and $(\mathcal{B-A})\mathbf{z}_k^{m-1}>0$, we obtain that $s_k<s_{k-1}$. On the other hand, $s_{k-1}=\frac{s_{k-1}\mathcal{(B-A)}\mathbf{z}_k^{m-1}}{\mathcal{(B-A)}\mathbf{z}_k^{m-1}}$ implies that $s_k=\max\frac{\mathcal{A}\mathbf{z}_{k}^{m-1}}{\mathcal{(B-A)}\mathbf{z}_k^{m-1}}\geq s=\lambda/(1-\lambda)>0$, where $\lambda$ is the unique positive eigenvalue of the tensor pair $(\A,\B)$. We summarize the above in \Cref{alg:MTNI}.

\begin{algorithm}[htb]
\caption{Modified Tensor Noda Iteration (MTNI)}\label{alg:MTNI}
\begin{algorithmic}[1]
\State Given $\mathbf{b}>0$, solve $\mathcal{(B-A)}\mathbf{x}_0^{m-1}=\mathbf{b}$
\State Given $\mathbf{x}_0=\frac{\mathbf{x}_0}{\|\mathbf{x}_0\|}$,  $\rho_0=\max\frac{\A\mathbf{x}_0^{m-1}}{\B\mathbf{x}_0^{m-1}}$, tol$>0$, $\varepsilon>0$
\For {$k = 1,2,3,\ldots$}
\State Solve $(\rho_{k-1}\B-\A)\mathbf{y}_k^{m-1}=(\B-\A)\mathbf{x}_{k-1}^{m-1}$
\State $\tau_{k-1}=\min\frac{(\B-\A)\mathbf{x}_{k-1}^{m-1}}{(\B-\A)\mathbf{y}_k^{m-1}}$
\State $\rho_k=(1+\varepsilon)(\rho_{k-1}-\frac{(1-\rho_{k-1})\tau_{k-1}}{1-\tau_{k-1}})$
\State $\mathbf{x}_k=\frac{\mathbf{y}_k}{\|\mathbf{y}_k\|}$
\State $\bar{s}_k=\max\frac{\A\mathbf{x}_k^{m-1}}{(\B-\A)\mathbf{x}_k^{m-1}}$
\State $\bar{\rho}_k=\bar{s}_k/(1+\bar{s}_k)$ \State$\underline{s}_k=\min\frac{\A\mathbf{x}_k^{m-1}}{(\B-\A)\mathbf{x}_k^{m-1}}$
\State $\underline{\rho}_k=\underline{s}_k/(1+\underline{s}_k)$
\If{$|\overline{\rho}_k-\underline{\rho}_k|/\overline{\rho}_k<$tol}
\State break
\EndIf
\EndFor
\State Output: $\lambda\leftarrow \bar{\rho}_k$, $\mathbf{x}_{\ast}\leftarrow\mathbf{x}_{k}$
\end{algorithmic}
\end{algorithm}

In \Cref{alg:MTNI}, the initialization part is designed to guarantee $(\mathcal{B-A})\mathbf{x}_0^{m-1}>0$. If we only initialize $\mathbf{x}_0>0$ without step 1, then $(\mathcal{B-A})\mathbf{x}_0^{m-1}$ may not be positive. The counter example can be found in \cite[page 93]{DW}. In step 4 of \Cref{alg:MTNI}, we use the Jacobi iteration to solve the $\mathcal{M}$-tensor equation, which can guarantee that the solution $\mathbf{y}_k$ is the unique positive solution for the $\mathcal{M}$-tensor equation. Note that the solution of $\mathcal{D}\mathbf{x}^{m-1}=\mathbf{b}$ is $x_i=(b_i/d_{i\ldots i})^{1/(m-1)}$ , where $\mathcal{D}$ is a positive diagonal tensor and $\mathbf{b}$ is a positive vector. Alternatively, we can also use the Gauss-Seidel iteration or SOR iteration. The details of these iteration methods for $\mathcal{M}$-tensor equations can be found in \cite[page 108-110]{DW}.

In step 6 of \Cref{alg:MTNI}, we introduce a parameter $\varepsilon>0$ in the $\rho$--update. This parameter avoids $\rho_k$ approaching the eigenvalue $\lambda$. Otherwise, the coefficient tensor $(\rho_k\B-\A)$ will become closer to a singular $\mathcal{M}$-tensor gradually, thus solving the equation in step 4 will become more and more difficult. Inserting the parameter $\varepsilon$ can reduce the iteration steps for solving the equation in step 4.

\subsection{Exact and Inexact Generalized Tensor Noda Iteration (GTNI, IGTNI)}

By \Cref{thm: Collatz-Wielandt} and \Cref{cor: Collatz-Wielandt}, we can develop a generalized Noda iteration for the generalized $\mathcal{M}$-tensor pair $(\A,\B)$ similar to GNI, which is displayed in \Cref{alg:GTNI}.

\begin{algorithm}[htbp]
\caption{Generalized Tensor Noda Iteration (GTNI)}\label{alg:GTNI}
\begin{algorithmic}[1]
\State Given $\mathbf{x}_0>0$, $\|\mathbf{x}_0\|=1$, $\rho_0=1$, tol$>0$, and $\varepsilon>0$
\For {$k = 1,2,3,\ldots$}
\State Compute $(\rho_{k-1}\B-\A)\mathbf{y}_k^{m-1}=\A\mathbf{x}_{k-1}^{m-1}$
\State Compute $\rho_k=(1+\varepsilon_k)\rho_{k-1}\left(1-\min\frac{\A\mathbf{x}_{k-1}^{m-1}}{\A\mathbf{y}_k^{m-1}+\A\mathbf{x}_{k-1}^{m-1}}\right)$
\State Normalize the vector $\mathbf{y}_k$: $\mathbf{x}_k=\mathbf{y}_k/\|\mathbf{y}_k\|$
\State Compute $\overline{\rho}_k=\max\limits_i\frac{\mathcal{A}\mathbf{x}_k^{m-1}}{\mathcal{B}\mathbf{x}_k^{m-1}}$
\State Compute $\underline{\rho}_k=\min\limits_i\frac{\mathcal{A}\mathbf{x}_k^{m-1}}{\mathcal{B}\mathbf{x}_k^{m-1}}$
\If {$|\overline{\rho}_k-\underline{\rho}_k|/\overline{\rho}_k<$tol}
\State break
\EndIf
\EndFor
\State Output: $\lambda\leftarrow\rho_k$ and $\mathbf{x}_{\ast}\leftarrow\mathbf{x}_k.$
\end{algorithmic}
\end{algorithm}

Inspired by \cite{GCV}, we can compute an approximate solution $\mathbf{y}_k$ in step 3 of \Cref{alg:GTNI} such that
\begin{displaymath}
(\rho_{k-1}\B-\A)\mathbf{y}_k^{m-1}=\A\mathbf{x}_{k-1}^{m-1}+\mathbf{f}_k,
\end{displaymath}
where $\mathbf{f}_k$ is the residual vector, which is bounded by the inner tolerance. Hence
\begin{displaymath}
\begin{aligned}
\rho_k&=(1+\varepsilon_k)\max\frac{\A\mathbf{x}_k^{m-1}}{\B\mathbf{x}_k^{m-1}}=(1+\varepsilon_k)\rho_{k-1}\max\frac{\A\mathbf{y}_k^{m-1}}{\B\mathbf{y}_k^{m-1}}\\
&=(1+\varepsilon_k)\rho_{k-1}\left(1-\min\frac{\A\mathbf{x}_{k-1}+\mathbf{f}_k}{\A\mathbf{y}_k^{m-1}+\A\mathbf{x}_{k-1}^{m-1}+\mathbf{f}_k}\right).\\
\end{aligned}
\end{displaymath}
Therefore, we get an inexact generalized Noda iteration for tensor. We summarize these in \Cref{alg:IGTNI}.

\begin{remark}\label{remark:choice of varepsilon IGTNI}
In step 4 of \Cref{alg:GTNI,alg:IGTNI}, the choice of $\varepsilon_k$ is determined by the following halving procedure: First, let $\varepsilon_k=1$ and check whether
\begin{equation}\label{equ:choice of varepsilon IGTNI}
(1+\varepsilon_k)\left(1-\min\frac{\A\mathbf{x}_{k-1}^{m-1}+\mathbf{f}_k}{\A\mathbf{y}_k^{m-1}+\A\mathbf{x}_{k-1}^{m-1}+\mathbf{f}_k}\right)\leq1.
\end{equation}
Otherwise, we update $\varepsilon_k$ by $\varepsilon_k\leftarrow\varepsilon_k/2$ until \eqref{equ:choice of varepsilon IGTNI} holds.
\end{remark}

\begin{algorithm}[htbp]
\caption{Inexact Generalized Tensor Noda Iteration (IGTNI)}\label{alg:IGTNI}
\begin{algorithmic}[1]
\State Given $\mathbf{x}_0>0$, $\|\mathbf{x}_0\|=1$, $\rho_0=1$, tol$>0$, and $\varepsilon>0$
\For {$k = 1,2,3,\ldots$}
\State Compute $(\rho_{k-1}\B-\A)\mathbf{y}_k^{m-1}=\A\mathbf{x}_{k-1}^{m-1}+\mathbf{f}_k$, where $|\mathbf{f}_k|\leq\beta_k\A\mathbf{x}_{k-1}^{m-1},\beta_k\in[0,1),$
\State Compute $\rho_k=(1+\varepsilon_k)\rho_{k-1}\left(1-\min\frac{\A\mathbf{x}_{k-1}^{m-1}+\mathbf{f}_k}{\A\mathbf{y}_k^{m-1}+\A\mathbf{x}_{k-1}^{m-1}+\mathbf{f}_k}\right)$
\State Normalize the vector $\mathbf{y}_k$: $\mathbf{x}_k=\mathbf{y}_k/\|\mathbf{y}_k\|$
\State Compute $\overline{\rho}_k=\max\limits_i\frac{\mathcal{A}\mathbf{x}_k^{m-1}}{\mathcal{B}\mathbf{x}_k^{m-1}}$
\State Compute $\underline{\rho}_k=\min\limits_i\frac{\mathcal{A}\mathbf{x}_k^{m-1}}{\mathcal{B}\mathbf{x}_k^{m-1}}$
\If {$|\overline{\rho}_k-\underline{\rho}_k|/\overline{\rho}_k<$tol}
\State break
\EndIf
\EndFor
\State Output: $\lambda\leftarrow\rho_k$ and $\mathbf{x}_{\ast}\leftarrow\mathbf{x}_k.$
\end{algorithmic}
\end{algorithm}

\subsection{Generalized Newton-Noda iteration (GNNI)}

If we choose Newton method to solve the multi-linear equation in step 4 of \Cref{alg:MTNI}, we can simply set $\varepsilon_k=0$. Inspired by \cite{GLL,LGL,LGL16}, we construct a Newton-Noda iterative method for the generalized $\mathcal{M}$-tensor pair $(\A,\B)$ as follows.

To apply Newton method, we define functions $\mathbf{r}:\mathbb{R}_+^{n+1}\rightarrow\mathbb{R}^n$ and $\mathbf{f}:\mathbb{R}_+^{n+1}\rightarrow\mathbb{R}^{n+1}$ by
\begin{equation}\label{equ:define r and f for Newton method}
\mathbf{r}(\mathbf{x},\rho)=(\rho\B-\A)\mathbf{x}^{m-1},~\mathbf{f}(\mathbf{x},\rho)=\left[\begin{aligned}&-\mathbf{r}(\mathbf{x},\rho)\\&\frac{1}{2}(1-\mathbf{x}^\top\mathbf{x})\end{aligned}\right],
\end{equation}
We consider using Newton's method to solve $\mathbf{f}(\mathbf{x},\rho)=0$. The Jacobian of $\mathbf{f}(\mathbf{x},\rho)$ is given by
\begin{equation}\label{equ:Jacobi of f}
\mathbf{Jf}(\mathbf{x},\rho)=-
\begin{bmatrix}
&\mathbf{J_xr}(\mathbf{x},\rho)&\B\mathbf{x}^{m-1}\\
&\mathbf{x}^\top&0\\
\end{bmatrix},
\end{equation}
where $\mathbf{J_xr}(\mathbf{x},\rho)$ is the matrix of partial derivatives of $\mathbf{r}(\mathbf{x},\rho)$ with respect to $\mathbf{x}$. Direct computation gives
\begin{equation}\label{equ:definition of Jxr}
\mathbf{J_xr}(\mathbf{x},\rho)=(m-1)(\rho\B-\A)\mathbf{x}^{m-2},
\end{equation}

We derive the Newton's method for solving $\mathbf{f}(\mathbf{x},\rho)=0$ in the followings. Given an approximation $(\hat{\mathbf{x}}_k,\hat{\rho}_k)$, Newton's method produces the next approximation $(\hat{\mathbf{x}}_{k+1},\hat{\rho}_{k+1})$ as follows:
\begin{subequations}
\begin{align}
-\begin{bmatrix}
&\mathbf{J_xr}(\hat{\mathbf{x}}_k,\hat{\rho}_k)&\B\hat{\mathbf{x}}_k^{m-1}\\
&\hat{\mathbf{x}}_k^\top&0\\
\end{bmatrix}\left[
\begin{aligned}
&\mathbf{d}_k\\
&\delta_k\\
\end{aligned}\right]&=\left[
\begin{aligned}
&\quad\mathbf{r}(\hat{\mathbf{x}}_k,\hat{\rho}_k)\\
&\frac{1}{2}(\hat{\mathbf{x}}_k^\top\hat{\mathbf{x}}_k-1)\\
\end{aligned}\right],\label{equ:Newton method for solving f 1}\\
\hat{\mathbf{x}}_{k+1}&=\hat{\mathbf{x}}_k+\mathbf{d}_k,\label{equ:Newton method for solving f 2}\\
\hat{\rho}_{k+1}&=\hat{\rho}_k+\delta_k.\label{equ:Newton method for solving f 3}
\end{align}
\end{subequations}
Using elimination in \eqref{equ:Newton method for solving f 1} and assuming $\hat{\mathbf{x}}_k^\top\hat{\mathbf{x}}_k=1$, we get
\begin{equation}\label{equ:Newton method delta k}
\delta_k=\frac{-1}{(m-1)\hat{\mathbf{x}}_k^\top\big(\mathbf{J_xr}(\hat{\mathbf{x}}_k,\hat{\rho}_k)\big)^{-1}\B\hat{\mathbf{x}}_k^{m-1}}=\frac{-1}{(m-1)\hat{\mathbf{x}}_k^\top\hat{\mathbf{w}}_k},
\end{equation}
where we denote
\begin{equation}\label{equ:Newton method w k}
\hat{\mathbf{w}}_k=(\mathbf{J_xr}\big(\hat{\mathbf{x}}_k,\hat{\rho}_k)\big)^{-1}\B\hat{\mathbf{x}}_k^{m-1}.
\end{equation}
Then a back substitution, together with \eqref{equ:definition of Jxr}, gives
\begin{equation}\label{equ:Newton method d k}
\mathbf{d}_k=\frac{-1}{m-1}\left(\hat{\mathbf{x}}_k-\frac{\hat{\mathbf{w}}_k}{\hat{\mathbf{x}}_k^\top\hat{\mathbf{w}}_k}\right).
\end{equation}
Let $\hat{\mathbf{y}}_k=\hat{\mathbf{w}}_k/\|\hat{\mathbf{w}}_k\|$. From \eqref{equ:Newton method delta k} and \eqref{equ:Newton method d k}, we have
\begin{subequations}
\begin{align}
\hat{\mathbf{x}}_{k+1}&=\hat{\mathbf{x}}_k+\mathbf{d}_k=\frac{1}{m-1}\left((m-2)\hat{\mathbf{x}}_k+\frac{1}{\hat{\mathbf{x}}_k^\top\hat{\mathbf{y}}_k}\hat{\mathbf{y}}_k\right),\label{equ:Newton method update x k+1}\\
\hat{\rho}_{k+1}&=\hat{\rho}_k+\delta_k=\hat{\rho}_k-\frac{1}{(m-1)\hat{\mathbf{x}}_k^\top\hat{\mathbf{y}}_k\|\hat{\mathbf{w}}_k\|}.\label{equ:Newton method update rho k+1}
\end{align}
\end{subequations}

Based on these discussions, we propose the generalized Newton-Noda iteration (GNNI) as \Cref{alg:GNNI}.

\begin{remark}\label{remark:choice of theta k 1}
In step 7 of \Cref{alg:GNNI}, we need to choose $\theta_k$ properly such that $(\B-\A)\mathbf{x}_{k+1}^{m-1}>0$ to make sure that we can use \Cref{thm: Collatz-Wielandt} and \Cref{cor: Collatz-Wielandt} on $\overline{\rho}_{k+1},\underline{\rho}_{k+1}$ and $\mathbf{w}_{k+1}>0$. We can achieve this by the following operation.

First, let $\theta_k=1$ and check whether $(\B-\A)\tilde{\mathbf{x}}_{k+1}^{m-1}>0$. If not, we update $\theta_k$ by $\theta_k\leftarrow\theta_k/2$ until $(\B-\A)\tilde{\mathbf{x}}_{k+1}^{m-1}>0$ holds. This can always achieve since $(\B-\A)\mathbf{x}_k^{m-1}>0$ and the multi-linear operator $(\B-\A)$ is continuous.

In order to guarantee that sequence $\{\overline{\rho}_k\}$ is monotonically decreasing, we need additional restrictions on the choice of $\theta_k$. We will discuss this in detail in Section 4.
\end{remark}

Based on \Cref{remark:choice of theta k 1}, we can show that $(\rho_k\B-\A)\mathbf{x}_k^{m-2}$ is a nonsingular M-matrix.

\begin{theorem}\label{thm:(mu B-A)x is also M-matrix}
Let $(\A,\B)$ be a generalized $\mathcal{M}$-tensor pair. Let $\{\mathbf{x}_k\}$ and $\{\rho_k\}$ be generated by \Cref{alg:GNNI}. Then $(\rho_k\B-\A)\mathbf{x}_k^{m-2}$ is a nonsingular M-matrix and $\mathbf{x}_k>0$ for all $k$.
\end{theorem}

\begin{proof}
We prove this by induction on $k$. Denote $M_k=(\rho_k\B-\A)\mathbf{x}_k^{m-2}$. First, we consider $k=0$. It's easy to see that $M_0$ is a Z-matrix by condition $(\mathrm{C}4')$ and $M_0\mathbf{x}_0=(\rho_0\B-\A)\mathbf{x}_0^{m-1}\geq0$. Define $J\subseteq[n]$ as $J=\{i|(M_0\mathbf{x}_0)_i=0\}$ and $\mathbf{e}_{J}\in\mathbb{R}^n$ as the vector such that the $i$-th element of $\mathbf{e}_J$ equals to 1 for all $i\in J$ and the rest elements are 0. Then we can find a proper $\varepsilon>0$ such that, $M_0(\mathbf{x}_0+\varepsilon\mathbf{e}_J)>0$. Thus $M_0$ is a nonsingular M-matrix. Therefore, we can get a unique $\mathbf{w}_0>0$ and hence $\tilde{\mathbf{x}}_1=(m-2)\mathbf{x}_0+\theta_k\mathbf{w}_k/\|\mathbf{w}_k\|>0.$

Suppose that the conclusion is true for $k-1$, that is, $\rho_{k-1}<1$, $\mathbf{x}_{k-1}>0$, and $(\rho_{k-1}\B-\A)\mathbf{x}_{k-1}^{m-2}$ is a nonsingular M-matrix. By analogous analysis as above, we can see that $\mathbf{x}_k>0$. According to \Cref{remark:choice of theta k 1}, we have $(\B-\A)\mathbf{x}_k^{m-1}>0$ and hence $\rho_k=\max\frac{\A\mathbf{x}_k^{m-1}}{\B\mathbf{x}_k^{m-1}}<1$. Then we can prove that $(\rho_k\B-\A)\mathbf{x}_k^{m-2}$ is a nonsingular M-matrix and $\mathbf{x}_{k+1}>0$ by an analogous analysis as above. By the induction hypothesis, the conclusion is true for all $k$.
\end{proof}

\begin{algorithm}[htbp]
\caption{Generalized Newton-Noda Iteration (GNNI)}\label{alg:GNNI}
\begin{algorithmic}[1]
\State Given $\mathbf{b}>0$, solve $\mathcal{(B-A)}\mathbf{x}_0^{m-1}=\mathbf{b}$
\State Given $\mathbf{x}_0=\frac{\mathbf{x}_0}{\|\mathbf{x}_0\|}$,  $\rho_0=\max\limits_i\frac{\A\mathbf{x}_0^{m-1}}{\B\mathbf{x}_0^{m-1}}$, and tol$>0$
\For {$k = 0,1,2,\ldots$}
\State Compute $\mathbf{J}_\mathbf{x}\mathbf{r}(\mathbf{x}_k,\rho_k)=(m-1)(\rho_k\mathcal{B}-\mathcal{A})\mathbf{x}_k^{m-2}$
\State Solve the linear system $\mathbf{J}_\mathbf{x}\mathbf{r}(\mathbf{x}_k,\rho_k)\mathbf{w}_k=\mathcal{B}\mathbf{x}_k^{m-1}$
\State Normalize the vector $\mathbf{w}_k$: $\mathbf{y}_k=\mathbf{w}_k/\|\mathbf{w}_k\|$
\State Compute the vector $\tilde{\mathbf{x}}_{k+1}=(m-2)\mathbf{x}_k+\theta_k\mathbf{y}_k$
\State Normalize the vector $\tilde{\mathbf{x}}_{k+1}$: $\mathbf{x}_{k+1}=\tilde{\mathbf{x}}_{k+1}/\|\tilde{\mathbf{x}}_{k+1}\|$
\State Compute $\overline{\rho}_{k+1}=\max\limits_i\frac{\mathcal{A}\mathbf{x}_{k+1}^{m-1}}{\mathcal{B}\mathbf{x}_{k+1}^{m-1}}$
\State Compute $\underline{\rho}_{k+1}=\min\limits_i\frac{\mathcal{A}\mathbf{x}_{k+1}^{m-1}}{\mathcal{B}\mathbf{x}_{k+1}^{m-1}}$
\State Let $\rho_{k+1}=\overline{\rho}_{k+1}$
\If {$|\overline{\rho}_{k+1}-\underline{\rho}_{k+1}|/\overline{\rho}_{k+1}<$tol}
\State break
\EndIf
\EndFor
\State Output: $\lambda\leftarrow\rho_{k+1}$ and $\mathbf{x}_{\ast}\leftarrow\mathbf{x}_{k+1}.$
\end{algorithmic}
\end{algorithm}

\subsection{Complexity analysis}
We analyze the cost per iteration of our algorithms in this part. Assuming tensors $\A$ and $\B$ are dense, it is easy to see that the main cost of each iteration for \Cref{alg:MTNI,alg:GTNI,alg:IGTNI} is solving the $\mathcal{M}$-tensor equations. As we mentioned in Section 3.1, we use Jacobi iteration method to solve these $\mathcal{M}$-tensor equations. Denote $\mathcal{M} = \rho_{k-1}\B-\A$ and $\mathcal{D}$ is a diagonal tensor whose diagonal element $d_{i\ldots i}$ equals to $\rho_{k-1}b_{i\ldots i}-a_{i\ldots i}$. For each inner iteration, we need 3 steps of calculations: First, we calculate $(\mathcal{D-M})\mathbf{y}_k^{m-1}$. Then we compute $\mathbf{b}_k=(\mathcal{D-M})\mathbf{y}_k^{m-1}+(\B-\A)\mathbf{x}_{k-1}^{m-1}$ (or $(\mathcal{D-M})\mathbf{y}_k^{m-1}+\A\mathbf{x}_{k-1}^{m-1}$, or $(\mathcal{D-M})\mathbf{y}_k^{m-1}+(\B-\A)\mathbf{x}_{k-1}^{m-1}+\mathbf{f}_k$). Finally, we do the operation $((\mathbf{b}_k)_i/d_{i\ldots i})^{1/(m-1)}$ for $i=1,\ldots,n$. Thus, the cost for one inner iteration is $(m-1)n^m+n+2n$ operations. Therefore, the cost for \Cref{alg:MTNI,alg:GTNI,alg:IGTNI} is $O(n^m)$. For \Cref{alg:GNNI}, the main cost for every iteration consists of four products: $\A\mathbf{x}_{k+1}^{m-2}$, $\B\mathbf{x}_{k+1}^{m-2}$, $\A\mathbf{x}_{k+1}^{m-1}$, and $\B\mathbf{x}_{k+1}^{m-1}$. We can compute $\A\mathbf{x}_{k+1}^{m-1}$ by $\A\mathbf{x}_{k+1}^{m-1}=(\A\mathbf{x}_{k+1}^{m-2})\mathbf{x}$ for cost of $2n^2-n$ operations. Therefore, the main cost of each iteration for \Cref{alg:GNNI} is $2(m-1)n^m+4n^2-2n=O(n^m)$. If $\A,\B$ are symmetric, then the cost per iteration of \Cref{alg:MTNI,alg:GTNI,alg:IGTNI,alg:GNNI} reduce to $O(n^m/m!)$ \cite{Schatz}.

\section{Convergence analysis}

In this section, we will prove the convergence of \Cref{alg:MTNI,alg:IGTNI,alg:GNNI} and their convergence rates.

\subsection{The convergence of MTNI}

We analyze \Cref{alg:MTNI} in this part. First, we prove some properties of the sequence $\{(\mathbf{x}_k,\mathbf{y}_k,\rho_k)\}$ in the following lemma.

\begin{lemma}\label{lemma:rho_k bounded below MTNI}
Let $(\mathcal{A},\mathcal{B})$ be a generalized $\mathcal{M}$-tensor pair. Let the sequences $\{(\mathbf{x}_k,\mathbf{y}_k,\rho_k)\}$ be generated by \Cref{alg:MTNI}. Then $\mathbf{x}_k>0$, $\mathcal{A}\mathbf{y}_k^{m-1}>0$, and $(\mathcal{B}-\mathcal{A})\mathbf{y}_k^{m-1}>0$ for all $k\geq0$, and the sequence $\{\rho_k\}$ is bounded below by $(1+\varepsilon)\lambda$, i.e., $\rho_k\geq(1+\varepsilon)\lambda$, for all $k$.
\end{lemma}

\begin{proof}
Since $(\B-\A)\mathbf{x}_0^{m-1}>0$, we have $(\rho_0\B-\A)\mathbf{y}_1^{m-1}=(\B-\A)\mathbf{x}_0^{m-1}>0$. By \Cref{thm:mu B-A is also M-tensor}, $(\rho_0\mathcal{B}-\mathcal{A})$ is a nonsingular $\mathcal{M}$-tensor. Hence $\mathbf{y}_1>0$ by \Cref{lemma:M-equations}. Then from \eqref{equ:iteration format 2}, we have
\begin{displaymath}
s_0(\B-\A)\mathbf{y}_1^{m-1}=\A\mathbf{y}_1^{m-1}+(\B-\A)\mathbf{x}_0^{m-1}>0.
\end{displaymath}
Since $s_0=\rho_0/(1-\rho_0)>0$, we have $(\B-\A)\mathbf{y}_1^{m-1}>0$. Thus $\tau_0=\min\frac{(\B-\A)\mathbf{x}_0^{m-1}}{(\B-\A)\mathbf{y}_1^{m-1}}>0$ and $s_1=s_0-\tau_0<s_0$. In addition,
\begin{displaymath}
s_1=s_0-\tau_0=\max\frac{\A\mathbf{y}_1^{m-1}}{(\B-\A)\mathbf{y}_1^{m-1}}\geq s=\lambda/(1-\lambda)>0.
\end{displaymath}
Since the function $f(t)=t/(1+t)$ is monotonically increasing, we get that
\begin{displaymath}
\rho_1=(1+\varepsilon)\frac{s_0-\tau_0}{1+(s_0-\tau_0)}\geq(1+\varepsilon)\frac{s}{1+s}=(1+\varepsilon)\lambda.
\end{displaymath}
Now suppose $\mathbf{y}_k>0$ and $(\B-\A)\mathbf{y}_k^{m-1}>0$. Then $\mathbf{x}_k=\frac{\mathbf{y}_k}{\|\mathbf{y}_k\|}>0$ and
\begin{displaymath}
(\rho_k\B-\A)\mathbf{y}_{k+1}^{m-1}=(\B-\A)\mathbf{x}_k^{m-1}>0.
\end{displaymath}
Similarly, we have $(\B-\A)\mathbf{y}_{k+1}^{m-1}>0$. Thus $\tau_k=\min\frac{(\B-\A)\mathbf{x}_k^{m-1}}{(\B-\A)\mathbf{y}_{k+1}^{m-1}}>0$ and $s_{k+1}=s_k-\tau_k\geq s$. That is,
\begin{displaymath}
\tau_{k+1}=(1+\varepsilon)\frac{s_{k+1}}{1+s_{k+1}}\geq(1+\varepsilon)\frac{s}{1+s}=(1+\varepsilon)\lambda.
\end{displaymath}
By the induction, the statement holds for any $k$.
\end{proof}

We prove that the sequence $\{\overline{s}_k\}$ is monotonically decreasing in the following lemma.

\begin{lemma}\label{lemma:convergence of s_max MTNI}
Let $(\mathcal{A},\mathcal{B})$ be a generalized $\mathcal{M}$-tensor pair. Let the sequences $\{(\mathbf{x}_k,\mathbf{y}_k,\bar{s}_k)\}$ be generated by \Cref{alg:MTNI}. Then the sequence $\{\bar{s}_k\}$ is monotonically decreasing and bounded below by $s$, i.e., $\bar{s}_1\geq\cdots\geq\bar{s}_k\geq\bar{s}_{k+1}\geq\cdots\geq s$.
\end{lemma}

\begin{proof}
According to \Cref{thm: Collatz-Wielandt}, we have $\bar{s}_k\geq s$. Based on \eqref{equ:iteration format 2} and \eqref{equ:iteration format 3}, we have
\begin{displaymath}
\begin{aligned}
(\B-\A)\mathbf{x}_{k-1}^{m-1}&=\left[\frac{s_{k-1}}{1+s_{k-1}}(\B-\A)-\frac{1}{1+s_{k-1}}\A\right]\mathbf{y}_k^{m-1}\\
&\geq\left[\frac{s_{k-1}}{1+s_{k-1}}(\B-\A)-\frac{1}{1+s_{k-1}}\bar{s}_k(\B-\A)\right]\mathbf{y}_k^{m-1}\\
&=\frac{s_{k-1}-\bar{s}_k}{1+s_{k-1}}(\B-\A)\mathbf{y}_k^{m-1}.
\end{aligned}
\end{displaymath}
Note that $s_{k-1}-\bar{s}_k=\min\frac{(\B-\A)\mathbf{x}_{k-1}^{m-1}}{(\B-\A)\mathbf{z}_k^{m-1}}=(1+s_{k-1})\min\frac{(\B-\A)\mathbf{x}_{k-1}^{m-1}}{(\B-\A)\mathbf{y}_k^{m-1}}>0$ and hence
$(\B-\A)\mathbf{x}_{k-1}^{m-1}\geq\frac{s_{k-1}-\bar{s}_k}{1+s_{k-1}}(\B-\A)\mathbf{y}_k^{m-1}>0$. According to \Cref{lemma:monotony of M equation}, we can see that $\mathbf{x}_{k-1}\geq\left(\frac{s_{k-1}-\bar{s}_k}{1+s_{k-1}}\right)^{1/(m-1)}\mathbf{y}_k>0$. This implies that $\A\mathbf{x}_{k-1}^{m-1}\geq\frac{s_{k-1}-\bar{s}_k}{1+s_{k-1}}\A\mathbf{y}_k^{m-1}$. Besides, we have
\begin{displaymath}
\begin{aligned}
\A\mathbf{x}_{k-1}^{m-1}\leq\bar{s}_{k-1}(\B-\A)\mathbf{x}_{k-1}^{m-1}&=\bar{s}_{k-1}(\rho_{k-1}\B-\A)\mathbf{y}_k^{m-1}\\
&=\bar{s}_{k-1}\left[\frac{s_{k-1}}{1+s_{k-1}}(\B-\A)-\frac{1}{1+s_{k-1}}\A\right]\mathbf{y}_k^{m-1}.
\end{aligned}
\end{displaymath}
Rearranging the above inequalities, we have $\frac{\bar{s}_{k-1}s_{k-1}}{s_{k-1}-\bar{s}_k+\bar{s}_{k-1}}(\B-\A)\mathbf{y}_k^{m-1}\geq\A\mathbf{y}_k^{m-1}$. This implies that $\frac{\bar{s}_{k-1}s_{k-1}}{s_{k-1}-\bar{s}_k+\bar{s}_{k-1}}\geq\bar{s}_k$. Since $\bar{s}_k=s_{k-1}-(1+s_{k-1})\min\frac{(\B-\A)\mathbf{x}_{k-1}^{m-1}}{(\B-\A)\mathbf{y}_k^{m-1}}\leq s_{k-1}$, we finally get that $\bar{s}_k\leq\bar{s}_{k-1}$. Thus $\bar{s}_1\geq\cdots\geq\bar{s}_k\geq\bar{s}_{k+1}\geq\cdots\geq s$.
\end{proof}

\begin{remark}\label{remark:convergence of s_max MTNI}
(1)~According to \Cref{lemma:convergence of s_max MTNI}, we can see that $\{\overline{\rho}_k\}$ is also monotonically decreasing and bounded below by $\lambda$ since the function $f(t)=\frac{t}{1+t}$ is monotonically increasing.

(2)~Similarly, we can prove that the sequence $\{\underline{s}_k\}$ is monotonically increasing and bounded above by $s$, and the sequence $\{\underline{\rho}_k\}$ is also monotonically increasing and bounded above by $\lambda$.
\end{remark}

By \Cref{lemma:rho_k bounded below MTNI}, we can see that there exists a subsequence $\{\mathbf{x}_{k_j}\}$ converging to a nonnegative vector $\mathbf{v}$. The following lemma shows that $\mathbf{v}$ is actually positive.

\begin{lemma}\label{lemma:v>0 MTNI}
Let $(\mathcal{A},\mathcal{B})$ be a generalized $\mathcal{M}$-tensor pair, and let $\{\mathbf{x}_k\}$ be generated by \Cref{alg:MTNI}. Then for any convergent subsequence $\{\mathbf{x}_{k_j}\}\subseteq\{\mathbf{x}_k\}$, $\lim\limits_{j\rightarrow\infty}\mathbf{x}_{k_j}>0$.
\end{lemma}

\begin{proof}
Let $\mathbf{v}=\lim\limits_{j\rightarrow\infty}\mathbf{x}_{k_j}$. From \Cref{lemma:rho_k bounded below MTNI}, $\mathbf{v}\geq0$. Let $S_0$ be the set of all indices $i$ such that $v_i=0$. Since $\|\mathbf{x}_{k_j}\|=1$, $S_0$ is a proper subset of $\{1,\ldots,n\}$. Denote $|S_0|$ as the number of elements in $S_0$, we need to prove $|S_0|=0$.

Since $\B-\A$ is a nonsingular $\mathcal{M}$-tensor, there exist a nonnegative tensor $\mathcal{R}$ and a scalar $\gamma>\rho(\mathcal{R})$ such that $\B-\A=\gamma\mathcal{I-R}$. Then from $\overline{s}_{k_j}=\max\left(\frac{\A\mathbf{x}_{k_j}^{m-1}}{(\B-\A)\mathbf{x}_{k_j}^{m-1}}\right)$ we get $\overline{s}_{k_j}(\gamma\mathcal{I-R})\mathbf{x}_{k_j}^{m-1}\geq\A\mathbf{x}_{k_j}^{m-1}$. Hence $\overline{s}_{k_j}\gamma\mathbf{x}_{k_j}^{[m-1]}\geq\overline{s}_{k_j}\mathcal{R}\mathbf{x}_{k_j}^{m-1}+\A\mathbf{x}_{k_j}^{m-1}\geq\A\mathbf{x}_{k_j}^{m-1}$. According to \Cref{lemma:convergence of s_max MTNI}, sequence $\{\overline{s}_{k_j}\}$ is bounded. Therefore, we have
\begin{displaymath}
\lim\limits_{j\rightarrow\infty}\frac{\sum\limits_{i_2,\ldots,i_m}a_{ii_2\ldots i_m}x_{k_j,i_2}\cdots x_{k_j,i_m}}{x_{k_j,i}^{m-1}}<\infty.
\end{displaymath}
The rest of the proof is similar to \cite[Lemma 4]{LGL}.
\end{proof}

Now we can prove the convergence of the sequence $\{\mathbf{x}_k\}$.

\begin{lemma}\label{lemma:v=x MTNI}
Let $(\mathcal{A},\mathcal{B})$ be a generalized $\mathcal{M}$-tensor pair, and let $\{\mathbf{x}_k\}$ be generated by \Cref{alg:MTNI}. Suppose for a convergent subsequence $\{\mathbf{x}_{k_j}\}\subseteq\{\mathbf{x}_k\}$, $\lim\limits_{j\rightarrow\infty}\mathbf{x}_{k_j}=\mathbf{v}$. Denote $\mathbf{x}_{\ast}>0$, $\|\mathbf{x}_{\ast}\|=1$ to be the positive eigenvector corresponding to the unique positive eigenvalue $\lambda$ of the tensor pair $(\A,\B)$. Then $\mathbf{v}=\mathbf{x}_{\ast}$, thus $\lim\limits_{k\rightarrow\infty}\mathbf{x}_k=\mathbf{x}_{\ast}$.
\end{lemma}

\begin{proof}
By simple derivation from the iteration step $\rho_k=(1+\varepsilon)\big(\rho_{k-1}-\frac{(1-\rho_{k-1})\tau_{k-1}}{1-\tau_{k-1}}\big)$, we can get that $\rho_k=(1+\varepsilon)\big(1-\min\frac{(1-\rho_{k-1})(\B-\A)\mathbf{y}_k^{m-1}}{(\B-\A)\mathbf{y}_k^{m-1}-(\B-\A)\mathbf{x}_{k-1}^{m-1}}\big)=(1+\varepsilon)\max\frac{\A\mathbf{y}_k^{m-1}}{\B\mathbf{y}_k^{m-1}}=(1+\varepsilon)\max\frac{\A\mathbf{x}_k^{m-1}}{\B\mathbf{x}_k^{m-1}}$, and hence the sequence $\{\rho_k\}$ is bounded above. Combining \Cref{lemma:convergence of s_max MTNI}, we know that the sequence $\{\overline{\rho}_k\}$ converges. Thus
\begin{displaymath}
\begin{aligned}
\overline{\rho}_k-\overline{\rho}_{k+1}&=\overline{\rho}_k-\rho_k\max\frac{\A\mathbf{y}_{k+1}^{m-1}}{\rho_k\B\mathbf{y}_{k+1}^{m-1}}=\overline{\rho}_k\left[1-(1+\varepsilon)\max\frac{\A\mathbf{y}_{k+1}^{m-1}}{\rho_k\B\mathbf{y}_{k+1}^{m-1}}\right]\\
&=\overline{\rho}_k\min\frac{\rho_k\B\mathbf{y}_{k+1}^{m-1}-(1+\varepsilon)\A\mathbf{y}_{k+1}^{m-1}}{\rho_k\B\mathbf{y}_{k+1}^{m-1}}=\min\frac{\rho_k\B\mathbf{x}_{k+1}^{m-1}-(1+\varepsilon)\A\mathbf{x}_{k+1}^{m-1}}{(1+\varepsilon)\B\mathbf{x}_{k+1}^{m-1}}\rightarrow0.\\
\end{aligned}
\end{displaymath}
Suppose $\lim\limits_{j\rightarrow\infty}\mathbf{x}_{k_j+1}=\mathbf{v}>0$, since $(\B-\A)\mathbf{x}_{k_j+1}^{m-1}>0$, we have
\begin{displaymath}\lim\limits_{j\rightarrow\infty}\B\mathbf{x}_{k_j+1}^{m-1}=\lim\limits_{j\rightarrow\infty}(\B-\A)\mathbf{x}_{k_j+1}^{m-1}+\A\mathbf{v}^{m-1}>0.
\end{displaymath}
Without loss of generality, we can assume that $\lim\limits_{j\rightarrow\infty}\rho_{k_j}=\rho>0$, hence
\begin{displaymath}
\lim\limits_{j\rightarrow\infty}\min\frac{\rho_{k_j}\B\mathbf{x}_{k_j+1}^{m-1}-(1+\varepsilon)\A\mathbf{x}_{k_j+1}^{m-1}}{(1+\varepsilon)\B\mathbf{x}_{k_j+1}^{m-1}}=\min\frac{\rho\B\mathbf{v}^{m-1}-(1+\varepsilon)\A\mathbf{v}^{m-1}}{(1+\varepsilon)\B\mathbf{v}^{m-1}}=0.
\end{displaymath}
Similarly, by considering $\underline{\rho}_k-\underline{\rho}_{k+1}$, we can prove that $\max\frac{\rho\B\mathbf{v}^{m-1}-(1+\varepsilon)\A\mathbf{v}^{m-1}}{(1+\varepsilon)\B\mathbf{v}^{m-1}}=0$. Therefore, $\rho\B\mathbf{v}^{m-1}=(1+\varepsilon)\A\mathbf{v}^{m-1}$. By \Cref{thm:uniqueness of the eigenvalue}, $\mathbf{v}=\mathbf{x}_{\ast}$ and $\rho=(1+\varepsilon)\lambda$. Then $\lim\limits_{j\rightarrow\infty}\mathbf{x}_{k_j}=\mathbf{x}_{\ast}$ for any convergent subsequence $\{\mathbf{x}_{k_j}\}\subseteq\{\mathbf{x}_k\}$ due to the uniqueness of $\mathbf{x}_{\ast}$. Thus we can say that $\lim\limits_{k\rightarrow\infty}\mathbf{x}_k=\mathbf{x}_{\ast}$.
\end{proof}

We summarize the global convergence of \Cref{alg:MTNI} in the following theorem.

\begin{theorem}\label{theorem:convergence of MTNI}
Let $(\mathcal{A},\mathcal{B})$ be a generalized $\mathcal{M}$-tensor pair, and let $\{\mathbf{x}_k\}$, $\{\mathbf{y}_k\}$, $\{\rho_k\}$, $\{\bar{s}_k\}$, and $\{\bar{\rho}_k\}$ be generated by \Cref{alg:MTNI}. Then $\lambda=\lim\limits_{k\rightarrow\infty}\bar{\rho}_k$ is the unique positive generalized eigenvalue for the tensor pair $(\mathcal{A,B})$ and $\mathbf{x}_{\ast}=\lim\limits_{k\rightarrow\infty}\mathbf{x}_k$ is the corresponding positive eigenvector, i.e., $\mathcal{A}\mathbf{x}_{\ast}^{m-1}=\lambda\mathcal{B}\mathbf{x}_{\ast}^{m-1}.$
\end{theorem}

\begin{proof}
From \Cref{lemma:convergence of s_max MTNI}, we know that the sequence $\{\bar{s}_k\}$ is monotonically decreasing and bounded below by $s$ and thus convergent. Denote $\bar{s}=\lim\limits_{k\rightarrow\infty}\bar{s}_k$, then $\bar{s}\geq s$.
According to \Cref{lemma:v=x MTNI},
\begin{displaymath}
\bar{s}=\lim\limits_{k\rightarrow\infty}\bar{s}_k=\lim\limits_{k\rightarrow\infty}\max\frac{\A\mathbf{x}_k^{m-1}}{(\B-\A)\mathbf{x}_k^{m-1}}=\max\frac{\A\mathbf{x}_{\ast}^{m-1}}{(\B-\A)\mathbf{x}_{\ast}^{m-1}}.
\end{displaymath}
Thus $\bar{s}=s$, which is equivalent to $\lim\limits_{k\rightarrow\infty}\bar{\rho}_k=\lambda$.
\end{proof}

We prove the convergence rate of \Cref{alg:MTNI} in the following theorem.

\begin{theorem}\label{thm:convergence rate MTNI}
Let $(\mathcal{A},\mathcal{B})$ be a generalized $\mathcal{M}$-tensor pair, and let $\{\rho_k\}$, $\{\mathbf{x}_k\}$, $\{\mathbf{y}_k\}$, and $\{\bar{\rho}_k\}$ be generated by \Cref{alg:MTNI}. Then the convergence of the sequence $\{\bar{\rho}_k\}$ is at least linear.
\end{theorem}

\begin{proof}
First, we can easily see that $\overline{\rho}_k=\max\frac{\A\mathbf{x}_k^{m-1}}{\B\mathbf{x}_k^{m-1}}$ for all $k\geq1$. Denote $\zeta_k=\overline{\rho}_k-\lambda$. Then
\begin{displaymath}
\begin{aligned}
\zeta_k-\zeta_{k+1}&=\overline{\rho}_k-\lambda-\overline{\rho}_{k+1}+\lambda=\overline{\rho}_k-\max\frac{\A\mathbf{x}_{k+1}^{m-1}}{\B\mathbf{x}_{k+1}^{m-1}}\\
&=\overline{\rho}_k\left(1-\max\frac{\A\mathbf{y}_{k+1}^{m-1}}{\overline{\rho}_k\B\mathbf{y}_{k+1}^{m-1}}\right)=\min\frac{(\overline{\rho}_k-\rho_k)\B\mathbf{y}_{k+1}^{m-1}+(\B-\A)\mathbf{x}_{k+1}^{m-1}}{\B\mathbf{y}_{k+1}^{m-1}}\\
\end{aligned}
\end{displaymath}
Thus
\begin{displaymath}
\begin{aligned}
\frac{\zeta_{k+1}}{\zeta_k}&=1-\min\frac{(\overline{\rho}_k-\rho_k)\B\mathbf{y}_{k+1}^{m-1}+(\B-\A)\mathbf{x}_{k+1}^{m-1}}{\zeta_k\B\mathbf{y}_{k+1}^{m-1}}\\
&=\max\frac{(\zeta_k-\overline{\rho}_k+\rho_k)\B\mathbf{y}_{k+1}^{m-1}-(\B-\A)\mathbf{x}_k^{m-1}}{\zeta_k\B\mathbf{y}_{k+1}^{m-1}}\\
&=\max\frac{(\A-\lambda\B)\mathbf{y}_{k+1}^{m-1}}{(\overline{\rho}_k-\lambda)\B\mathbf{y}_{k+1}^{m-1}}=\max\frac{(\A-\lambda\B)\mathbf{x}_{k+1}^{m-1}}{(\overline{\rho}_k-\lambda)\B\mathbf{x}_{k+1}^{m-1}}\\
\end{aligned}
\end{displaymath}
By \Cref{lemma:convergence of s_max MTNI} and \Cref{remark:convergence of s_max MTNI}, we have $\overline{\rho}_k\geq\overline{\rho}_{k+1}$. Thus $\overline{\rho}_k\B\mathbf{x}_{k+1}^{m-1}\geq\overline{\rho}_{k+1}\B\mathbf{x}_{k+1}^{m-1}\geq\A\mathbf{x}_{k+1}^{m-1}$. Therefore,
\begin{displaymath}
\lim\limits_{k\rightarrow\infty}\frac{\zeta_{k+1}}{\zeta_k}=\lim\limits_{k\rightarrow\infty}\max\frac{(\A-\lambda\B)\mathbf{x}_{k+1}^{m-1}}{(\overline{\rho}_k-\lambda)\B\mathbf{x}_{k+1}^{m-1}}\leq\lim\limits_{k\rightarrow\infty}\max\frac{(\A-\lambda\B)\mathbf{x}_{k+1}^{m-1}}{(\A-\lambda\B)\mathbf{x}_{k+1}^{m-1}}=1
\end{displaymath}
\end{proof}

\subsection{The convergence of IGTNI}

In this section, we show the convergence of \Cref{alg:IGTNI}.

\begin{lemma}\label{lemma:rho_k bounded below IGTNI}
Let $(\mathcal{A},\mathcal{B})$ be a generalized $\mathcal{M}$-tensor pair. Let the sequences $\{\mathbf{x}_k\}, \{\mathbf{y}_k\}$, and$\{\rho_k\}$ be generated by \Cref{alg:IGTNI}. Then $\mathbf{x}_k>0$ for all $k\geq0$, and the sequence $\{\rho_k\}$ is bounded below by $\lambda$, i.e., $\rho_k\geq\lambda$, for all $k$.
\end{lemma}

\begin{proof}
We prove these by induction on $k$. We begin from the equation $(\rho_0\B-\A)\mathbf{y}_1^{m-1}=\A\mathbf{x}_0^{m-1}+\mathbf{f}_1$.
Since $|\mathbf{f}_1|\leq\beta_1\A\mathbf{x}_0^{m-1}$ with $\beta_1\in[0,1)$ and $\mathbf{x}_0>0$, we have
\begin{displaymath}
0<(1-\beta_1)\A\mathbf{x}_0^{m-1}\leq\A\mathbf{x}_0^{m-1}+\mathbf{f}_1\leq(1+\beta_1)\A\mathbf{x}_0^{m-1}.
\end{displaymath}
From \Cref{prop:nonsingular M-tensor}, we know that $\rho_0\B-\A=\B-\A$ is a nonsingular $\mathcal{M}$-tensor, hence there exists a unique positive solution $\mathbf{y}_1$ and then $\mathbf{x}_1=\mathbf{y}_1/\|\mathbf{y}_1\|>0$.

In addition, $\rho_1=(1+\varepsilon_1)\rho_0\left(1-\min\frac{\A\mathbf{x}_0^{m-1}+\mathbf{f}_1}{\A\mathbf{y}_1^{m-1}+\A\mathbf{x}_0^{m-1}+\mathbf{f}_1}\right)=(1+\varepsilon_1)\max\frac{\A\mathbf{y}_1^{m-1}}{\B\mathbf{y}_1^{m-1}}$. By \Cref{cor: Collatz-Wielandt},
\begin{displaymath}
\rho_1=(1+\varepsilon_1)\max\frac{\A\mathbf{y}_1^{m-1}}{\B\mathbf{y}_1^{m-1}}\geq(1+\varepsilon_1)\min\limits_{(\B-\A)\mathbf{x}^{m-1}>0}\max\frac{\A\mathbf{x}^{m-1}}{\B\mathbf{x}^{m-1}}\geq\lambda.
\end{displaymath}
Suppose that $\mathbf{x}_{k-1}>0$. Similarly, we have $(\rho_{k-1}\B-\A)\mathbf{y}_k^{m-1}=\A\mathbf{x}_{k-1}^{m-1}+\mathbf{f}_k>0$ and thus $\mathbf{y}_k>0$, $\mathbf{x}_k=\mathbf{y}_k/\|\mathbf{y}_k\|>0$. Besides, $\rho_k=(1+\varepsilon_k)\rho_{k-1}\left(1-\min\frac{\A\mathbf{x}_{k-1}^{m-1}+\mathbf{f}_k}{\A\mathbf{y}_k^{m-1}+\A\mathbf{x}_{k-1}^{m-1}+\mathbf{f}_k}\right)=(1+\varepsilon_k)\max\frac{\A\mathbf{y}_k^{m-1}}{\B\mathbf{y}_k^{m-1}}\geq\lambda$. By the induction hypothesis, the statement holds for any $k$.
\end{proof}

Based on \Cref{lemma:rho_k bounded below IGTNI}, we can deduce a lemma similar to \Cref{lemma:v>0 MTNI}.

\begin{lemma}\label{lemma:v>0 IGTNI}
Let $(\mathcal{A},\mathcal{B})$ be a generalized $\mathcal{M}$-tensor pair, and let $\{\mathbf{x}_k\}$ be generated by \Cref{alg:IGTNI}. Then for any convergent subsequence $\{\mathbf{x}_{k_j}\}\subseteq\{\mathbf{x}_k\}$, $\lim\limits_{j\rightarrow\infty}\mathbf{x}_{k_j}>0$.
\end{lemma}

If the sequence $\{\varepsilon_k\}$ selected by the halving procedure is not bounded below, then there exists a subsequence $\{\varepsilon_{k_j}\}$ such that $\lim\limits_{j\rightarrow\infty}\varepsilon_{k_j}=0$. Based on \Cref{lemma:v>0 IGTNI} and the halving procedure in \Cref{remark:choice of varepsilon IGTNI}, it immediately follows that $\lim\limits_{j\rightarrow\infty}\|\mathbf{y}_{k_j}\|=+\infty$ and hence $\lim\limits_{j\rightarrow\infty}(\rho_{k_j-1}\B-\A)\mathbf{x}_{k_j}^{m-1}=0$. Assume that $\lim\limits_{j\rightarrow\infty}\rho_{k_j-1}=\rho$ and $\lim\limits_{j\rightarrow\infty}\mathbf{x}_{k_j}=\mathbf{v}$, then $(\rho\B-\A)\mathbf{v}^{m-1}=0$. Therefore, we can easily see that $\{(\mathbf{x}_k,\rho_k)\}$ converges to $(\mathbf{x}_{\ast},\lambda)$ if the sequence $\{\varepsilon_k\}$ selected by the halving procedure is not bounded below. Thus from now on we do the convergence analysis based on the assumption that $\{\varepsilon_k\}$ is bounded below by a constant $\varepsilon>0$.

\begin{lemma}\label{lemma:convergence of rho_max IGTNI}
Let $(\mathcal{A},\mathcal{B})$ be a generalized $\mathcal{M}$-tensor pair. Let the sequences $\{\mathbf{x}_k\}, \{\mathbf{y}_k\}$, $\{\bar{\rho}_k\}$ be generated by \Cref{alg:IGTNI}. Then the sequence $\{\bar{\rho}_k\}$ is monotonically decreasing and bounded below by $\lambda$, i.e., $\bar{\rho}_1\geq\cdots\geq\bar{\rho}_k\geq\bar{\rho}_{k+1}\geq\cdots\geq \lambda$.
\end{lemma}

\begin{proof}
By the halving procedure in \Cref{remark:choice of varepsilon IGTNI}, we can find a proper $\varepsilon_k$ such that \eqref{equ:choice of varepsilon IGTNI} holds. By the assumption, there exists a constant $\varepsilon>0$ such that $\varepsilon_k\geq\varepsilon>0$. Thus the sequence $\{\rho_k\}$ is monotonically decreasing and bounded below by $(1+\varepsilon)\lambda$. Since $\rho_k=(1+\varepsilon)\overline{\rho}_k$, we have $\overline{\rho}_1\geq\cdots\geq\overline{\rho}_k\geq\overline{\rho}_{k+1}\cdots\geq\lambda$.
\end{proof}

Now we can conclude that the sequence $\{\mathbf{x}_k\}$ is convergent. The proof is similar to \Cref{lemma:v=x MTNI}.

\begin{lemma}\label{lemma:v=x IGTNI}
Let $(\mathcal{A},\mathcal{B})$ be a generalized $\mathcal{M}$-tensor pair. Let $\{\mathbf{x}_k\}$ be generated by \Cref{alg:IGTNI}. Suppose that for a convergent subsequence $\{\mathbf{x}_{k_j}\}\subseteq\{\mathbf{x}_k\}$, $\lim\limits_{j\rightarrow\infty}\mathbf{x}_{k_j}=\mathbf{v}$. Denote $\mathbf{x}_{\ast}>0$, $\|\mathbf{x}_{\ast}\|=1$ to be the unique positive eigenvector corresponding to the unique positive eigenvalue $\lambda$ of the tensor pair $(\A,\B)$. Then $\mathbf{v}=\mathbf{x}_{\ast}$, thus $\lim\limits_{k\rightarrow\infty}\mathbf{x}_k=\mathbf{x}_{\ast}$.
\end{lemma}

The global convergence of \Cref{alg:IGTNI} is summarized in the following theorem.

\begin{theorem}\label{theorem:convergence of IGTNI}
Let $(\mathcal{A},\mathcal{B})$ be a generalized $\mathcal{M}$-tensor pair, and let $\{\rho_k\}$, $\{\mathbf{x}_k\}$, $\{\mathbf{y}_k\}$, and $\{\bar{\rho}_k\}$ be generated by \Cref{alg:IGTNI}. Then $\lambda=\lim\limits_{k\rightarrow\infty}\bar{\rho}_k$ is the unique positive generalized eigenvalue for the tensor pair $(\mathcal{A,B})$ and $\mathbf{x}_{\ast}=\lim\limits_{k\rightarrow\infty}\mathbf{x}_k$ is the corresponding positive eigenvector, i.e., $\mathcal{A}\mathbf{x}_{\ast}^{m-1}=\lambda\mathcal{B}\mathbf{x}_{\ast}^{m-1}.$
\end{theorem}

We can also prove the linear convergence of IGTNI similar to \Cref{thm:convergence rate MTNI}.

\begin{theorem}\label{thm:convergence rate IGTNI}
Let $(\mathcal{A},\mathcal{B})$ be a generalized $\mathcal{M}$-tensor pair, and let $\{\rho_k\}$, $\{\mathbf{x}_k\}$, $\{\mathbf{y}_k\}$, and $\{\bar{\rho}_k\}$ be generated by \Cref{alg:IGTNI}. Then the convergence of the sequence $\bar{\rho}_k$ is at least linear.
\end{theorem}

\begin{remark}\label{remark:convergence of GTNI}
Obviously, the conclusions in this section also hold for \Cref{alg:GTNI} if we let $\mathbf{f}_k=0$.
\end{remark}

\subsection{The convergence of GNNI}

Finally, we give the convergence analysis for \Cref{alg:GNNI}. First, we need to discuss how to choose a proper $\theta_k$.

\begin{lemma}\label{lemma:wk(yk-xk) bounded}
Assume that the sequences $\{\rho_k\}$, $\{\mathbf{x}_k\}$, and $\{\mathbf{y}_k\}$ are generated by \Cref{alg:GNNI}, with $\{\rho_k\}$ bounded. Then the sequence $\{\|\mathbf{w}_k\|\|\mathbf{y}_k-\mathbf{x}_k\|\}$ is bounded, that is, there is a constant $\alpha_1>0$ such that $\|\mathbf{w}_k\|\|\mathbf{y}_k-\mathbf{x}_k\|\leq\alpha_1$ for all $k$.
\end{lemma}

\begin{proof}
By the definition of $\mathbf{J_xr}(\mathbf{x}_k,\rho_k)$, we have
\begin{displaymath}
\mathbf{J_xr}(\mathbf{x}_k,\rho_k)\mathbf{y}_k-\mathbf{J_xr}(\mathbf{x}_k,\rho_k)(\mathbf{y}_k-\mathbf{x}_k)=\mathbf{J_xr}(\mathbf{x}_k,\rho_k)\mathbf{x}_k=(m-1)\mathbf{r}(\mathbf{x}_k,\rho_k)\geq0,
\end{displaymath}
thus $\mathbf{J_xr}(\mathbf{x}_k,\rho_k)(\mathbf{y}_k-\mathbf{x}_k)\leq \mathbf{J_xr}(\mathbf{x}_k,\rho_k)\mathbf{y}_k$. Assume $\mathbf{x}_k\neq\mathbf{y}_k$ and then $\|\mathbf{y}_k-\mathbf{x}_k\|\neq0$. Let $\mathbf{p}_k=(\mathbf{y}_k-\mathbf{x}_k)/\|\mathbf{y}_k-\mathbf{x}_k\|$, then $\|\mathbf{p}_k\|=1$ and
\begin{equation}\label{equ:|wk||yk-xk| bounded}
\mathbf{J_xr}(\mathbf{x}_k,\rho_k)\mathbf{p}_k\leq\frac{\B\mathbf{x}_k^{m-1}}{\|\mathbf{w}_k\|\|\mathbf{y}_k-\mathbf{x}_k\|}.
\end{equation}
Suppose that $\{\|\mathbf{w}_k\|\|\mathbf{y}_k-\mathbf{x}_k\|\}$ is not bounded. Since $\{\mathbf{x}_k\}$, $\{\mathbf{y}_k\}$, $\{\rho_k\}$, and $\{\mathbf{p}_k\}$ are bounded, we can find a subsequence $\{k_j\}$ such that $\lim\limits_{j\rightarrow\infty}\|\mathbf{w}_{k_j}\|\|\mathbf{y}_{k_j}-\mathbf{x}_{k_j}\|=\infty$ with $\lim\limits_{j\rightarrow\infty}\|\mathbf{w}_{k_j}\|=\infty$, $\lim\limits_{j\rightarrow\infty}\mathbf{x}_{k_j}=\mathbf{v}$, $\lim\limits_{j\rightarrow\infty}\rho_{k_j}=\rho$, and $\lim\limits_{j\rightarrow\infty}\mathbf{p}_{k_j}=\mathbf{p}$. Similar to \Cref{lemma:v>0 MTNI}, we have $\mathbf{v}>0$. Since $\rho_{k_j}\B\mathbf{x}_{k_j}^{m-1}\geq\A\mathbf{x}_{k_j}^{m-1}$ for all $j$, let $j$ approach infinity, we have $\rho\B\mathbf{v}^{m-1}\geq\A\mathbf{v}^{m-1}$.

Consider the matrix pair $(\A\mathbf{v}^{m-2},\B\mathbf{v}^{m-2})$. We can easily see that $\A\mathbf{v}^{m-2}$ satisfies conditions $\mathrm{(C1)}$ and $\mathrm{(C2)}$. Besides, $(\B\mathbf{v}^{m-2})_{ij}\leq(\A\mathbf{v}^{m-2})_{ij}, \forall i\neq j$, hence it also satisfies condition $\mathrm{(C4)}$. For condition $\mathrm{(C3)}$, note that $\rho_0<1$ and sequence $\{\rho_k\}$ is monotonically decreasing and thus $(\B\mathbf{v}^{m-2})\mathbf{v}>\rho\B\mathbf{v}^{m-1}\geq(\A\mathbf{v}^{m-2})\mathbf{v}$. Thus the matrix pair $(\A\mathbf{v}^{m-2},\B\mathbf{v}^{m-2})$ satisfies conditions $\mathrm{(C1)-(C4)}$.

According to \cite[Theorem 2.3]{BOD95}, this matrix pair has a unique positive generalized eigenvalue, we denote it as $\rho(\A\mathbf{v}^{m-2},\B\mathbf{v}^{m-2})$. Besides, we have $\rho=\max\frac{(\A\mathbf{v}^{m-2})\mathbf{v}}{(\B\mathbf{v}^{m-2})\mathbf{v}}\geq\rho(\A\mathbf{v}^{m-2},\B\mathbf{v}^{m-2})$.

If $\rho>\rho(\A\mathbf{v}^{m-2},\B\mathbf{v}^{m-2})$, then $(\rho\B-\A)\mathbf{v}^{m-2}$ is a nonsingular M-matrix, by step 5 in \Cref{alg:GNNI} we have $\lim\limits_{j\rightarrow\infty}\mathbf{w}_{k_j}=\mathbf{w}>0$, which leads to a contradiction. Thus $\rho=\rho(\A\mathbf{v}^{m-2},\B\mathbf{v}^{m-2})$. By a similar analysis as \cite[Theorem 3.2]{BOD95}, we can see that this equality holds if and only if $\max\frac{(\A\mathbf{v}^{m-2})\mathbf{v}}{(\B\mathbf{v}^{m-2})\mathbf{v}}=\min\frac{(\A\mathbf{v}^{m-2})\mathbf{v}}{(\B\mathbf{v}^{m-2})\mathbf{v}}$and hence $\rho\B\mathbf{v}^{m-1}=\A\mathbf{v}^{m-1}$. By \Cref{thm:uniqueness of the eigenvalue}, $\rho$ is the unique positive generalized eigenvalue for the tensor pair $(\A,\B)$ and $\mathbf{v}$ is the corresponding unique positive eigenvector, that is, $\rho=\lambda$ and $\mathbf{v}=\mathbf{x}_{\ast}$. Then by \eqref{equ:|wk||yk-xk| bounded} we have $\lim\limits_{j\rightarrow\infty}\mathbf{J_xr}(\mathbf{x}_{k_j},\rho_{k_j})\mathbf{p}_{k_j}=\mathbf{J_xr}(\mathbf{x},\lambda)\mathbf{p}\leq0$. On the other hand, $\mathbf{J_xr}(\mathbf{x}_{k_j},\rho_{k_j})\mathbf{p}_{k_j}>0$ for all $j$. Thus $\mathbf{J_xr}(\mathbf{x},\lambda)\mathbf{p}=0$. That is, $\A\mathbf{x}^{m-2}\mathbf{p}=\lambda\B\mathbf{x}^{m-2}\mathbf{p}.$

However, we know that $\A\mathbf{x}^{m-2}\mathbf{x}=\lambda\B\mathbf{x}^{m-2}\mathbf{x}$. Hence by \Cref{thm:uniqueness of the eigenvalue}, we have $\mathbf{p}=\pm\mathbf{x}$. Recall the definition of $\mathbf{p}_k$, since $\|\mathbf{x}_k\|=\|\mathbf{y}_k\|=1$ and $\mathbf{x}_k,\mathbf{y}_k>0$, the elements of $\mathbf{p}_k$ cannot be all nonnegative or all non-positive and then $\mathbf{p}$ is neither positive nor negative. This leads to a contradiction. Therefore, the sequence $\{\|\mathbf{w}_k\|\|\mathbf{y}_k-\mathbf{x}_k\|\}$ is bounded.
\end{proof}

As mentioned in \Cref{remark:choice of theta k 1}, we prove that proper $\{\theta_k\}$ can be found such that the sequence $\{\rho_k\}$ is monotonically decreasing. We begin from the following equation
\begin{equation}\label{equ:rhok - rhok+1}
\overline{\rho}_k-\overline{\rho}_{k+1}=\overline{\rho}_k-\max\frac{\A\mathbf{x}_{k+1}^{m-1}}{\B\mathbf{x}_{k+1}^{m-1}}=\overline{\rho}_k-\max\frac{\A\tilde{\mathbf{x}}_{k+1}^{m-1}}{\B\tilde{\mathbf{x}}_{k+1}^{m-1}}=\min\frac{\mathbf{r}(\tilde{\mathbf{x}}_{k+1},\overline{\rho}_k)}{\B\tilde{\mathbf{x}}_{k+1}^{m-1}}.
\end{equation}
Denote $\mathbf{h}_k(\theta)=\mathbf{r}\big((m-2)\mathbf{x}_k+\theta\mathbf{y}_k,\overline{\rho}_k\big)$, then we have $\mathbf{r}(\tilde{\mathbf{x}}_{k+1},\overline{\rho}_k)=\mathbf{h}_k(\theta_k)$.

\begin{lemma}\label{lemma:choice of theta 1}
For any given constant $\eta>0$, there are scalars $\theta_k\in(0,1]$ such that
\begin{equation}\label{equ:h(theta)}
\mathbf{h}_k(\theta_k)\geq\frac{\theta_k\B\mathbf{x}_k^{m-1}}{(1+\eta)\|\mathbf{w}_k\|}.
\end{equation}
\end{lemma}

\begin{proof}
For $m=2$, \eqref{equ:h(theta)} holds for $\theta_k=1$ since
$$\mathbf{h}_k(1)=(\overline{\rho}_k\B-\A)\mathbf{y}_k=\frac{\B\mathbf{x}_k}{\|\mathbf{w}_k\|}\geq\frac{\B\mathbf{x}_k}{(1+\eta)\|\mathbf{w}_k\|}.$$
For $m\geq3$, let
\begin{equation}\label{equ:g(theta)}
\mathbf{g}_k(\theta)=\mathbf{h}_k(\theta)-\frac{\theta\B\mathbf{x}_k^{m-1}}{(1+\eta)\|\mathbf{w}_k\|}.
\end{equation}
Then $\mathbf{g}_k(0)=\mathbf{r}\big((m-2)\mathbf{x}_k,\overline{\rho}_k\big)=(m-2)^{m-1}(\overline{\rho}_k\B-\A)\mathbf{x}_k^{m-1}\geq0$ and
\begin{displaymath}
\begin{aligned}
\mathbf{g}_k'(0)&=(m-1)(\overline{\rho}_k\B-\A)\big((m-2)\mathbf{x}_k\big)^{m-2}\mathbf{y}_k-\frac{\B\mathbf{x}_k^{m-1}}{(1+\eta)\|\mathbf{w}_k\|}\\
&=(m-2)^{m-2}\mathbf{J}_\mathbf{x}\mathbf{r}(\mathbf{x}_k,\overline{\rho}_k)\mathbf{y}_k-\frac{\B\mathbf{x}_k^{m-1}}{(1+\eta)\|\mathbf{w}_k\|}\\
&=\left[(m-2)^{m-2}-\frac{1}{1+\eta}\right]\frac{\B\mathbf{x}_k^{m-1}}{\|\mathbf{w}_k\|}>0.
\end{aligned}
\end{displaymath}
Hence, there are scalars $\theta_k\in(0,1]$ such that $\mathbf{g}_k(\theta_k)\geq0$.
\end{proof}

According to \Cref{lemma:choice of theta 1} we can choose $\theta_k\in(0,1]$ in \Cref{alg:GNNI} such that the sequence $\{\overline{\rho}_k\}$ is strictly decreasing. Besides, from the proof of \Cref{lemma:choice of theta 1}, we take $\eta=0$ if $m\geq4$. Similar to \cite{LGL}, we have the following analysis. Using Taylor's Expansion, we have
\begin{equation}\label{equ:Taylor Expansion of h(theta)}
\begin{aligned}
\mathbf{h}_k(\theta_k)&=\mathbf{r}(\tilde{\mathbf{x}}_{k+1},\rho_k)\\
&=\mathbf{r}\big((m-1)\mathbf{x}_k,\rho_k\big)+\mathbf{J_xr}\big((m-1)\mathbf{x}_k,\rho_k\big)(\theta_k\mathbf{y}_k-\mathbf{x}_k)+\mathbf{R}(\theta_k\mathbf{y}_k,\mathbf{x}_k,\rho_k),\\
\end{aligned}
\end{equation}
where for some constant $\alpha_2$ (independent of $k$)
\begin{equation}\label{equ:Remainder of Taylor of h(theta)}
\|\mathbf{R}(\theta_k\mathbf{y}_k,\mathbf{x}_k,\rho_k)\|\leq\alpha_2\|\theta_k\mathbf{y}_k-\mathbf{x}_k\|^2.
\end{equation}
Besides,
\begin{displaymath}
\begin{aligned}
\mathbf{J}_{\mathbf{x}}\mathbf{r}((m-1)\mathbf{x}_k,\rho_k)&=(m-1)^{m-2}\mathbf{J}_{\mathbf{x}}\mathbf{r}(\mathbf{x}_k,\rho_k)\mathbf{x}_k\\
&=(m-1)^{m-2}\mathbf{r}(\mathbf{x}_k,\rho_k)=\mathbf{r}((m-1)\mathbf{x}_k,\rho_k).\\
\end{aligned}
\end{displaymath}
Then by \eqref{equ:Taylor Expansion of h(theta)}
\begin{equation}\label{equ:Taylor h(theta) 2}
\begin{aligned}
\mathbf{h}_k(\theta_k)&=(m-1)^{m-2}\mathbf{J}_{\mathbf{x}}\mathbf{r}(\mathbf{x}_k,\rho_k)\theta_k\mathbf{y}_k+\mathbf{R}(\theta_k\mathbf{y}_k,\mathbf{x}_k,\rho_k)\\
&=(m-1)^{m-2}\frac{\theta_k\B\mathbf{x}_k^{m-1}}{\|\mathbf{w}_k\|}+\mathbf{R}(\theta_k\mathbf{y}_k,\mathbf{x}_k,\rho_k).\\
\end{aligned}
\end{equation}

\begin{lemma}\label{lemma:choice of theta 2}
Let the sequence $\{\overline{\rho}_k,\mathbf{x}_k,\mathbf{y}_k\}$ be generated by \Cref{alg:GNNI}, with $\{\overline{\rho}_k\}$ bounded. Assume that $\|\mathbf{y}_k-\mathbf{x}_k\|\leq\frac{\min\big((\B-\A)\mathbf{x}_k^{m-1}\big)}{\alpha_1\alpha_2}$, where $\alpha_1$ and $\alpha_2$ are as in \eqref{equ:|wk||yk-xk| bounded} and \eqref{equ:Remainder of Taylor of h(theta)}. Then \eqref{equ:h(theta)} holds with $\theta_k=1$.
\end{lemma}

\begin{proof}
From the proof of \Cref{lemma:choice of theta 1} we can see that \eqref{equ:h(theta)} always holds when $m=2$ and $\theta_k=1$. So we assume $m\geq3$. By \Cref{lemma:wk(yk-xk) bounded} and the assumption,
\begin{equation}\label{equ:|yk-xk| bounded}
\frac{\B\mathbf{x}_k^{m-1}}{\alpha_2\|\mathbf{w}_k\|\|\mathbf{y}_k-\mathbf{x}_k\|}\geq\frac{\min(\B\mathbf{x}_k^{m-1})}{\alpha_1\alpha_2}\geq\|\mathbf{y}_k-\mathbf{x}_k\|\mathbf{e},
\end{equation}
where $\mathbf{e}=[1,\ldots,1]^\top$. It follows from \eqref{equ:Remainder of Taylor of h(theta)} and \eqref{equ:|yk-xk| bounded} that
\begin{equation}\label{equ:(B-A)xk/wk geq R}
\frac{\B\mathbf{x}_k^{m-1}}{\|\mathbf{w}_k\|}\geq\alpha_2\|\mathbf{y}_k-\mathbf{x}_k\|^2\mathbf{e}\geq\|\mathbf{R}(\mathbf{y}_k,\mathbf{x}_k,\rho_k)\|.
\end{equation}
Then by \eqref{equ:Taylor h(theta) 2} and \eqref{equ:(B-A)xk/wk geq R} we have
\begin{displaymath}
\begin{aligned}
\mathbf{h}_k(1)&=(m-1)^{m-2}\frac{\B\mathbf{x}_k^{m-1}}{\|\mathbf{w}_k\|}+\mathbf{R}(\mathbf{y}_k,\mathbf{x}_k,\rho_k)\\
&\geq\big(1+(m-2)^{m-2}\big)\frac{\B\mathbf{x}_k^{m-1}}{\|\mathbf{w}_k\|}+\mathbf{R}(\mathbf{y}_k,\mathbf{x}_k,\rho_k)\\
&\geq(m-2)^{m-2}\frac{\B\mathbf{x}_k^{m-1}}{\|\mathbf{w}_k\|}\geq\frac{\B\mathbf{x}_k^{m-1}}{(1+\eta)\|\mathbf{w}_k\|}.
\end{aligned}
\end{displaymath}
Thus \eqref{equ:h(theta)} holds with $\theta_k=1$.
\end{proof}

We can now strengthen \Cref{lemma:choice of theta 1} as follows.

\begin{lemma}\label{lemma:choice of theta 3}
The condition \eqref{equ:h(theta)} holds for a sequence $\{\theta_k\}$ with $\theta_k\in[\xi,1]$ for a fixed $\xi>0$.
\end{lemma}

\begin{proof}
The condition \eqref{equ:h(theta)} holds with $\theta_k=1$ when $m=2$. For $m\geq3$, a particular sequence $\{\theta_k\}$ can be defined by
\begin{equation}\label{equ:definition of theta 2}
\theta_k=\left\{
\begin{aligned}
&1,\quad&{\rm if~}\mathbf{h}_k(1)\geq\frac{\B\mathbf{x}_k^{m-1}}{(1+\eta)\|\mathbf{w}_k\|};\\
&\eta_k,\quad&{\rm otherwise};
\end{aligned}\right.
\end{equation}
where $\eta_k=\sup\{\xi_k:\mathbf{g}_k'(\theta)\geq0{\rm~on~}[0,\xi_k]\}$ with $\mathbf{g}_k(\theta)$ given by \eqref{equ:g(theta)}. Recall that $\mathbf{g}_k(0)\geq0$ and
\begin{equation}\label{equ:g'(0)}
\mathbf{g}_k'(0)=\left[(m-2)^{m-2}-\frac{1}{1+\eta}\right]\frac{\B\mathbf{x}_k^{m-1}}{\|\mathbf{w}_k\|}>0.
\end{equation}
We can easily see that $0<\eta_k\leq1$, $\mathbf{g}_k(\eta_k)\geq0$, and $\mathbf{g}_k'(\eta_k)\ngtr0$.

Suppose $\theta_k$ is not bounded below by any $\xi>0$, then there exists a subsequence $\{\theta_{k_j}\}$ such that $\lim\limits_{j\rightarrow\infty}\eta_{k_j}=0$. Since $\{\mathbf{x}_k\}$ is bounded, we may assume that $\lim\limits_{j\rightarrow\infty}\mathbf{x}_{k_j}=\mathbf{v}$ exists. Note that we have $\mathbf{v}>0$. Since the sequence $\{\rho_k\}$ is decreasing and bounded below by $\lambda$, we know from \Cref{lemma:wk(yk-xk) bounded} that $\|\mathbf{w}_k\|\|\mathbf{y}_k-\mathbf{x}_k\|\leq K$ for some constant $K>0$. Now the sequence $\{\mathbf{w}_{k_j}\}$ must be bounded, otherwise $\mathbf{h}_{k_j}(1)\geq\frac{\B\mathbf{x}_{k_j}^{m-1}}{(1+\eta)\|\mathbf{w}_{k_j}\|}$ for some $k_j$ by \Cref{lemma:choice of theta 2} and $\eta_{k_j}$ would be undefined. Thus by \eqref{equ:g'(0)} we have $\mathbf{g}_{k_j}(0)\geq\mathbf{q}$ for some $\mathbf{q}>0$ and all $j$ sufficiently large. Note that $|\mathbf{g}_{k_j}'(\eta_{k_j})-\mathbf{g}_{k_j}'(0)|=|\mathbf{h}_{k_j}'(\eta_{k_j})-\mathbf{h}_{k_j}'(0)|\leq M\eta_{k_j}$ for a constant $M>0$ since $\mathbf{h}_k(\theta)=\mathbf{r}\big((m-2)\mathbf{x}_k+\theta\mathbf{y}_k,\rho_k\big)$ has high order derivatives and $\{\mathbf{x}_k\}$, $\{\mathbf{y}_k\}$, $\{\rho_k\}$ are all bounded. Since $\lim\limits_{j\rightarrow\infty}\eta_{k_j}=0$, we then have $\mathbf{g}_{k_j}'(\eta_{k_j})>0$ for $j$ sufficiently large, which leads to a contradiction.
\end{proof}

By \Cref{lemma:choice of theta 3}, it is easy to see that we can always find proper $\{\theta_k\}$ satisfying condition \eqref{equ:h(theta)} by halving procedure. We prove the convergence of \Cref{alg:GNNI} when $\{\theta_k\}$ satisfies condition \eqref{equ:h(theta)} in the following.

\begin{theorem}\label{thm:covergence of GNNI}
Let $\mathcal{A},\mathcal{B}\in T_{m,n}$ be a generalized $\mathcal{M}$-tensor pair, and let $\{\rho_k\}$, $\{\mathbf{x}_k\}$, and $\{\mathbf{y}_k\}$ be generated by \Cref{alg:GNNI}. Then the monotonically decreasing sequence $\{\rho_k\}$ converges to $\lambda$, and $\{\mathbf{x}_k\}$ converges to $\mathbf{x}_{\ast}$. Moreover, $\{\mathbf{y}_k\}$ converges to $\mathbf{x}_{\ast}$ as well.
\end{theorem}

\begin{proof}
By \eqref{equ:rhok - rhok+1} and \eqref{equ:h(theta)} we have
\begin{equation}\label{equ:rhok-rhok+1 (convergence thm)}
\rho_k-\rho_{k+1}=\min\frac{\mathbf{h}_k(\theta_k)}{\B\tilde{\mathbf{x}}_{k+1}^{m-1}}\geq\min\frac{\theta_k\B\mathbf{x}_k^{m-1}}{(1+\eta)\|\mathbf{w}_k\|\B\tilde{\mathbf{x}}_{k+1}^{m-1}}\geq\min\frac{\xi\B\mathbf{x}_k^{m-1}}{(1+\eta)\|\mathbf{w}_k\|\B\tilde{\mathbf{x}}_{k+1}^{m-1}}.
\end{equation}

Since $\theta_k\in(0,1]$, we have $\|\tilde{\mathbf{x}}_{k+1}\|=\|(m-2)\mathbf{x}_k+\theta_k\mathbf{y}_k\|\leq m-1.$ Since the sequence $\{\rho_k\}$ converges, we obtain from \eqref{equ:rhok-rhok+1 (convergence thm)} that $\lim\limits_{k\rightarrow\infty}\|\mathbf{w}_k\|^{-1}\min(\B\mathbf{x}_k^{m-1})=0$.

Suppose that $\min(\B\mathbf{x}_k^{m-1})$ is not bounded below by a positive constant, then there exists a subsequence $\{k_j\}$ such that $\lim\limits_{j\rightarrow\infty}\min(\B\mathbf{x}_{k_j}^{m-1})=0$. We may assume that $\lim\limits_{j\rightarrow\infty}\mathbf{x}_{k_j}=\mathbf{v}$. Hence $\min(\B\mathbf{v}^{m-1})=\lim\limits_{j\rightarrow\infty}\min(\B\mathbf{x}_{k_j}^{m-1})=0$. However, we have $\mathbf{v}>0$ and $\B\mathbf{v}^{m-1}=(\B-\A)\mathbf{v}^{m-1}+\A\mathbf{v}^{m-1}>0$, which leads to a contradiction. So $\min(\B\mathbf{x}_k^{m-1})$ is bounded below by a positive constant. Thus $\lim\limits_{j\rightarrow\infty}\|\mathbf{w}_k\|^{-1}=0$.

Let $\mathbf{v}$ be any limit point of $\{\mathbf{x}_k\}$ with $\lim\limits_{j\rightarrow\infty}\mathbf{x}_{k_j}=\mathbf{v}>0$. If $\lim\limits_{j\rightarrow\infty}\rho_{k_j}=\rho>\rho(\A\mathbf{v}^{m-2},\B\mathbf{v}^{m-2})$, then $\lim\limits_{j\rightarrow\infty}\mathbf{w}_{k_j}$ exists, contradictory to $\lim\limits_{j\rightarrow\infty}\|\mathbf{w}_k\|^{-1}=0$. Thus $\rho=\rho(\A\mathbf{v}^{m-2},\B\mathbf{v}^{m-2})$, we then have $\mathbf{v}=\mathbf{x}_{\ast}$ as in the proof of \Cref{lemma:wk(yk-xk) bounded}. Therefore, any convergent subsequence of $\{\mathbf{x}_k\}$ converges to the same limit $\mathbf{x}_{\ast}$, hence $\lim\limits_{k\rightarrow\infty}\mathbf{x}_k=\mathbf{x}_{\ast}$. And thus $\lim\limits_{k\rightarrow\infty}\rho_k=\max\frac{\A\mathbf{x}^{m-1}}{\B\mathbf{x}^{m-1}}=\lambda$. By \Cref{lemma:wk(yk-xk) bounded}, there exists a constant $\alpha_1>0$ such that $\|\mathbf{y}_k-\mathbf{x}_k\|\leq\alpha_1\|\mathbf{w}_k\|^{-1}$. Since $\lim\limits_{k\rightarrow\infty}\|\mathbf{w}_k\|^{-1}=0$ and $\lim\limits_{k\rightarrow\infty}\mathbf{x}_k=\mathbf{x}_{\ast}$, we have $\lim\limits_{k\rightarrow\infty}\mathbf{y}_k=\mathbf{x}_{\ast}$.
\end{proof}

To perform the convergence rate analysis for \Cref{alg:GNNI}, we start with a result about Newton's method.

\begin{theorem}\label{thm:convergence rate 1 GNNI}
Let $\mathbf{f}(\mathbf{x},\rho)$  be defined by \eqref{equ:define r and f for Newton method} with $\mathbf{f}(\mathbf{x}_{\ast},\lambda)=0$. Then the Jacobian $\mathbf{Jf}(\mathbf{x},\lambda)$ given by \eqref{equ:Jacobi of f} is nonsingular. Let $\{\mathbf{x}_k\}$ and $\{\rho_k\}$ be generated by \Cref{alg:GNNI}, with $\{\theta_k\}$ as in \Cref{lemma:choice of theta 2}. Then there exists a constant $\beta$ such that for all $(\mathbf{x}_k,\rho_k)$ sufficiently close to $(\mathbf{x}_{\ast},\lambda)$
\begin{equation}\label{equ:xk rhok quadratically converge}
\left\|\left[\begin{aligned}
&\hat{\mathbf{x}}_{k+1}\\
&\hat{\rho}_{k+1}\\
\end{aligned}\right]-\left[
\begin{aligned}
&\mathbf{x}_{\ast}\\
&\lambda\\
\end{aligned}\right]\right\|\leq\beta\left\|\left[
\begin{aligned}
&\mathbf{x}_k\\
&\rho_k\\
\end{aligned}\right]-\left[
\begin{aligned}
&\mathbf{x}_{\ast}\\
&\lambda\\
\end{aligned}\right]\right\|^2,
\end{equation}
where $\{\hat{\mathbf{x}}_{k+1},\hat{\rho}_{k+1}\}$ is generated by Newton step \eqref{equ:Newton method for solving f 1}-\eqref{equ:Newton method for solving f 3} from $\{\mathbf{x}_k,\rho_k\}$, instead of $\{\hat{\mathbf{x}}_k,\hat{\rho}_k\}$.
\end{theorem}

\begin{proof}
Assume that $(\mathbf{z}^\top,\zeta)^\top\in\mathbb{R}^{n+1}$ satisfies
\begin{equation}\label{equ:z zeta 1}
0=(\mathbf{z}^\top,\zeta)\mathbf{Jf}(\mathbf{x}_{\ast},\lambda)=(\mathbf{z}^\top,\zeta)
\begin{bmatrix}
&-\mathbf{J_xr}(\mathbf{x}_{\ast},\lambda)&-\B\mathbf{x}_{\ast}^{m-1}\\
&-\mathbf{x}_{\ast}^\top&0\\
\end{bmatrix}.
\end{equation}
We need to show that $(\mathbf{z}^\top,\zeta)=0$. By \eqref{equ:definition of Jxr},
\begin{equation}\label{equ:z zeta 2}
(m-1)\mathbf{z}^\top(\lambda\B-\A)\mathbf{x}_{\ast}^{m-2}+\zeta\mathbf{x}_{\ast}^\top=0.
\end{equation}
Postmultiplying both sides by $\mathbf{x}_{\ast}$, we have $\zeta=-(m-1)\mathbf{z}^\top(\lambda\B-\A)\mathbf{x}_{\ast}^{m-1}=0$. Then by \eqref{equ:z zeta 2}, $\lambda\mathbf{z}^\top\B\mathbf{x}_{\ast}^{m-2}=\mathbf{z}^\top\A\mathbf{x}_{\ast}^{m-2}$. Thus by Perron-Frobenius \Cref{thm:Perron Frobenius for matrix} we have $\mathbf{z}=c\mathbf{x}_{\ast}$ for some $c$. By the equation $\mathbf{z}^\top\B\mathbf{x}_{\ast}^{m-1}=0$ from \eqref{equ:z zeta 1}, we have $c=0$ and hence $\mathbf{z}=0$, $(\mathbf{z}^\top,\zeta)=0$. Thus $\mathbf{Jf}(\mathbf{x}_{\ast},\lambda)$ is nonsingular. We also know that $\mathbf{Jf}(\mathbf{x}_{\ast},\lambda)$ satisfies a Lipschitz condition at $(\mathbf{x}_{\ast},\lambda)$ since its Fr$\rm{\acute{e}}$chet derivative is continuous in a neighborhood of $(\mathbf{x}_{\ast},\lambda)$. Then the inequality \eqref{equ:xk rhok quadratically converge} is a basic result of Newton's method, referring \cite[Theorem 5.1.2]{Kelley} for example.
\end{proof}

Next we examine how $\|\mathbf{x}_k-\mathbf{x}\|$ and $|\rho_k-\lambda|$ are related in \Cref{thm:convergence rate 1 GNNI}.

\begin{theorem}\label{thm:convergence rate 2 GNNI}
Let $\{\mathbf{x}_k,\rho_k\}$ be generated by \Cref{alg:GNNI}, with $\{\theta_k\}$ as in \Cref{lemma:choice of theta 3}. Then there exist constants $c_1,c_2>0$ such that $c_1\|\mathbf{x}_k-\mathbf{x}_{\ast}\|\leq|\rho_k-\lambda|\leq c_2\|\mathbf{x}_k-\mathbf{x}_{\ast}\|$ for all $k\geq0$.
\end{theorem}

\begin{proof}
For an arbitrary $\mathbf{x}\in\mathbb{R}_+^n$, $\mathbf{x}>0$, consider the equation
\begin{displaymath}
\frac{\A\mathbf{x}^{m-1}}{\B\mathbf{x}^{m-1}}=\frac{(\lambda\B-\lambda\B+\A)\mathbf{x}^{m-1}}{\B\mathbf{x}^{m-1}}=\lambda\mathbf{e}-\frac{\mathbf{r}(\mathbf{x},\lambda)}{\B\mathbf{x}^{m-1}}.
\end{displaymath}
The Fr$\rm{\acute{e}}$chet derivative of $\frac{\A\mathbf{x}^{m-1}}{\B\mathbf{x}^{m-1}}$ is then given by
\begin{displaymath}
-D(\B\mathbf{x}^{m-1})^{-1}\mathbf{J_xr}(\mathbf{x},\lambda)+(m-1)(D(\B\mathbf{x}^{m-1}))^{-2}\B\mathbf{x}^{m-2}D(\mathbf{r}(\mathbf{x},\lambda)),
\end{displaymath}
where $D(\mathbf{v})$ is defined as
\begin{displaymath}
D(\mathbf{v})=
\begin{bmatrix}
v_1 &     &        &     \\
    & v_2 &        &     \\
    &     & \ddots &     \\
    &     &        & v_n \\
\end{bmatrix}
\end{displaymath}
for any $\mathbf{v}=(v_1,v_2,\ldots,v_n)^\top\in\mathbb{R}^{n}$.
By Taylor's Formula, we have
\begin{displaymath}
\frac{\A\mathbf{x}_k^{m-1}}{\B\mathbf{x}_k^{m-1}}-\frac{\A\mathbf{x}_{\ast}^{m-1}}{\B\mathbf{x}_{\ast}^{m-1}}=-D(\B\mathbf{x}_{\ast}^{m-1})\mathbf{J_xr}(\mathbf{x}_{\ast},\lambda)(\mathbf{x}_k-\mathbf{x}_{\ast})+O(\|\mathbf{x}_k-\mathbf{x}_{\ast}\|^2)
\end{displaymath}
since $\mathbf{r}(\mathbf{x}_{\ast},\lambda)=0$. Therefore, there exists a constant $c_2>0$ such that
\begin{displaymath}
|\rho_k-\lambda|=\max\left(\frac{\A\mathbf{x}_k^{m-1}}{\B\mathbf{x}_k^{m-1}}-\frac{\A\mathbf{x}_{\ast}^{m-1}}{\B\mathbf{x}_{\ast}^{m-1}}\right)\leq\left\|\frac{\A\mathbf{x}_k^{m-1}}{\B\mathbf{x}_k^{m-1}}-\frac{\A\mathbf{x}_{\ast}^{m-1}}{\B\mathbf{x}_{\ast}^{m-1}}\right\|\leq c_2\|\mathbf{x}_k-\mathbf{x}_{\ast}\|.
\end{displaymath}
On the other hand,
\begin{displaymath}
|\rho_k-\lambda|=\max\left(\frac{\A\mathbf{x}_k^{m-1}}{\B\mathbf{x}_k^{m-1}}-\frac{\A\mathbf{x}_{\ast}^{m-1}}{\B\mathbf{x}_{\ast}^{m-1}}\right)\geq\max\left(-D(\B\mathbf{x}_{\ast}^{m-1})\mathbf{J_xr}(\mathbf{x}_{\ast},\lambda)(\mathbf{x}_k-\mathbf{x}_{\ast})\right)-c_3\|\mathbf{x}_k-\mathbf{x}_{\ast}\|^2
\end{displaymath}
for some $c_3>0$. Denote $\mathbf{q}_k=(\mathbf{x}_k-\mathbf{x}_{\ast})/\|\mathbf{x}_k-\mathbf{x}_{\ast}\|$, then $\|\mathbf{q}_k\|=1$. Since $\mathbf{x}_k,\mathbf{x}_{\ast}>0$ and $\|\mathbf{x}_k\|=\|\mathbf{x}_{\ast}\|=1$, we have $\mathbf{q}_k\ngeq0$ and $\mathbf{q}_k\nleq0$ for each $k$, that is, the elements of $\mathbf{q}_k$ are neither all nonnegative nor all non-positive. We claim that there exists a constant $c_4>0$ such that $\max\big(-D(\B\mathbf{x}_{\ast}^{m-1})\mathbf{J_xr}(\mathbf{x}_{\ast},\lambda)\mathbf{q}_k\big)\geq c_4$. If not, then there exists a subsequence $\{\mathbf{q}_{k_j}\}$ such that $\lim\limits_{j\rightarrow\infty}\mathbf{q}_{k_j}=\mathbf{q}$ with $\mathbf{q}\ngeq0$ and $\mathbf{q}\nleq0$, $\|\mathbf{q}\|=1$, and $\max\big(-D(\B\mathbf{x}_{\ast}^{m-1})\mathbf{J_xr}(\mathbf{x}_{\ast},\lambda)\mathbf{q}\big)\leq0$. Thus we have $\mathbf{J_xr}(\mathbf{x}_{\ast},\lambda)\mathbf{q}\geq0$. If $\mathbf{J_xr}(\mathbf{x}_{\ast},\lambda)\mathbf{q}\neq0$ then we can find $s>0$ large enough such that $s\mathbf{x}_{\ast}+\mathbf{q}>0$, $\mathbf{J_xr}(\mathbf{x}_{\ast},\lambda)(s\mathbf{x}_{\ast}+\mathbf{q})\geq0$, and $\mathbf{J_xr}(\mathbf{x}_{\ast},\lambda)(s\mathbf{x}_{\ast}+\mathbf{q})\neq0$. Thus $\mathbf{J_xr}(\mathbf{x}_{\ast},\lambda)$ is a nonsingular M-matrix, contradicting the fact that it is actually singular since $\mathbf{J_xr}(\mathbf{x}_{\ast},\lambda)\mathbf{x}_{\ast}=0$. Thus $\mathbf{J_xr}(\mathbf{x}_{\ast},\lambda)\mathbf{q}=0$ and $\mathbf{q}=\pm\mathbf{x}_{\ast}$, which leads to a contradiction.
\end{proof}

Now we can prove that the convergence of \Cref{alg:GNNI} is quadratic when $k$ is large enough.

\begin{theorem}\label{thm:convergence rate 3 GNNI}
Let $\{\mathbf{x}_k,\rho_k\}$ be generated by \Cref{alg:GNNI}, with $\{\theta_k\}$ as in \Cref{lemma:choice of theta 3}. Then, for $k$ sufficiently large, $\rho_k$ converges to $\lambda$ quadratically and $\mathbf{x}_k$ converges to $\mathbf{x}_{\ast}$ quadratically.
\end{theorem}

\begin{proof}
We assume that $(\mathbf{x}_k,\rho_k)$ is sufficiently close to $(\mathbf{x}_{\ast},\lambda)$. Let $\{\hat{\mathbf{x}}_{k+1},\hat{\rho}_{k+1}\}$ be generated by Newton step \eqref{equ:Newton method for solving f 1}-\eqref{equ:Newton method for solving f 3} from $\{\mathbf{x}_k,\lambda\}$, instead of $\{\hat{\mathbf{x}}_k,\hat{\rho}_k\}$, and assume that \eqref{equ:xk rhok quadratically converge} holds.

Let $\zeta_k=\rho_k-\lambda$. By \eqref{equ:xk rhok quadratically converge} and \Cref{thm:convergence rate 2 GNNI}, we have $\hat{\rho}_{k+1}-\lambda=O(\zeta_k^2)$. From \eqref{equ:Newton method update rho k+1}, we have $\rho_k-\lambda-\frac{1}{(m-1)\mathbf{x}_k^\top\mathbf{y}_k\|\mathbf{w}_k\|}=\hat{\rho}_{k+1}-\lambda=O(\zeta_k^2)$. It follows that
\begin{equation}\label{equ:norm of w k}
\|\mathbf{w}_k\|=\frac{1}{(m-1)\mathbf{x}_k^\top\mathbf{y}_k\zeta_k(1-O(\zeta_k))}.
\end{equation}
Then by \Cref{lemma:wk(yk-xk) bounded}, we have $\|\mathbf{y}_k-\mathbf{x}_k\|=O(\|\mathbf{w}_k\|^{-1})=O(\zeta_k).$
Thus
\begin{equation}\label{equ:epsilon k+1 and epsilon k 1}
\zeta_{k+1}=\zeta_k+\max\frac{\A\tilde{\mathbf{x}}_{k+1}^{m-1}}{\B\tilde{\mathbf{x}}_{k+1}^{m-1}}-\rho_k=\zeta_k-\min\frac{\mathbf{r}(\tilde{\mathbf{x}}_{k+1},\rho_k)}{\B\tilde{\mathbf{x}}_{k+1}^{m-1}}.
\end{equation}
By \Cref{lemma:choice of theta 2}, we have $\theta_k=1$ near convergence. Thus when $k$ is large enough, we have by \eqref{equ:Taylor Expansion of h(theta)}-\eqref{equ:Taylor h(theta) 2}
\begin{displaymath}
\mathbf{r}(\tilde{\mathbf{x}}_{k+1},\rho_k)=(m-1)^{m-2}\frac{\B\mathbf{x}_k^{m-1}}{\|\mathbf{w}_k\|}+O(\|\mathbf{y}_k-\mathbf{x}_k\|^2)=(m-1)^{m-2}\frac{\B\mathbf{x}_k^{m-1}}{\|\mathbf{w}_k\|}+O(\zeta_k^2).
\end{displaymath}
Then by \eqref{equ:norm of w k} we have
\begin{equation}\label{equ:epsilon k+1 and epsilon k 2}
\begin{aligned}
\zeta_k\mathbf{e}-\frac{\mathbf{r}(\tilde{\mathbf{x}}_{k+1},\rho_k)}{\B\tilde{\mathbf{x}}_{k+1}^{m-1}}&=\zeta_k\mathbf{e}-\frac{\mathbf{x}_k^\top\mathbf{y}_k\zeta_k\big(1-O(\zeta_k)\big)(m-1)^{m-1}\B\mathbf{x}_k^{m-1}}{\B\tilde{\mathbf{x}}_{k+1}^{m-1}}+O(\zeta_k^2)\\
&=\zeta_k\mathbf{e}-\frac{\mathbf{x}_k^\top(\mathbf{x}_k+\mathbf{y}_k-\mathbf{x}_k)\zeta_k\big(1-O(\zeta_k)\big)(m-1)^{m-1}\B\mathbf{x}_k^{m-1}}{\B\big((m-1)\mathbf{x}_k+\mathbf{y}_k-\mathbf{x}_k\big)^{m-1}}+O(\zeta_k^2)\\
&=\zeta_k\mathbf{e}-\frac{\big(1+O(\zeta_k)\big)\zeta_k\big(1-O(\zeta_k)\big)(m-1)^{m-1}\B\mathbf{x}_k^{m-1}}{(m-1)^{m-1}\big(1+O(\zeta_k)\big)^{m-1}\B\mathbf{x}_k^{m-1}}+O(\zeta_k^2)\\
&=\zeta_k\mathbf{e}-\frac{\big(1+O(\zeta_k)\big)\zeta_k\big(1-O(\zeta_k)\big)\mathbf{e}}{\big(1+O(\zeta_k)\big)^{m-1}\mathbf{e}}+O(\zeta_k^2)\\
&=\frac{\zeta_k\big(1+O(\zeta_k)\big)^{m-1}-\big(\zeta_k-O(\zeta_k^3)\big)}{\big(1+O(\zeta_k)\big)^{m-1}}\mathbf{e}+O(\zeta_k^2)\\
&=O(\zeta_k^2).\\
\end{aligned}
\end{equation}
Using \eqref{equ:epsilon k+1 and epsilon k 2} in \eqref{equ:epsilon k+1 and epsilon k 1}, we get $\zeta_{k+1}=O(\zeta_k^2)$. Therefore, $\rho_k$ converges to $\lambda$ quadratically and then $\mathbf{x}_k$ converges to $\mathbf{x}_{\ast}$ quadratically by \Cref{thm:convergence rate 2 GNNI}.
\end{proof}

Since $|\overline{\rho}_k-\underline{\rho}_k|/\overline{\rho}_k$ is used in the stopping criterion in \Cref{alg:GNNI}, we also present the following result.

\begin{theorem}\label{thm:convergence rate 4 GNNI}
Let $\{\mathbf{x}_k,\rho_k\}$ be generated by \Cref{alg:GNNI}, with $\{\theta_k\}$ as in \Cref{lemma:choice of theta 3}. Then, with $\theta_k=1$ for $k$ sufficiently large, $|\overline{\rho}_k-\underline{\rho}_k|/\overline{\rho}_k$ converges to 0 quadratically.
\end{theorem}

\begin{proof}
From $\frac{\A\tilde{\mathbf{x}}_{k+1}^{m-1}}{\B\tilde{\mathbf{x}}_{k+1}^{m-1}}=\overline{\rho}_k\mathbf{e}-\frac{\mathbf{r}(\tilde{\mathbf{x}}_{k+1},\overline{\rho}_k)}{\B\tilde{\mathbf{x}}_{k+1}^{m-1}}$ we have
\begin{displaymath}
\begin{aligned}
\frac{|\overline{\rho}_{k+1}-\underline{\rho}_{k+1}|}{\overline{\rho}_{k+1}}=\frac{\Big|\overline{\rho}_k-\min\big(\overline{\rho}_k\mathbf{e}-\frac{\mathbf{r}(\tilde{\mathbf{x}}_{k+1},\overline{\rho}_k)}{\B\tilde{\mathbf{x}}_{k+1}^{m-1}}\big)\Big|}{\overline{\rho}_{k+1}}&=\frac{\Big|\zeta_{k+1}-\min\big(\zeta_k\mathbf{e}-\frac{\mathbf{r}(\tilde{\mathbf{x}}_{k+1},\overline{\rho}_k)}{\B\tilde{\mathbf{x}}_{k+1}^{m-1}}\big)\Big|}{\overline{\rho}_{k+1}}\\
&=O(\zeta_k^2)=O\big(|\overline{\rho}_k-\underline{\rho}_k|^2/\overline{\rho}_k^2\big),
\end{aligned}
\end{displaymath}
where we have used $\zeta_{k+1}=O(\zeta_k^2)$ and $\zeta_k\leq\overline{\rho}_k-\underline{\rho}_k$.
\end{proof}

\section{Numerical experiments}
In this section, we present some numerical examples to verify our theory for these algorithms, and to illustrate their effectiveness. All numerical tests were done using MATLAB version R2021a and the Tensor Toolbox version 3.2.1 \cite{TensorToolbox}. This toolbox defines a new data type ``tensor''. For a ``double'' type data A, we can convert it to a ``tensor'' $\A$ by the function $\A$=tensor(A). For a ``tensor'' $\A$ and a vector $\mathbf{x}$, we can compute the tensor-vector product $\A\mathbf{x}^{m-1}$ by the function ttsv($\A$,$\mathbf{x}$,$-1$) and compute the product $\A\mathbf{x}^{m-2}$ by ttsv($\A$,$\mathbf{x}$,$-2$). In addition, we can use the function symtensor($\A$) to symmetrize a ``tensor'' $\A$. The experiments were performed on a laptop computer with an Intel Quad-Core i5-5287U CPU (2.90GHz) and 8 GB of RAM.

In all numerical experiments, we use the following settings. We set the maximum iterations to be 300. The tolerance parameter ``tol'' is set to be $10^{-13}$. For \Cref{alg:GNNI}, we take $\eta=1$ for $m=3$ and $\eta=0$ for $m=4$. For \Cref{alg:IGTNI}, referring to \cite[page 19]{Ch20}, we define the inner tolerance for $\mathbf{f}_k$ by $\|\mathbf{f}_k\|\leq\max\{\beta_k\min(\A\mathbf{x}_{k-1}^{m-1}),10^{-12}\}$ for $k\geq2$ and $\|\mathbf{f}_1\|\leq10^{-3}\min(\A\mathbf{x}_0^{m-1})$, where $\beta_k=\min\{10^{-3},\frac{\overline{\rho}_k-\underline{\rho}_k}{\overline{\rho}_k}\}$. The parameter $\varepsilon$ in step 6 of \Cref{alg:MTNI} is selected to be $0.01$ for Example 1 and Example 4 and selected to be $0.005$ for Example 2 and Example 3. We set $\mathbf{b}=[1,\ldots,1]^\top$ for \Cref{alg:MTNI,alg:GNNI} and $\mathbf{x}_0=[1,\ldots,1]^\top$ for \Cref{alg:GTNI,alg:IGTNI}. The notation ``Residual'' is defined by Residual=$\|\A\mathbf{x}^{m-1}-\lambda\B\mathbf{x}^{m-1}\|_2$.

\subsection{Generate nonsingular $\mathcal{M}$-tensors}

In our numerical tests, we need to construct a positive tensor $\mathcal{A}$ and a nonsingular $\mathcal{M}$--tensor $\mathcal{C}$, then $\mathcal{B}$ will be $\mathcal{A+C}$. We construct the nonsingular $\mathcal{M}$--tensor $\mathcal{C}$ as follows \cite[page 112]{DW}. First, we generate a nonnegative tensor $\mathcal{R}\in T_{m,n}$ containing random values drawn from the standard uniform distribution on $(0,1)$. Next, set the scalar $\gamma=(1+\omega)\cdot\max\limits_{i=1,2,\ldots,n}(\mathcal{R}\mathbf{1}^{m-1})_i$, where $\mathbf{1}=(1,1,\ldots,1)^\top$ and $\omega=0.01$. Obviously, $\mathcal{C}=\gamma\mathcal{I-R}$ is a diagonally dominant $\mathcal{Z}$--tensor, that is, $\mathcal{C}\mathbf{1}^{m-1}>0$. Thus $\mathcal{C}$ is a nonsingular $\mathcal{M}$-tensor.

\subsection{Randomly generated generalized $\mathcal{M}$-tensor pairs}

In this part, we test our algorithms on randomly generated tensor pairs $(\mathcal{A,B})$ satisfying the conditions $(\mathrm{C}1')-(\mathrm{C}4')$.

{\bf Example 1} In this example, we generate a positive tensor $\mathcal{A}\in T_{3,3}$ containing random values drawn from the standard uniform distribution on $(0,1)$ and construct the nonsingular $\mathcal{M}$--tensor $\mathcal{B-A}$ by the method in Section 5.1. The results are shown in \Cref{fig:Differ_rho(random 3x3x3),fig:Res(random 3x3x3)}.

In \Cref{fig:Differ_rho(random 3x3x3)}, the $y$ axis is $\overline{\rho}_{k}-\overline{\rho}_{k+1}$, we can see that the sequence $\{\bar{\rho}_k\}$ is monotonically decreasing for all of \Cref{alg:MTNI,alg:GTNI,alg:IGTNI,alg:GNNI}, just as we proved in \Cref{lemma:convergence of s_max MTNI,lemma:convergence of rho_max IGTNI} and \Cref{thm:covergence of GNNI}.

In \Cref{fig:Res(random 3x3x3)}, as we proved in \Cref{thm:convergence rate 4 GNNI}, we can see that \Cref{alg:GNNI} converges fastest. We show the other results in \Cref{table:example 1}. Here ``Inner Iter'' means the total number of inner iterations for solving the $\mathcal{M}$-tensor equations.

{\bf Example 2} We consider a larger example. Construct $\A,\B\in T_{4,50}$ using the same method as Example 1. For saving time, we set the maximum number of inner iterations of the first outer iteration in \Cref{alg:MTNI} to be 3000. The results are shown in \Cref{fig:Res(random 50x50x50x50)} and \Cref{table:example 2}. In \Cref{fig:Res(random 50x50x50x50)}, we can see that \Cref{alg:GTNI,alg:IGTNI,alg:GNNI} converge much faster than \Cref{alg:MTNI}.

{\bf Example 3} In this example, we compare our algorithms with GEAP (generalized eigenproblem adaptive power) method \cite{Kolda} and AG (adaptive gradient) method \cite{YYXSZ}. These two methods are designed for tensor pair $(\A,\B)$ such that $\A$ is symmetric and $\B$ is symmetric and positive definite. A real-valued tensor $\A\in T_{m,n}$ is \emph{symmetric} if $a_{i_{p(1)}\ldots i_{p(m)}}=a_{i_1\ldots i_m}$ for all $i_1,\ldots,i_m\in\{1,\ldots,n\}$ and $p\in\Pi_m$, where $\Pi_m$ denotes the space of all $m$-permutations. We let $\mathbb{S}^{[m,n]}$ denotes the space of all symmetric, real-valued, $m$-th order, $n$-dimensional tensors. A tensor $\B\in\mathbb{S}^{[m,n]}$ is \emph{positive definite} if $\B\mathbf{x}^m>0$ for all $\mathbf{x}\in\mathbb{R}^n$, $\mathbf{x}\neq0$. We denote $\mathbb{S}_+^{[m,n]}$ as the space of positive definite tensors in $\mathbb{S}^{[m,n]}$. We construct our example as follows. First, we generate two symmetric positive tensors $\mathcal{A},\mathcal{R}\in \mathbb{S}^{[6,4]}$ containing random values drawn from the standard uniform distribution on $(0,1)$. Then we set the scalar $\gamma=1.01\cdot\max\left(\max\limits_{i=1,\ldots,n}(\mathcal{R}\mathbf{1}^{m-1})_i,n^{m-1}\max\limits_{i_1,\ldots,i_m}(\mathcal{R-A})\right)$ and $\B=\gamma\mathcal{I-R+A}$. Therefore, $\B\in\mathbb{S}_+^{[6,4]}$ and $\B-\A$ is a nonsingular $\mathcal{M}$-tensor. The stopping criterion for this example is $\|\A\mathbf{x}^{m-1}-\lambda\B\mathbf{x}^{m-1}\|_2<$tol. The results are shown in \Cref{fig:Res(random 4x4x4x4x4x4 symmetric)} and \Cref{table:example 3}.

\subsection{Computing the Perron pair for a weakly irreducible nonnegative tensor}
\

{\bf Example 4 \cite[Example 2]{Ch20}} According to the Perron-Frobenius Theorem for irreducible nonnegative tensors in \cite[Theorem 3.26]{QL18}, for a weakly irreducible nonnegative tensor $\A$, $\rho(\A)$ is the unique eigenvalue with a positive eigenvector $\mathbf{x}$, and $\mathbf{x}$ is the unique nonnegative eigenvector associated with $\rho(\A)$, up to a multiplicative constant. $(\rho(\A),\mathbf{x})$ is called the Perron pair of $\A$.

Consider tensor $\A\in T_{4,50}$ such that $\A=\omega\mathcal{D}+\mathcal{C}$ of an 4-uniform connected hypergraph \cite{HQ,HQX}, where $\mathcal{D}$ is the diagonal tensor with diagonal element $d_{i\ldots i}$ equal to the degree of vertex $i$ for each $i$, and $\mathcal{C}$ is the adjacency tensor defined in \cite{Cooper,HQ,HQX}. According to \cite[Theorem 4.1]{QL18}, $\A$ is weakly irreducible. We consider the hypergraph with edge set $E=\{(i,j,j+1,j+2)\}$ for $i=1,2,3,4,5$ and $j=i+1,\ldots,n-2$. We choose $w=1$ and set $\B=\rho\mathcal{I}$, where $\rho>\rho(\A)$. The stopping criterion for this example is also $\|\A\mathbf{x}^{m-1}-\lambda\B\mathbf{x}^{m-1}\|_2<$tol. The results of this example are shown in \Cref{fig:Res(hypergraph50x50x50x50)1,fig:Res(hypergraph50x50x50x50)2} and \Cref{table:example 4}. We compare our algorithms with GEAP, AG, and NQZ \cite{NQZ09} here.

\begin{figure}[htbp]
\centering
\subfloat[\scriptsize Monotonically decreasing of $\bar{\rho}_k$\\ \centering (Example 1)]{\label{fig:Differ_rho(random 3x3x3)}
\includegraphics[width=0.31\linewidth]{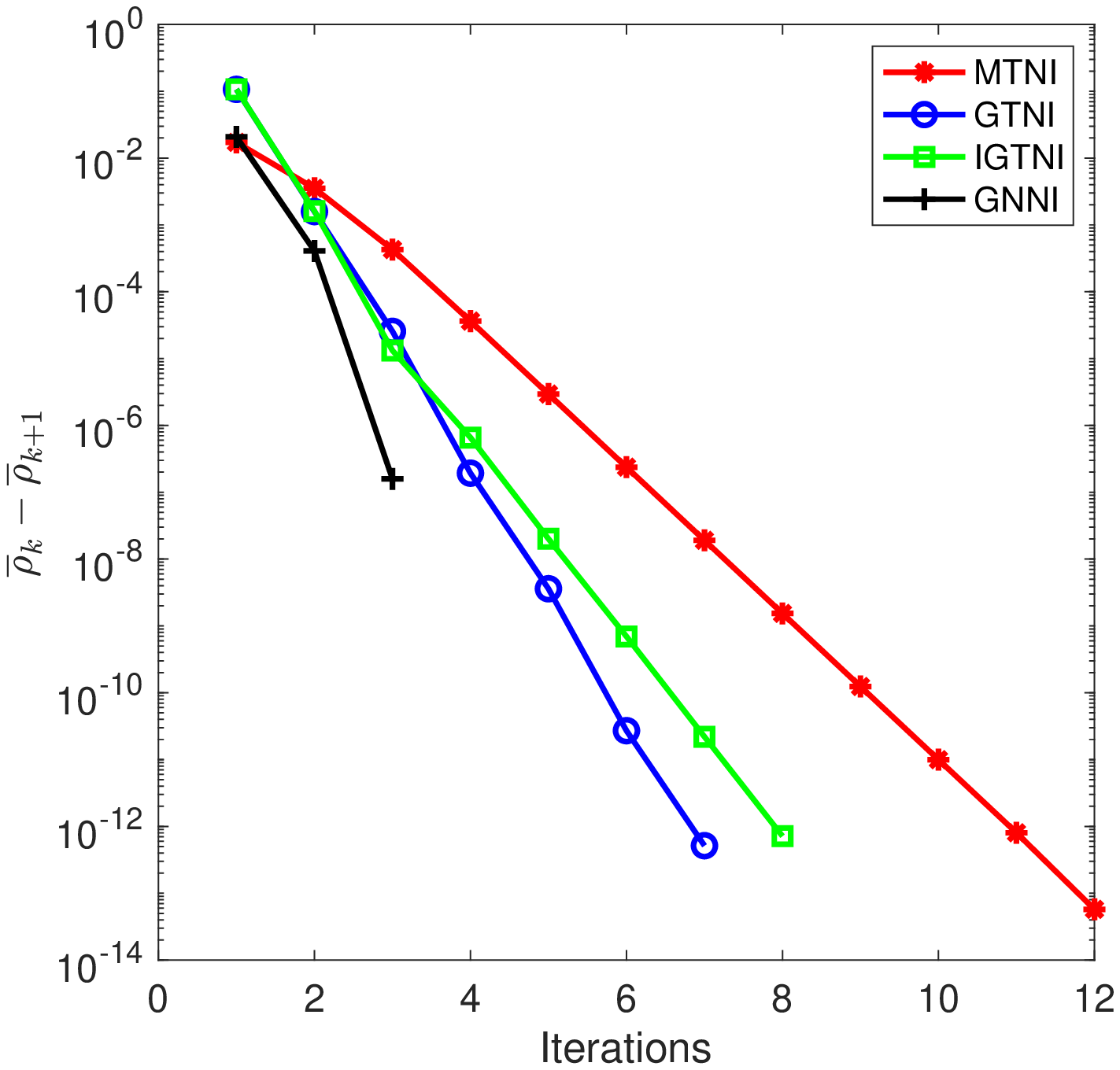}
}
\subfloat[\scriptsize Residuals (Example 1)]{\label{fig:Res(random 3x3x3)}
\includegraphics[width=0.31\linewidth]{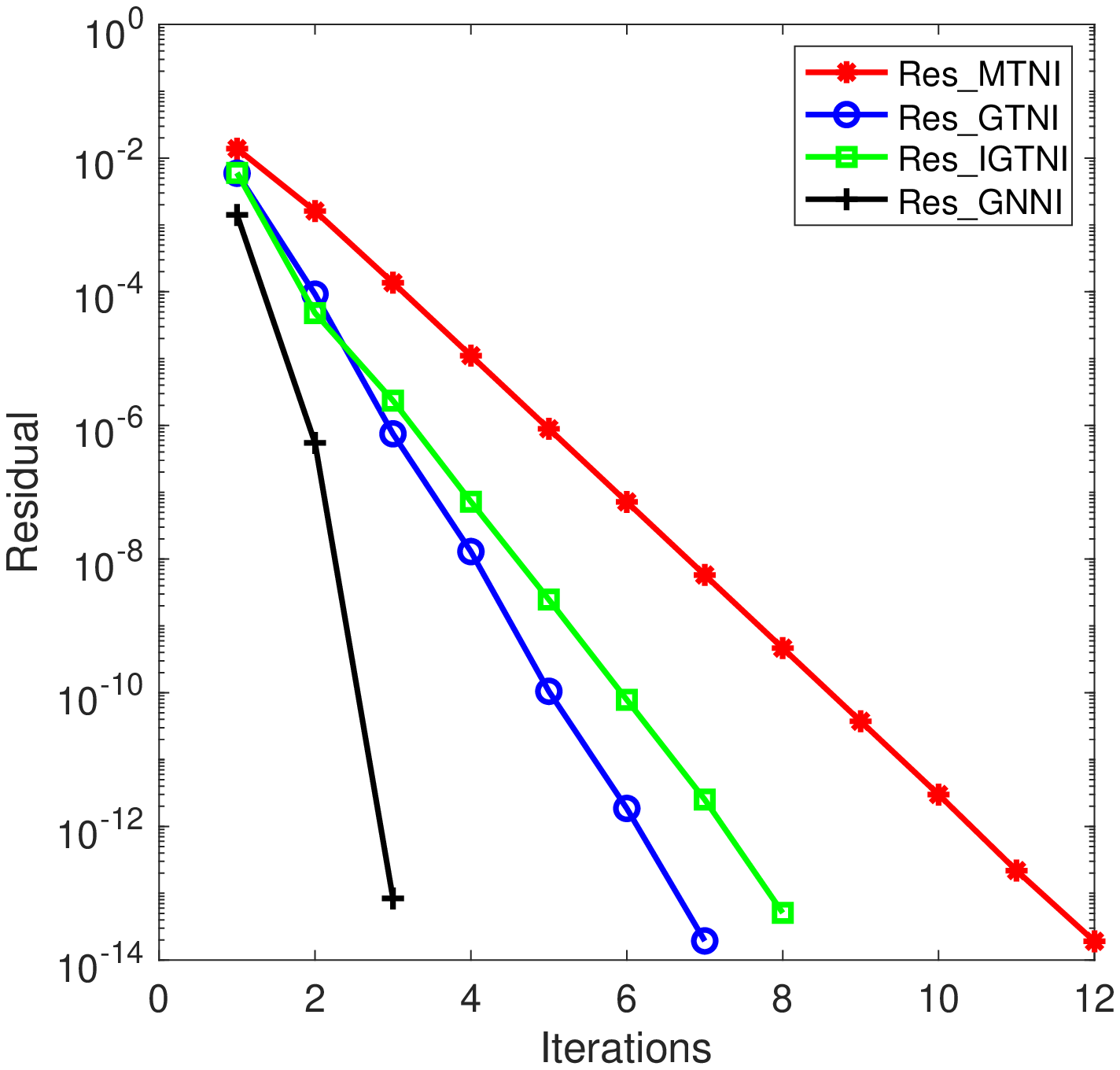}
}
\subfloat[{\scriptsize Residuals (Example 2)}]{\label{fig:Res(random 50x50x50x50)}
\includegraphics[width=0.31\linewidth]{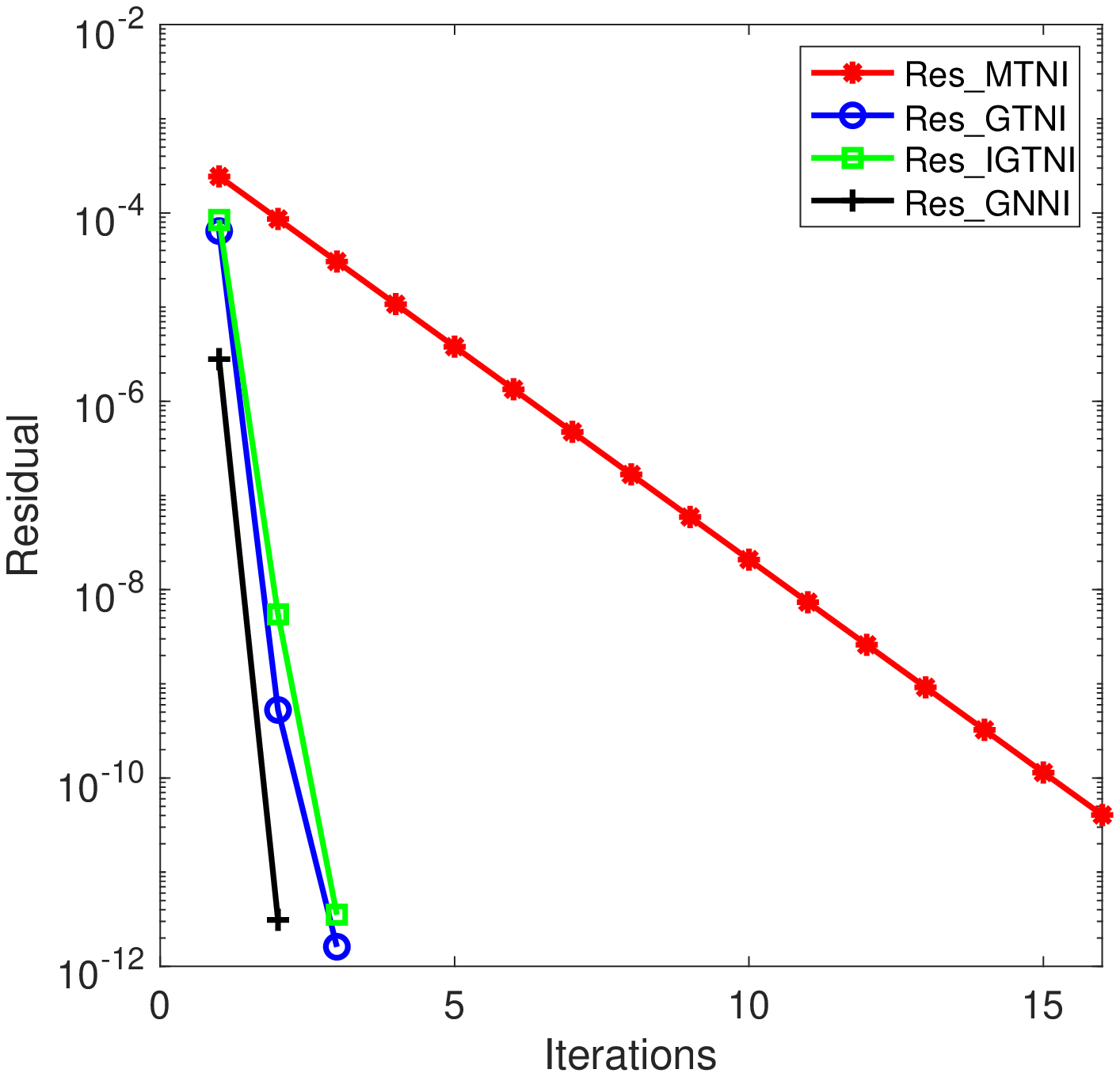}
}\\
\subfloat[{\scriptsize Residuals (Example 3)}]{\label{fig:Res(random 4x4x4x4x4x4 symmetric)}
\includegraphics[width=0.31\linewidth]{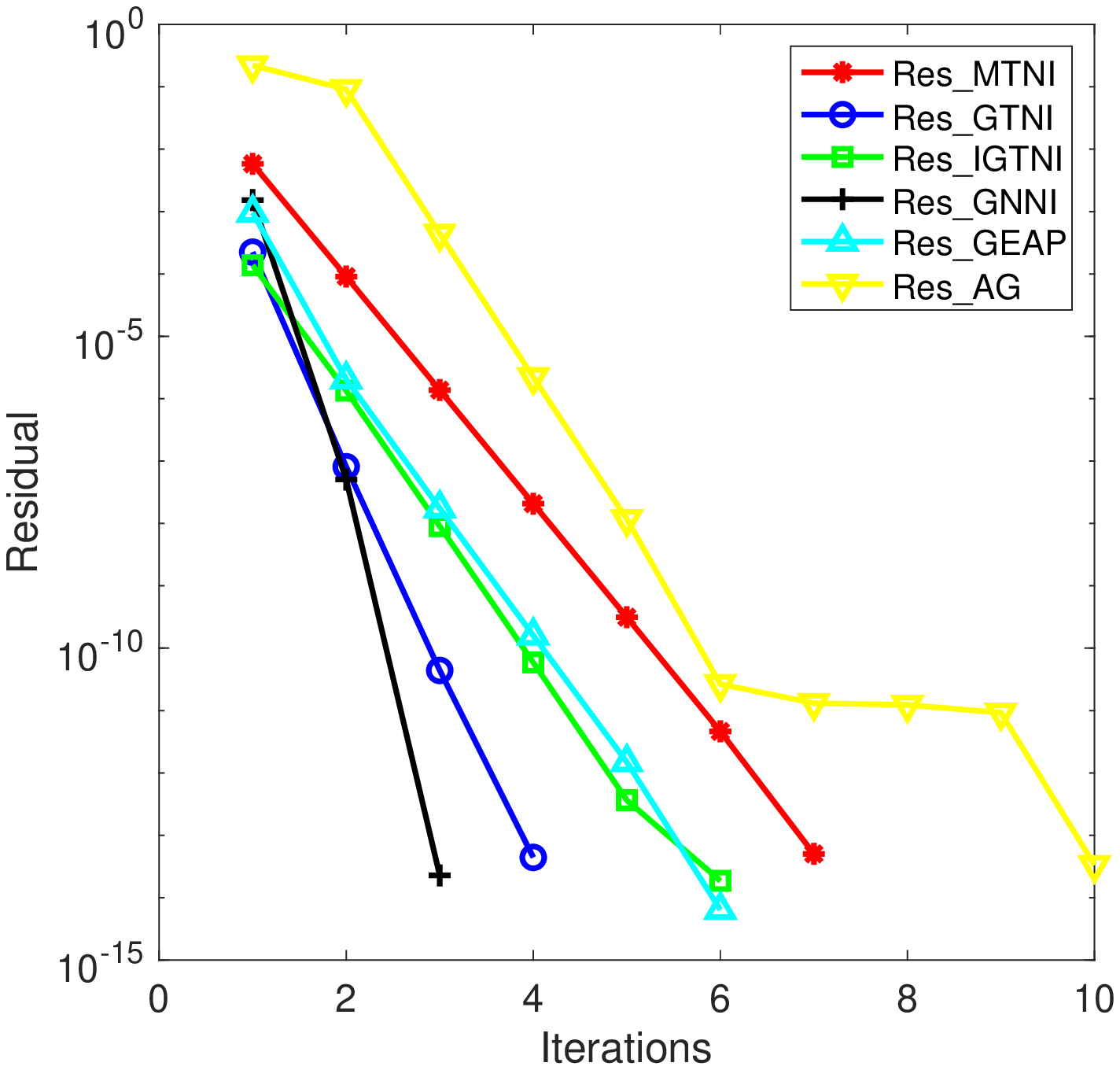}
}
\subfloat[\scriptsize Residuals (Example 4)]{\label{fig:Res(hypergraph50x50x50x50)1}
\includegraphics[width=0.31\linewidth]{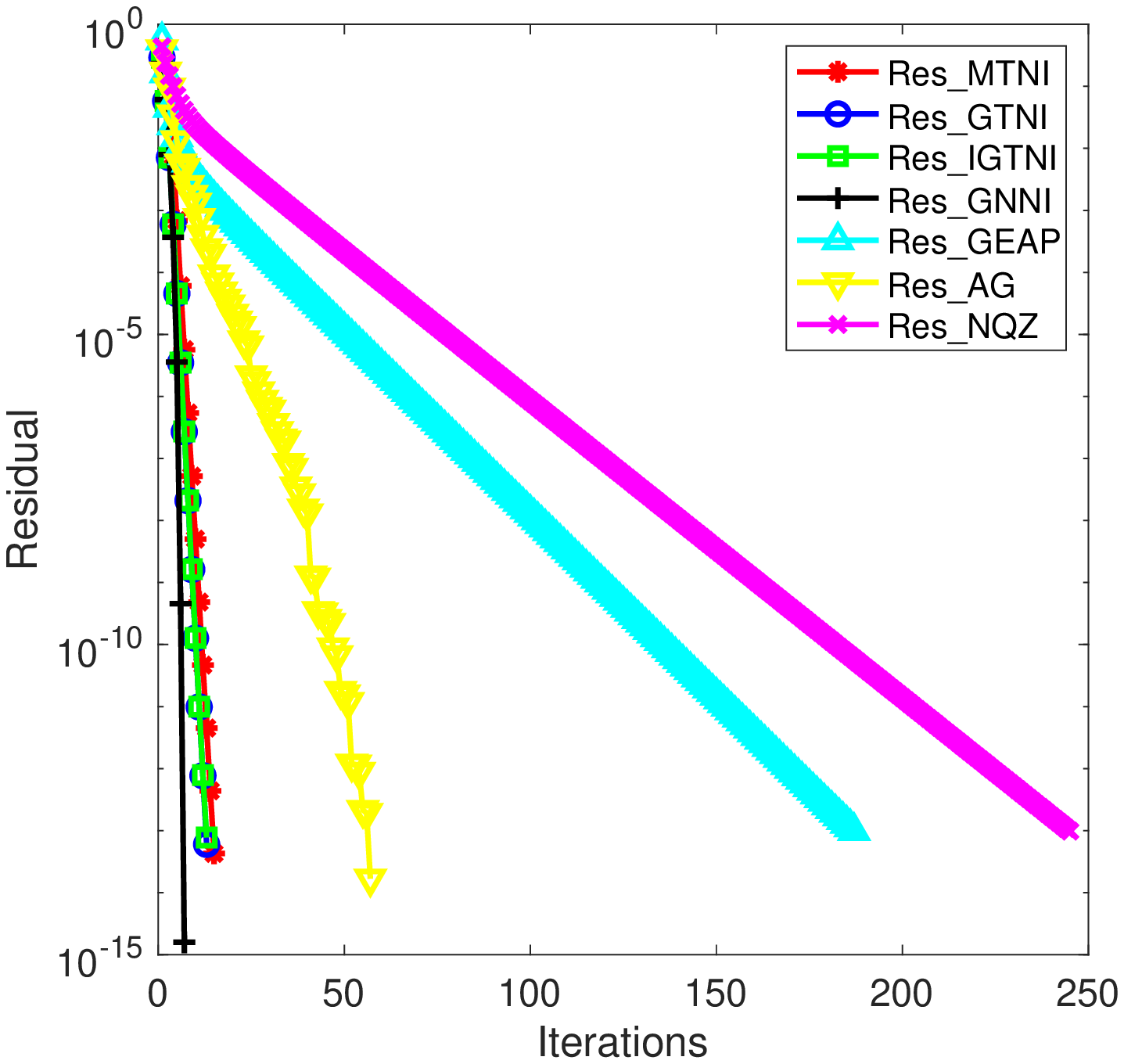}
}
\subfloat[\scriptsize Residuals (Example 4)]{\label{fig:Res(hypergraph50x50x50x50)2}
\includegraphics[width=0.31\linewidth]{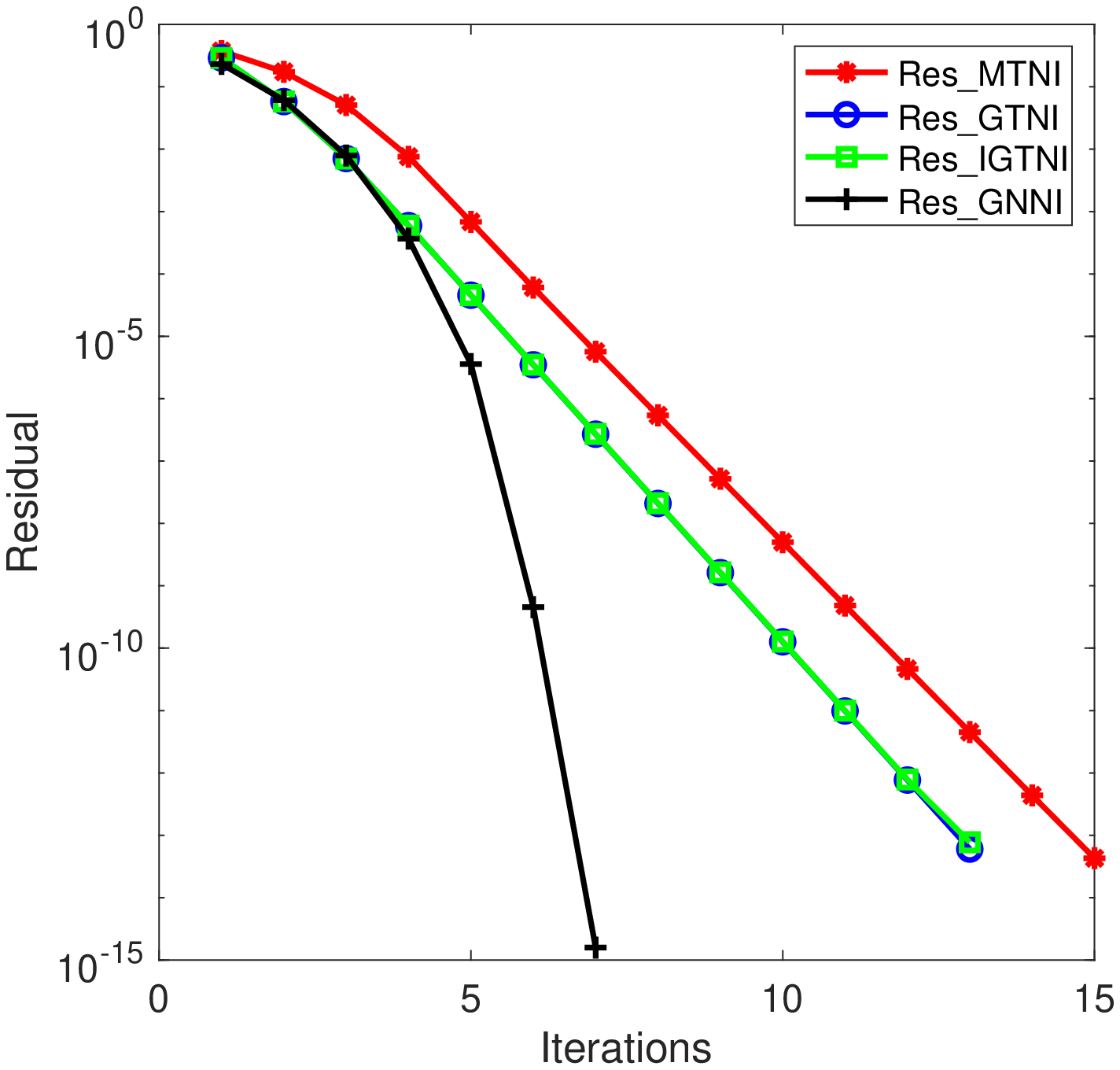}
}
\caption{Results of Example 1,2,3,4}
\end{figure}

\begin{table}[htbp]
{\footnotesize
\caption{Result of Example 1}\label{table:example 1}
\begin{center}
\begin{tabular}{|c|c|c|c|c|c|c|}\hline
method & $\lambda$ & $\mathbf{x}$ & Outer Iter & Inner Iter & Residual & Time(s)\\ \hline
MTNI & 0.8774 & [0.6028, 0.5234, 0.6023] & 12 & 17086 & 1.9082e-14 & 0.7588\\ \hline
GTNI & 0.8774 & [0.6028, 0.5234, 0.6023] & 7 & 1096 & 1.9347e-14 & 0.0757\\ \hline
IGTNI & 0.8774 & [0.6028, 0.5234, 0.6023] & 8 & 1184 & 5.0898e-14 & 0.0684\\ \hline
GNNI & 0.8774 & [0.6028, 0.5234, 0.6023] & 3 & / & 8.3544e-14 & 0.0115\\ \hline
\end{tabular}
\end{center}
}
\end{table}

\begin{table}[htbp]
{\footnotesize
\caption{Result of Example 2}\label{table:example 2}
\begin{center}
\begin{tabular}{|c|c|c|c|c|c|}\hline
method & $\lambda$ & Outer Iter & Inner Iter & Residual & Time(s)\\ \hline
MTNI & 0.9859 & 16 & 41600 & 4.0477e-11 & 161.4059\\ \hline
GTNI & 0.9859 & 3 & 6860 & 1.6140e-12 & 25.5504\\ \hline
IGTNI & 0.9859 & 3 & 6137 & 3.5401e-12 & 22.3650\\ \hline
GNNI & 0.9859 & 2 & / & 3.1140e-12 & 6.5572\\ \hline
\end{tabular}
\end{center}
}
\end{table}

\begin{table}[htbp]
{\footnotesize
\caption{Result of Example 3}\label{table:example 3}
\begin{center}
\begin{tabular}{|c|c|c|c|c|c|}\hline
method & $\lambda$ & Outer Iter & Inner Iter & Residual & Time(s)\\ \hline
MTNI & 0.6712 & 7 & 19475 & 5.0431e-14 & 0.9461\\ \hline
GTNI & 0.6712 & 4 & 379 & 4.4267e-14 & 0.0443\\ \hline
IGTNI & 0.6712 & 6 & 400 & 1.8631e-14 & 0.0366\\ \hline
GNNI & 0.6712 & 3 & / & 2.2748e-14 & 0.0083\\ \hline
GEAP & 0.6712 & 6 & / & 6.4047e-15 & 0.0074\\ \hline
AG & 0.6712 & 10 & / & 3.3375e-14 & 0.0042\\ \hline
\end{tabular}
\end{center}
}
\end{table}

\begin{table}[htbp]
{\footnotesize
\caption{Result of Example 4}\label{table:example 4}
\begin{center}
\begin{tabular}{|c|c|c|c|c|c|}\hline
method & $\lambda$ & Outer Iter & Inner Iter & Residual & Time(s)\\ \hline
MTNI  & 51.7310 & 15 & 6189 & 4.2891e-14 & 24.6596\\ \hline
GTNI  & 51.7310 & 13 & 5650 &  6.0176e-14 & 21.8478\\ \hline
IGTNI & 51.7310 & 13 & 4432 &  7.7549e-14 & 17.0098\\ \hline
GNNI  & 51.7310 & 7 & / &  1.5861e-15 & 0.4592\\ \hline
GEAP & 51.7310  & 187 & / & 9.8886e-14 & 1.4453\\ \hline
AG & 51.7310  & 57 & / & 1.6767e-14 & 1.1369\\ \hline
NQZ  & 51.7310 & 245 & / &  9.8212e-14 & 0.9080\\ \hline
\end{tabular}
\end{center}
}
\end{table}

\subsection{Computing the positive eigenpair for nonlinear eigenvalue problem with eigenvector nonlinearity (NEPv)}
\
In this part, We consider to apply our methods on the nonlinear eigenvalue problem with eigenvector nonlinearity (NEPv) and compare the numerical results with other methods. The NEPv is originated from the Bose-Einstein condensates (BECs), which can be regarded as a special case of tensor generalized eigenvalue problem.

In condensed matter physics, a Bose–Einstein condensate (BEC) is a state of matter that is typically formed when a gas of bosons at very low densities is cooled to temperatures very close to absolute zero. Under such conditions, a large fraction of bosons occupy the same quantum ground state. Suppose that this ground state can be represented by a wave function $\psi(\mathbf{x},t)$. Then $\psi$ is the solution of the following energy functional minimization problem under normalization constraints \cite[page 10]{MR3005624}:
\begin{equation}\label{equ:BEC minimization}
\left\{
\begin{aligned}
&\min~E\big(\psi(\cdot,t)\big)=\int_{\mathbb{R}^d}\big[\frac{1}{2}|\nabla\psi(\mathbf{x},t)|^2+V(\mathbf{x}|\psi(\mathbf{x},t)|^2)+\frac{\beta}{2}|\psi(\mathbf{x},t)|^4\big]d\mathbf{x},\\
&~\mathrm{s.t.}\int_{\mathbb{R}^d}|\psi(\mathbf{x},t)|^2d\mathbf{x}=1,~E\big(\psi(\mathbf{x},t)\big)<\infty,\\
\end{aligned}
\right.
\end{equation}
where $\mathbf{x}\in\mathbb{R}^d$ is the spatial coordinate vector ($d=1,2,3$), $V(\mathbf{x})$ is an external trapping potential, and the given real constant $\beta$ is the dimensionless interaction coefficient, see \cite[page 12]{MR3005624}.

The equation governing the motion of the condensate can be derived by \cite{MR3005624,MR4353495}
\begin{equation}\label{equ:GPE}
\left\{
\begin{aligned}
&\text{\i}\partial_t\psi(\mathbf{x},t)=-\frac{1}{2}\Delta\psi(\mathbf{x},t)+V(\mathbf{x})\psi(\mathbf{x},t)+\beta|\psi(\mathbf{x},t)|^2\psi(\mathbf{x},t),\\
&\int_{\mathbb{R}^d}|\psi(\mathbf{x},t)|^2d\mathbf{x}=1,\\
\end{aligned}
\right.
\end{equation}
which is a nonlinear Schr\"{o}dinger equation (NLSE) with cubic nonlinearity, known as the Gross-Pitaevskii equation (GPE). Here ``\i'' denotes the imaginary unit.

Using the finite difference discretization \cite{MR2983017}, one-dimensional case of the BEC problem \eqref{equ:BEC minimization} can be transformed into a nonconvex quartic optimization problem over a spherical constraint:
\begin{equation}\label{equ:nonconvex quartic optimization}
\left\{
\begin{aligned}
&\min\limits_{\mathbf{u}\in\mathbb{R}^{N-2}}\frac{\alpha}{2}\sum\limits_{i=1}^{N-2}u_i^4+\mathbf{u}^\top B\mathbf{u},\\
&\mathrm{s.t.}~\|\mathbf{u}\|_2^2=1,\\
\end{aligned}
\right.
\end{equation}
where $\alpha=\beta/h$, $h=(b-a)/(N-1)$, $[a,b]$ is the computational domain, and $N$ is the total number of partition points on $[a,b]$. The vector $\mathbf{u}=\sqrt{h}\Psi$, where $\Psi=(\psi_2,\ldots,\psi_{N-1})^\top$ and $\psi_j$ is the numerical approximation of $\psi$ on the partition point $x_j\in[a,b]$ for $j=1,\ldots,N$. The elements of $B=(b_{jk})\in\mathbb{R}^{(N-2)\times(N-2)}$ are given by
\begin{displaymath}
b_{jk}=\left\{
\begin{aligned}
&\frac{1}{h^2}+V(x_j),&j=k,\\
&-\frac{1}{2h^2},&|j-k|=1,\\
&0,&\text{otherwise}.\\
\end{aligned}
\right.
\end{displaymath}

Wu \emph{et al.} \cite{MR3704854} and Tian \emph{et al.} \cite{MR4129007} used Newton methods to compute the ground states of BECs. Huang \emph{et al.} \cite[section 3]{MR4390567} mentioned that the optimization problem \eqref{equ:nonconvex quartic optimization} is equivalent to the nonlinear eigenvalue problem
\begin{equation}\label{equ:NEPv}
\left\{
\begin{aligned}
&\alpha\A\mathbf{u}^3+B\mathbf{u}=\lambda\mathbf{u},\\\
&\|\mathbf{u}\|^2=1,\\
\end{aligned}
\right.
\end{equation} 
where $\A$ equals to the unit tensor $\mathcal{I}$. According to \cite[Lemma 1]{MR4390567}, there exists a unique eigenpair $(\lambda,\mathbf{u})$ with $\mathbf{u}>0$, and $\lambda>0$ is the smallest eigenvalue of NEPv \eqref{equ:NEPv}.

Recall that the definition of the identity tensor $\mathcal{E}$ is $\mathcal{E}\mathbf{x}^{m-1}=\|\mathbf{x}\|^{m-2}\mathbf{x}$ for all $\mathbf{x}\in\mathbb{R}^n$. Using this definition, we can rewrite \eqref{equ:NEPv} as
\begin{equation}\label{equ:NEPv rewritten}
\left\{
\begin{aligned}
&\mathcal{E}\mathbf{u}^3=\frac{1}{\lambda}(\alpha\A+\mathcal{E}\times_1B)\mathbf{u}^3,\\
&\|\mathbf{u}\|_2^2=1,\\
\end{aligned}
\right.
\end{equation}
which is a generalized eigenvalue problem for the tensor pair $(\mathcal{E},\alpha\A+\mathcal{E}\times_1B)$. Multiplying both sides by a constant $c=\max\limits_jb_{jj}$, \eqref{equ:NEPv rewritten} becomes
\begin{equation}\label{equ:NEPv rewritten 2}
\left\{
\begin{aligned}
&c\mathcal{E}\mathbf{u}^3=\frac{c}{\lambda}(\alpha\A+\mathcal{E}\times_1B)\mathbf{u}^3,\\
&||\mathbf{u}||_2^2=1.\\
\end{aligned}
\right.
\end{equation}
Obviously, the tensor $c\mathcal{E}$ is nonnegative and weakly irreducible. When $\beta$ is large enough, there exists a vector $\mathbf{v}\in\mathbb{R}^{N-2}$ satisfying $\mathbf{v}>0$ such that $(\alpha\A+\mathcal{E}\times_1B)\mathbf{v}^3>c\mathcal{E}\mathbf{v}^3$. For any $(i_2,i_3,i_4)\neq(i,i,i)$, we can easily see that
\begin{displaymath}
\begin{aligned}
(\alpha\A+\mathcal{E}\times_1B)_{i,i_2,i_3,i_4}&=\sum\limits_je_{ji_2i_3i_4}b_{ij}=b_{ii}e_{ii_2i_3i_4}+\sum\limits_{j\neq i}e_{ji_2i_3i_4}b_{ij}\\
&\leq b_{ii}e_{ii_2i_3i_4}\leq ce_{ii_2i_3i_4}=(c\mathcal{E})_{ii_2i_3i_4}.\\
\end{aligned}
\end{displaymath}
Therefore the tensor pair $(c\mathcal{E},\alpha\A+\mathcal{E}\times_1B)$ is a generalized $\mathcal{M}$-tensor pair and hence the unique eigenpair $(\frac{c}{\lambda},\mathbf{u})$ can be computed by our algorithms.

{\bf Example 5 \cite[Example 1]{MR4353495}} Consider the finite difference approximation with a grid size $h=2L/(N+1)$ of \eqref{equ:GPE} with Dirichlet boundary conditions on $[-L,L]\times[-L,L]$, where $L>0$ is large enough, i.e.,
\begin{displaymath}
\frac{\beta}{h^2}\diag(\mathbf{u}^{[2]})\mathbf{u}+B\mathbf{u}=\lambda\mathbf{u},~\mathbf{u}^\top\mathbf{u}=1,
\end{displaymath}
where $\mathbf{u}\in\mathbb{R}^n$, $n=N^2$. The matrix $B=A+V$ where $A=I\otimes L_h+L_h\otimes I$ is a negative 2D Laplacian matrix with
\begin{displaymath}
L_h=\frac{1}{h^2}
\begin{bmatrix}
2&-1& & & \\
-1&2&-1& & \\
 &\ddots&\ddots&\ddots& \\
 & &-1&2&-1\\
 & & &-1&2\\
\end{bmatrix}
\in\mathbb{R}^{N\times N}
\end{displaymath}
and $V=h^2\diag(1^2+1^2,1^2+2^2,\ldots,N^2+N^2)$ is the discretization of the harmonic potential $V(x,y)=x^2+y^2$. We compare our method GNNI with the Newton-Root-Finding Iteration (NRI) \cite{MR4353495} and the Newton-Noda Iteration (NNI) for NEPv \cite{MR4353495,DL22} in this example. For NNI, the linear system in each iteration can be ill-conditioned and is hard to solve without an appropriate preconditioner. Therefore, we use MATLAB function `bicgstab' to solve the linear system with tolerance as $10^{-9}$ and maximum iteration number as 200. For GNNI, the linear systems can be solved directly by MATLAB function `mldivide' (`$\backslash$') but we found that using `bicgstab' is faster when $n$ is relatively large. So we use `bicgstab' only on large examples for this method. For NRI, the linear system in the inner iterations has a tridiagonal (or block tridiagonal) structure, so we can solve it efficiently by the block tridiagonal LU factorization. In addition, the NRI method needs an initial interval $[a,b]$ for $\lambda$, we regard the initial $\lambda$ as known constants for this method in our experiments. 

The identity tensor $\mathcal{E}\in T_{4,n}$ can be generated by \cite[Property 2.4]{MR2854605}
\begin{equation}\label{equ:definition of identity tensor}
e_{i_1i_2i_3i_4}=\frac{1}{4!}\sum\limits_{p\in\Pi_4}\delta_{i_{p(1)}i_{p(2)}}\delta_{i_{p(3)}i_{p(4)}}
\end{equation}
for $i_1,i_2,i_3,i_4\in\{1,\ldots,n\}$, where $\delta$ is the standard Kronecker delta, i.e., $\delta_{ij}=1$ if $i=j$ and $\delta_{ij}=0$ if $i\neq j$. The results are shown in Table \ref{table:example 5}.

\begin{table}[htbp]
{\footnotesize
\caption{Result of Example 5}\label{table:example 5}
\begin{center}
\begin{tabular}{|c|c|c|c|c|c|}\hline
\multicolumn{6}{|l|}{$n=15^2$, $\beta=100000$, $L=8$}\\\hline
method & $\lambda$ & Outer Iter & Inner Iter & Residual & Time(s)\\ \hline
GNNI & 610.0432 & 5 & / & 3.5411e-13 & 0.0090\\ \hline
NRI & 610.0432 & 3 & [12,13,13] & 2.2609e-13 & 0.0465\\ \hline
NNI & 610.0432 & 5 & / & 2.4747e-12 & 0.0301\\ \hline
\multicolumn{6}{|l|}{$n=63^2$, $\beta=200000$, $L=8$}\\\hline
method & $\lambda$ & Outer Iter & Inner Iter & Residual & Time(s)\\ \hline
GNNI & 976.5971 & 4 & / & 1.9504e-12 & 1.9694\\ \hline
NRI & 976.5971 & 3 & [16,16,16] & 4.1177e-13 & 4.8632\\ \hline
NNI & 976.5971 & 4 & / & 6.8873e-13 & 2.3075\\ \hline
\multicolumn{6}{|l|}{$n=99^2$, $\beta=200000$, $L=8$}\\\hline
method & $\lambda$ & Outer Iter & Inner Iter & Residual & Time(s)\\ \hline
GNNI & 968.5025 & 4 & / & 2.7739e-12 & 18.0728\\ \hline
NRI & 968.5025 & 3 & [18,18,18] & 6.0493e-13 & 34.7647\\ \hline
NNI & 968.5025 & 4 & / & 1.0452e-12 & 19.4114\\ \hline
\end{tabular}
\end{center}
}
\end{table}

{\bf Example 6 \cite[Example 3]{DL22}} Consider the modified Gross-Pitaevskii equation (MGPE)
\begin{equation}\label{equ:MGPE}
A\mathbf{u}+V\mathbf{u}+\frac{\beta}{h^2}\mathbf{u}^{[3]}+2\frac{\alpha}{h^2}A\mathbf{u}^{[3]}=\lambda\mathbf{u},~\mathbf{u}^\top\mathbf{u}=1
\end{equation}
with an optical potential $V(x,y)=\frac{1}{2}(x^2+y^2)+40\big(\sin^2(\frac{\pi x}{2})+\sin^2(\frac{\pi y}{2})\big)$ on the domain $[-2,2]^2$, where $A=\frac{1}{2}(I\otimes L_h+L_h\otimes I)$. This equation can also be rewritten into a tensor generalized eigenvalue problem 
\begin{displaymath}
\left\{
\begin{aligned}
&c\mathcal{E}\mathbf{u}^3=\frac{c}{\lambda}\big(2\frac{\alpha}{h^2}\mathcal{I}\times_1A+\frac{\beta}{h^2}\mathcal{I}+\mathcal{E}\times_1(A+V)\big)\mathbf{u}^3,\\
&\|\mathbf{u}\|_2^2=1,\\
\end{aligned}
\right.
\end{displaymath}
where $c$ is a positive constant. Denote $\mathcal{A}=c\mathcal{E}$, $\mathcal{B}=2\frac{\alpha}{h^2}\mathcal{I}\times_1A+\frac{\beta}{h^2}\mathcal{I}+\mathcal{E}\times_1(A+V)$ and let $c=\max\big(\diag(A+V)\big)$, then we can easily see that $b_{ii_2i_3i_4}\leq a_{ii_2i_3i_4}$ for any $i=1,\ldots,n$ and $(i_2,i_3,i_4)\neq(i,i,i)$. Besides, there exsits a vector $\mathbf{v}>0$ such that $(\B-\A)\mathbf{v}^3>0$ when $\beta$ is large enough. Therefore, the MGPE also satisfies our assumptions and the unique positive eigenpair of \eqref{equ:MGPE} can be found by our methods. We compare the GNNI method with the NNI method \cite{DL22} in this example, the results are shown in Table \ref{table:example 6}.

\begin{table}[htbp]
{\footnotesize
\caption{Result of Example 6}\label{table:example 6}
\begin{center}
\begin{tabular}{|c|c|c|c|c|c|}\hline
\multicolumn{6}{|l|}{$n=63^2$, $\beta=100000$, $\alpha=100$, $L=2$}\\\hline
method & $\lambda$ & Outer Iter & Inner Iter & Residual & Time(s)\\ \hline
GNNI & 7.9343e+03 & 6 & / & 1.7693e-11 & 3.5730\\ \hline
NNI & 7.9343e+03 & 4 & / & 7.6481e-11 & 3.6068\\ \hline
\end{tabular}
\end{center}
}
\end{table}

\section{Conclusions}
In this paper, the Noda iteration(NI) method has been developed for computing the Perron pair for the generalized $\mathcal{M}$-tensor pair. We prove that MTNI, GTNI, IGTNI, and GNNI are convergent based on the techniques in \cite{CVLX,DW,LGL,ZQZ}. We test our methods on randomly generated tensor pairs, hypergraph eigenproblem as well as NEPv and the convergence on accuracy was illustrated. Acceleration of the methods may be a topic of future study. Specifically, we need to develop a faster method for solving the $\mathcal{M}$-tensor equation $(\rho_{k-1}\B-\A)\mathbf{y}_k^{m-1}=(\B-\A)\mathbf{x}_{k-1}^{m-1}$ and the choice of parameter $\varepsilon$ needs to be dicussed. For \Cref{alg:GNNI}, other ways for choosing the step size and some variations of Newton method need to be considered. In view of \Cref{alg:MTNI} and \Cref{alg:GTNI}, another work that needs to be done is to consider a more general iteration formula as $(\rho_{k-1}\B-\A)\mathbf{y}_k^{m-1}=(\alpha_{k-1}\A+\beta_{k-1}\B)\mathbf{x}_{k-1}^{m-1}$. As we can see from the numerical experiments, \Cref{alg:GTNI} performs faster than \Cref{alg:MTNI}. How to choose $\alpha_{k-1}$, $\beta_{k-1}$ in each iteration to make the method more efficient needs further discussion.

\section*{Acknowledgements}
The authors would like to thank the handling editor and the reviewers for their detailed  comments on our presentation. We also thank Prof. Qingzhi Yang of Nankai University, Prof. Ching-Sung Liu of National University of Kaohsiung, Prof. Xinming Wu of Fudan University, and Prof. Hehu Xie of Chinese Academy of Sciences for their inspiration and introducing their reprints and preprints \cite{MR4353495,MR4390567,GLL,LGL,Ch20,LGL16,DL22,MR3704854,MR4129007}.

\bibliographystyle{abbrv}
\bibliography{references}

\newpage

\section*{Appendix (MATLAB codes)}
In order to run the following codes correctly, the Tensor Toolbox \cite{TensorToolbox} needs to be added to the path.
\subsection*{ExampleGenerator.m}
\

\%\% Example 1 (random 3x3x3)

m = 3;

n = 3;

I = tenzeros(n$*$ones(1,m));

for i = 1:n

~~~~I(i$*$ones(1,m)) = 1;

end

R = tenrand(n$*$ones(1,m));

epsilon = 0.01;

c = (1 + epsilon) $*$ max(ttsv(R,ones(n,1),-1));

C = c $*$ I - R;

A = tenrand(n$*$ones(1,m));

B = A + C;

save('A(example 1 random 3x3x3).mat','A');

save('B(example 1 random 3x3x3).mat','B');

\%\% Example 2 (random 50x50x50x50)

m = 4;

n = 50;

I = tenzeros(n$*$ones(1,m));

for i = 1:n

~~~~I(i$*$ones(1,m)) = 1;

end

R = tenrand(n$*$ones(1,m));

epsilon = 0.01;

c = (1 + epsilon) $*$ max(ttsv(R,ones(n,1),-1));

C = c $*$ I - R;

A = tenrand(n$*$ones(1,m));

B = A + C;

save('A(example 2 random 50x50x50x50).mat','A');

save('B(example 2 random 50x50x50x50).mat','B');

\%\% Example 3 (random 4x4x4x4x4x4 symmetric)

m = 6;

n = 4;

R = rand(n$*$ones(1,m));

R = tensor(symtensor(tensor(round(R,4))));

I = tenzeros(size(R));

for i = 1:n

~~~~I(i$*$ones(1,m)) = 1;

end

A = rand(n$*$ones(1,m));

A = tensor(symtensor(tensor(round(A,4))));

gamma = max([1.01 $*$ max(ttsv(R,ones(n,1),-1)),max(
R(:)-A(:))$*$n\^\,(m-1)]);

B = gamma $*$ I - R + A;

save('A(example 3 random 4x4x4x4x4x4 symmetric).mat','A');

save('B(example 3 random 4x4x4x4x4x4 symmetric).mat','B');

\%\% Example 4 (hypergraph 50x50x50x50)

m = 4;

n = 50;

omega = 1;

D = tenzeros(n$*$ones(1,m));

for i = 1:5

~~~~D(i$*$ones(1,m)) = n - 2 -i;

end

D(6$*$ones(1,m)) = 12;

D(7$*$ones(1,m)) = 14;

for i = 8:n-2

~~~~D(i$*$ones(1,m)) = 15;

end

D((n-1)$*$ones(1,m)) = 10;

D(n$*$ones(1,m)) = 5;

C = tenzeros(n$*$ones(1,m));

for i1 = 1:5

~~~~for i2 = i1+1:n-2

~~~~~~~~C(i1,i2,i2+1,i2+2) = factorial(m) / factorial(m-1);

~~~~end

end

C = tensor(symtensor(C));

A = omega $*$ D + C;

I = zeros(n$*$ones(1,m));

for i = 1:n

~~~~I(i,i,i,i) = 1;

end

B = tensor(100 $*$ I);

save('A(example 4 hypergraph 50x50x50x50).mat','A');

save('B(example 4 hypergraph 50x50x50x50).mat','B');

\%\% Example 5 (Nonlinear eigenvalue--2D case)

m = 4;

N = 15;

n = N\^\,2;

L = 8;

h = 2$*$L / (N+1);

beta = 100000;

beta = beta / h\^\,2;

L = 1/h\^\,2 $*$ (2$*$eye(N)-diag(ones(N-1,1),1)-diag(ones(N-1,1),-1));

V = zeros(n,n);

for i = 1:N

~~~~for j = 1:N

~~~~~~~~V((i-1)$*$N+j,(i-1)$*$N+j) = h\^\,2 $*$ (i\^\,2+j\^\,2);

~~~~end

end

A = kron(eye(N),L)+kron(L,eye(N));

B = A + V;

kappa = sqrt(norm(A,1)*norm(A,inf)) + beta;

save('data(example 5 NEPv).mat','A','B','beta','m','n','kappa');

\%\% Example 6 (Nonlinear eigenvalue--MGPE)

m = 4;

N = 63;

n = N\^\,2;

L = 2;

h = 2$*$L / (N+1);

beta = 100000;

beta = beta / h\^\,2;

alpha = 100;

alpha = alpha / h\^\,2;

L = 1/h\^\,2 $*$ (2$*$eye(N)-diag(ones(N-1,1),1)-diag(ones(N-1,1),-1));

V = zeros(n,n);

for i = 1:N

~~~~for j = 1:N

~~~~~~~~V((i-1)$*$N+j,(i-1)$*$N+j) = 1/2 $*$ (i\^\,2+j\^\,2) + 40 $*$ (sin(pi$*$i/2)\^\,2+sin(pi$*$j/2)\^\,2);
    
~~~~end

end

A = 1/2 $*$ (kron(eye(N),L)+kron(L,eye(N)));

B = A + V;

kappa = sqrt(norm(A,1)$*$norm(A,inf)) + beta;

save('data(example 6 MGPE).mat','A','B','beta','m','n','kappa');

\subsection*{MTNI.m}
\

clearvas;

clc;

\%\% Generate Examples

\%\% Example 1

load("A(example 1 random 3x3x3).mat");

load("B(example 1 random 3x3x3).mat");

\%\% Example 2

\% load("A(example 2 random 50x50x50x50).mat");

\% load("B(example 2 random 50x50x50x50).mat");

\%\% Example 3 (random 4x4x4x4x4x4 symmetric)

\% load("A(example 3 random 4x4x4x4x4x4 symmetric).mat");

\% load("B(example 3 random 4x4x4x4x4x4 symmetric).mat");

\%\% Example 4 (hypergraph 50x50x50x50)

\% load("A(example 4 hypergraph 50x50x50x50).mat");

\% load("B(example 4 hypergraph 50x50x50x50).mat");

\%\% Initialization

tic;

m = length(size(A));

N = size(A);

n = N(1);

tol = 1e-13;

C = double(B - A);

D = zeros(n$*$ones(1,m));

d = zeros(n,1);

for i = 1 : n

~~~~d(i) = C(i,i,i); \% for m=3

\%~~~~d(i) = C(i,i,i,i); \% for m=4

~~~~D(i,i,i) = d(i); \% for m=3

\%~~~~D(i,i,i,i) = d(i); \% for m=4

end

b = ones(n,1);

x = ones(n,1);

x\_old = x;

Temp = tensor(D - C);

C = tensor(C);

b\_temp = double(ttsv(Temp,x,-1)) + b;

x\_new = (b\_temp ./ d) .\^\, (1/(m-1));

while norm(x\_new - x\_old,2) $>$ tol

~~~~x\_old = x\_new;

~~~~b\_temp = double(ttsv(Temp,x\_old,-1)) + b;

~~~~x\_new = (b\_temp ./ d) .\^\, (1/(m-1));

end

x = x\_new / norm(x\_new,2);

temp1 = ttsv(A,x,-1);

temp2 = ttsv(B,x,-1);

temp3 = temp2 - temp1;

rho\_max = max(temp1 ./ temp2);

rho = rho\_max;

y = ones(n,1);

\%\%

epsilon = 0.01; \% for example 1,4

\% epsilon = 0.005; \% for example 2,3

delta = 1e-13; \% for example 1,3

\% delta = 1e-15; \% for example 2,4

d = zeros(n,1);

MaxIter = 50;

Differ\_rho = zeros(1,MaxIter);

Res\_rho = zeros(1,MaxIter);

Res = zeros(1,MaxIter);

count = zeros(1,MaxIter);

for k = 1 : MaxIter

~~~~M = double(rho $*$ B - A);

~~~~D = zeros(n$*$ones(1,m));

~~~~for i = 1 : n

~~~~~~~~d(i) = M(i,i,i); \% for m=3

\%~~~~~~~~d(i) = M(i,i,i,i); \% for m=4

~~~~~~~~D(i,i,i) = d(i); \% for m=3

\%~~~~~~~~D(i,i,i,i) = d(i); \% for m=4

~~~~end

~~~~y\_old = y;

~~~~r = temp3;

~~~~Temp = tensor(D - M);

~~~~b\_temp = ttsv(Temp,y\_old,-1) + r;

~~~~y\_new = (b\_temp ./ d) .\^\, (1/(m-1));

~~~~count(k) = 0;

~~~~while norm(y\_new - y\_old,2) $>$ delta

~~~~~~~~y\_old = y\_new;

~~~~~~~~b\_temp = ttsv(Temp,y\_old,-1) + r;

~~~~~~~~y\_new = (b\_temp ./ d) .\^\, (1/(m-1));

~~~~~~~~count(k) = count(k) + 1;

~~~~~~~~if count(k) $>=$ 3000 \&\& k == 1

~~~~~~~~~~~~break;

~~~~~~~~end

~~~~end

~~~~y = y\_new;

~~~~rho\_old = rho\_max;

~~~~tau = min(temp3 ./ ttsv(C,y,-1));

~~~~rho = rho - (1 - rho) $*$ tau / (1 - tau);

~~~~rho = (1 + epsilon) $*$ rho;

~~~~x = y / norm(y,2);

~~~~temp1 = ttsv(A,x,-1);

~~~~temp2 = ttsv(B,x,-1);

~~~~temp3 = temp2 - temp1;

~~~~s\_max = max(temp1 ./ temp3);

~~~~rho\_max = s\_max / (1 + s\_max);

~~~~s\_min = min(temp1 ./ temp3);

~~~~rho\_min = s\_min / (1 + s\_min);

~~~~Differ\_rho(k) = rho\_old - rho\_max;

~~~~Res\_rho(k) = abs(rho\_max-rho\_min)/rho\_max;

~~~~Res(k) = norm(temp1-rho\_max$*$temp2,2);

\%~~~~if Res(k) $<$ tol  \% for symmetric case

\%~~~~~~~~break;

\%~~~~end

~~~~if Res\_rho(k) $<$ tol

~~~~~~~~break;

~~~~end

end

time=toc;

\subsection*{GTNI.m}
\

clearvas;

clc;

\%\% Generate Examples

\%\% Example 1

load("A(example 1 random 3x3x3).mat");

load("B(example 1 random 3x3x3).mat");

\%\% Example 2

\% load("A(example 2 random 50x50x50x50).mat");

\% load("B(example 2 random 50x50x50x50).mat");

\%\% Example 3 (random 4x4x4x4x4x4 symmetric)

\% load("A(example 3 random 4x4x4x4x4x4 symmetric).mat");

\% load("B(example 3 random 4x4x4x4x4x4 symmetric).mat");

\%\% Example 4 (hypergraph 50x50x50x50)

\% load("A(example 4 hypergraph 50x50x50x50).mat");

\% load("B(example 4 hypergraph 50x50x50x50).mat");

\%\% Initialization

tic;

m = length(size(A));

N = size(A);

n = N(1);

tol = 1e-13;

x = ones(n,1);

x = x / norm(x,2);

temp1 = ttsv(A,x,-1);

temp2 = ttsv(B,x,-1);

rho\_max = max(temp1 ./ temp2);

rho\_min = min(temp1 ./ temp2);

rho = 1;

y = ones(n,1);

\%\%

delta = 1e-13; \% for example 1,2

\%delta = 1e-15; \% for example 3,4

d = zeros(n,1);

MaxIter = 50;

Differ\_rho = zeros(1,MaxIter);

Res\_rho = zeros(1,MaxIter);

Res = zeros(1,MaxIter);

count = zeros(1,MaxIter);

Epsilon = zeros(1,MaxIter);

for k = 1 : MaxIter

~~~~M = double(rho $*$ B - A);

~~~~D = zeros(n$*$ones(1,m));

~~~~for i = 1 : n

~~~~~~~~d(i) = M(i,i,i); \% for m=3

\%~~~~~~~~d(i) = M(i,i,i,i); \% for m=4

~~~~~~~~D(i,i,i) = d(i); \% for m=3

\%~~~~~~~~D(i,i,i,i) = d(i); \% for m=4

~~~~end

~~~~y\_old = y;

~~~~r = temp1;

~~~~Temp = tensor(D - M);

~~~~b\_temp = ttsv(Temp,y\_old,-1) + r;

~~~~y\_new = (b\_temp ./ d) .\^\, (1/(m-1));

~~~~count(k) = 0;

~~~~while norm(y\_new - y\_old,2)$>$delta

~~~~~~~~y\_old = y\_new;

~~~~~~~~b\_temp = ttsv(Temp,y\_old,-1) + r;

~~~~~~~~y\_new = (b\_temp ./ d) .\^\, (1/(m-1));

~~~~~~~~count(k) = count(k) + 1;

~~~~end

~~~~y = y\_new;

~~~~Epsilon(k) = 1;

\%~~~~Epsilon(k) = 0.01; \% for example 4

~~~~temp = 1-min(temp1./(ttsv(A,y,-1)+temp1));

~~~~while 1

~~~~~~~~if (1+Epsilon(k)) $*$ temp $<=$ 1

~~~~~~~~~~~~break;

~~~~~~~~else

~~~~~~~~~~~~Epsilon(k) = Epsilon(k)/2;

~~~~~~~~end

~~~~end

~~~~epsilon = Epsilon(k);

~~~~rho = (1+epsilon) $*$ rho $*$ temp;

~~~~x = y / norm(y,2);

~~~~temp1 = ttsv(A,x,-1);

~~~~temp2 = ttsv(B,x,-1);

~~~~rho\_old = rho\_max;

~~~~rho\_max = max(temp1 ./ temp2);

~~~~rho\_min = min(temp1 ./ temp2);

~~~~Differ\_rho(k) = rho\_old - rho\_max;

~~~~Res\_rho(k) = abs(rho\_max-rho\_min)/rho\_max;

~~~~Res(k) = norm(temp1-rho\_max$*$temp2,2);

\%~~~~if Res(k) $<$ tol  \% for symmetric case

\%~~~~~~~~break;

\%~~~~end

~~~~if Res\_rho(k) $<$ tol

~~~~~~~~break;

~~~~end

end

time=toc;

\subsection*{IGTNI.m}
\

clearvas;

clc;

\%\% Generate Examples

\%\% Example 1

load("A(example 1 random 3x3x3).mat");

load("B(example 1 random 3x3x3).mat");

\%\% Example 2

\% load("A(example 2 random 50x50x50x50).mat");

\% load("B(example 2 random 50x50x50x50).mat");

\%\% Example 3 (random 4x4x4x4x4x4 symmetric)

\% load("A(example 3 random 4x4x4x4x4x4 symmetric).mat");

\% load("B(example 3 random 4x4x4x4x4x4 symmetric).mat");

\%\% Example 4 (hypergraph 50x50x50x50)

\% load("A(example 4 hypergraph 50x50x50x50).mat");

\% load("B(example 4 hypergraph 50x50x50x50).mat");

\%\% Initialization

tic;

m = length(size(A));

N = size(A);

n = N(1);

tol = 1e-13;

x = ones(n,1);

x = x / norm(x,2);

temp1 = ttsv(A,x,-1);

temp2 = ttsv(B,x,-1);

rho\_max = max(temp1 ./ temp2);

rho\_min = min(temp1 ./ temp2);

rho = 1;

y = ones(n,1);

\%\%

d = zeros(n,1);

MaxIter = 50;

Differ\_rho = zeros(1,MaxIter);

Res\_rho = zeros(1,MaxIter);

Res = zeros(1,MaxIter);

count = zeros(1,MaxIter);

gama = zeros(1,MaxIter);

f = zeros(1,MaxIter);

Epsilon= zeros(1,MaxIter);

for k = 1 : MaxIter

~~~~M = double(rho $*$ B - A);

~~~~D = zeros(n$*$ones(1,m));

~~~~for i = 1 : n

~~~~~~~~d(i) = M(i,i,i); \% for m=3

\%~~~~~~~~d(i) = M(i,i,i,i); \% for m=4

~~~~~~~~D(i,i,i) = d(i); \% for m=3

\%~~~~~~~~D(i,i,i,i) = d(i); \% for m=4

~~~~end

~~~~y\_old = y;

~~~~r = temp1;

~~~~gama(k) = min((rho\_max-rho\_min)/rho\_max,1e-3);

~~~~f(k) = max(gama(k)$*$min(r),1e-12);

~~~~Temp = tensor(D - M); 

~~~~b\_temp = ttsv(Temp,y\_old,-1) + r + f(k);

~~~~y\_new = (b\_temp ./ d) .\^\, (1/(m-1));

~~~~count(k) = 0;

~~~~while norm(y\_new-y\_old,2) $>$ 1e-13 \% for example 1,4

\%~~~~while norm(y\_new-y\_old,2) $>$ 1e-12 \% for example 2

\%~~~~while norm(y\_new-y\_old,2) $>$ 1e-14 \% for example 3

~~~~~~~~y\_old = y\_new;

~~~~~~~~b\_temp = ttsv(Temp,y\_old,-1) + r + f(k);

~~~~~~~~y\_new = (b\_temp ./ d) .\^\, (1/(m-1));

~~~~~~~~count(k) = count(k) + 1;

~~~~end

~~~~y = y\_new;

~~~~Epsilon(k) = 1;

\%~~~~Epsilon(k) = 0.01; \% for example 4

~~~~temp = 1-min((temp1+f(k))./(ttsv(A,y,-1)+temp1+f(k)));

~~~~while 1

~~~~~~~~if (1+Epsilon(k)) $*$ temp $<=$ 1

~~~~~~~~~~~~break;

~~~~~~~~else

~~~~~~~~~~~~Epsilon(k) = Epsilon(k)/2;

~~~~~~~~end

~~~~end

~~~~epsilon = Epsilon(k);

~~~~rho = (1+epsilon) $*$ rho $*$ temp;

~~~~x = y / norm(y,2);

~~~~temp1 = ttsv(A,x,-1);

~~~~temp2 = ttsv(B,x,-1);

~~~~rho\_old = rho\_max;

~~~~rho\_max = max(temp1 ./ temp2);

~~~~rho\_min = min(temp1 ./ temp2);

~~~~Differ\_rho(k) = rho\_old - rho\_max;

~~~~Res\_rho(k) = abs(rho\_max-rho\_min)/rho\_max;

~~~~Res(k) = norm(temp1-rho\_max$*$temp2,2);

\%~~~~if Res(k) $<$ tol  \% for symmetric case

\%~~~~~~~~break;

\%~~~~end

~~~~if Res\_rho(k) $<$ tol

~~~~~~~~break;

~~~~end

end

time = toc;

\subsection*{GNNI.m}
\

clearvas;

clc;

\%\% Load Examples

\%\% Example 1

load("A(example 1 random 3x3x3).mat");

load("B(example 1 random 3x3x3).mat");

m = length(size(A));

N = size(A);

n = N(1);

for i = 1:n

~~~~A(i,:,:) = tensor(symtensor(A(i,:,:)));

~~~~B(i,:,:) = tensor(symtensor(B(i,:,:)));

end

\%\% Example 2

\% load("A(example 2 random 50x50x50x50).mat");

\% load("B(example 2 random 50x50x50x50).mat");

\% m = length(size(A));

\% N = size(A);

\% n = N(1);

\% for i = 1:n

\%~~~~A(i,:,:,:) = tensor(symtensor(A(i,:,:,:)));

\%~~~~B(i,:,:,:) = tensor(symtensor(B(i,:,:,:)));

\% end

\%\% Example 3 (random 4x4x4x4x4x4 symmetric)

\% load("A(example 3 random 4x4x4x4x4x4 symmetric).mat");

\% load("B(example 3 random 4x4x4x4x4x4 symmetric).mat");

\% m = length(size(A));

\% N = size(A);

\% n = N(1);

\%\% Example 4 (hypergraph 50x50x50x50)

\% load("A(example 4 hypergraph 50x50x50x50).mat");

\% load("B(example 4 hypergraph 50x50x50x50).mat");

\% m = length(size(A));

\% N = size(A);

\% n = N(1);

\%\% Initialization

tic;

tol = 1e-13;

C = double(B - A);

D = zeros(n$*$ones(1,m));

d = zeros(n,1);

for i = 1 : n

~~~~d(i) = C(i,i,i); \% for m=3

\%~~~~d(i) = C(i,i,i,i); \% for m=4

~~~~D(i,i,i) = d(i); \% for m=3

\%~~~~D(i,i,i,i) = d(i); \% for m=4

end

b = ones(n,1);

x = ones(n,1);

x\_old = x;

Temp = tensor(D - C);

C = tensor(C);

b\_temp = double(ttsv(Temp,x,-1)) + b;

x\_new = (b\_temp ./ d) .\^\, (1/(m-1));

while norm(x\_new - x\_old,2) $>$ tol

~~~~x\_old = x\_new;

~~~~b\_temp = double(ttsv(Temp,x\_old,-1)) + b;

~~~~x\_new = (b\_temp ./ d) .\^\, (1/(m-1));

end

x = x\_new / norm(x\_new,2);

T = double(ttsv(A,x,-2));

T\_B = double(ttsv(B,x,-2));

temp1 = T $*$ x;

temp2 = T\_B $*$ x;

rho = max(temp1 ./ temp2);

\%\%

MaxIter = 50;

Differ\_rho = zeros(1,MaxIter);

Res\_rho = zeros(1,MaxIter);

Res = zeros(1,MaxIter);

R = zeros(MaxIter,n);

theta = zeros(1,MaxIter);

for k = 1: MaxIter

~~~~M = rho $*$ B - A;

~~~~b = temp2;

~~~~J\_x = (m - 1) $*$ (rho $*$ T\_B - T);

~~~~w = J\_x $\backslash$ b;

~~~~y = w / norm(w,2);

~~~~theta(k) = 1;

~~~~while 1

~~~~~~~~x\_new = (m-2) $*$ x + theta(k) $*$ y;

~~~~~~~~r = abs(ttsv(M,x\_new,-1));

~~~~~~~~R(k,1:n) = r - theta(k) $*$ b / (2$*$norm(w));  \% m=3

\%~~~~~~~~R(k,1:n) = r - theta(k) $*$ b / (norm(w));  \% m$>=$4

~~~~~~~~if R(k,1:n)$>=$0 \& ttsv(C,x\_new,-1)$>$0

~~~~~~~~~~~~break;

~~~~~~~~else

~~~~~~~~~~~~theta(k) = theta(k)/2;

~~~~~~~~end

~~~~end

~~~~x = x\_new / norm(x\_new,2);

~~~~rho\_old = rho;

~~~~T = double(ttsv(A,x,-2));

~~~~T\_B = double(ttsv(B,x,-2));

~~~~temp1 = T $*$ x;

~~~~temp2 = T\_B $*$ x;

~~~~rho\_max = max(temp1 ./ temp2);

~~~~rho\_min = min(temp1 ./ temp2);

~~~~rho = rho\_max;

~~~~Differ\_rho(k) = rho\_old - rho;

~~~~Res\_rho(k) = abs(rho\_max-rho\_min)/rho\_max;

~~~~Res(k) = norm(temp1-rho\_max$*$temp2,2);

~~~~if Res\_rho(k) $<$ tol

~~~~~~~~break;

~~~~end

\%~~~~if Res(k) $<$ tol    \% for symmetric cases

\%~~~~~~~~break;

\%~~~~end   

end

time = toc;
\subsection*{GNNI\_NEPv.m (for Example 5)}
\

clearvas;

clc;

\%\% Load Example

\%\% Example 5 (Nonlinear eigenvalue--2D case)

load("data(example 5 NEPv).mat");

bb = diag(B);

I = eye(n);

tt = max(bb);

\%\% Initialization

tic;

tol = 1e-13;

d = (beta - tt) $*$ ones(n,1) + diag(B);

b = ones(n,1);

x = ones(n,1);

x\_old = x;

temp = x\_old.\^\,(m-1);

b\_temp = bb .$*$ temp - tt $*$ temp - norm(x\_old)\^\,(m-2) $*$ (B $*$ x\_old - tt $*$ x\_old) + b;

x\_new = (b\_temp ./ d) .\^\, (1/(m-1));

while norm(x\_new - x\_old,2) $>$ 1e-1

~~~~x\_old = x\_new;

~~~~temp = x\_old.\^\,(m-1);

~~~~b\_temp = bb .$*$ temp - tt $*$ temp - norm(x\_old)\^\,(m-2) $*$ (B $*$ x\_old - tt $*$ x\_old) + b;
    
~~~~x\_new = (b\_temp ./ d) .\^\, (1/(m-1));

end

x = x\_new / norm(x\_new,2);

temp1 = tt $*$ x;

temp2 = beta $*$ x.\^\,(m-1) + B $*$ x;

rho = max(temp1 ./ temp2);

\%\%

MaxIter = 50;

Res\_rho = zeros(MaxIter,1);

Res = zeros(MaxIter,1);

R = zeros(MaxIter,n);

theta = zeros(MaxIter,1);

for k = 1: MaxIter

~~~~b = temp2;

~~~~T = 1/3 $*$ I + 2/3 $*$ (x $*$ x');

~~~~T\_B = beta $*$ diag(x.\^\,2) + 1/3$*$B + 2/3$*$(B$*$x)$*$x';

~~~~T = tt $*$ T;

~~~~J\_x = (m - 1) $*$ (rho $*$ T\_B - T);

~~~~w = bicgstab(J\_x,b,1e-9,200);

~~~~y = w / norm(w,2);

~~~~theta(k) = 1;

~~~~x\_new = (m-2) $*$ x + theta(k) $*$ y;
    
~~~~x = x\_new / norm(x\_new,2);
    
~~~~rho\_old = rho;
    
~~~~temp1 = tt $*$ x;
    
~~~~temp2 = beta $*$ x.\^\,(m-1) + B $*$ x;
    
~~~~rho\_max = max(temp1 ./ temp2);
    
~~~~rho\_min = min(temp1 ./ temp2);
    
~~~~rho = rho\_max;

~~~~Res\_rho(k) = abs(rho\_max-rho\_min)/rho\_max;

~~~~Res(k) = norm(temp1-rho\_max$*$temp2,2);

~~~~if Res\_rho(k) $<$ tol

~~~~~~~~break;

~~~~end

\%~~~~if Res(k) $<$ tol    \% for symmetric cases

\%~~~~~~~~break;

\%~~~~end  

end

time = toc;
\subsection*{GNNI\_MGPE.m (for Example 6)}
\

clearvas;

clc;

\%\% Load Example

\%\% Example 6 (Nonlinear eigenvalue--MGPE)

load("data(example 6 MGPE).mat");

bb = diag(B);

tt = max(bb);

aa = diag(A);

I = eye(n);

Beta = beta $*$ I + 2 $*$ alpha $*$ A;

\%\% Initialization

tic;

tol = 1e-12;

d = (beta - tt) $*$ ones(n,1) + bb + 2 $*$ alpha $*$ aa;

x = ones(n,1);

x = x / norm(x);

temp1 = tt $*$ x;

temp2 = beta $*$ x.\^\,(m-1) + 2$*$alpha$*$A$*$x.\^\,(m-1) + B $*$ x;

rho = max(temp1 ./ temp2);

\%\%

MaxIter = 50;

Res\_rho = zeros(1,MaxIter);

Res = zeros(1,MaxIter);

R = zeros(MaxIter,n);

theta = zeros(1,MaxIter);

for k = 1: MaxIter

~~~~b = temp2;

~~~~T\_B = diag(beta$*$xx) + A.$*$repmat(2$*$alpha$*$xx',n,1) + 1/3 $*$ B;

~~~~T\_B = T\_B + 2/3 $*$ (B$*$x) $*$ x';

~~~~T = 1/3 $*$ tt $*$ I + 2/3 $*$ tt $*$ (x $*$ x'); 

~~~~J\_x = (m - 1) $*$ (rho $*$ T\_B - T);

~~~~w = bicgstab(J\_x,b,1e-9,200);

~~~~y = w / norm(w,2);

~~~~theta(k) = 1;

~~~~x\_new = (m-2) $*$ x + theta(k) $*$ y;

~~~~x = x\_new / norm(x\_new,2);

~~~~rho\_old = rho;

~~~~temp1 = tt $*$ x;

~~~~temp2 = beta $*$ x.\^\,(m-1) + 2$*$alpha$*$A$*$x.\^\,(m-1) + B$*$x;

~~~~rho\_max = max(temp1 ./ temp2);

~~~~rho\_min = min(temp1 ./ temp2);

~~~~rho = rho\_max;

~~~~Res\_rho(k) = abs(rho\_max-rho\_min)/rho\_max;

~~~~Res(k) = norm(temp1-rho\_max$*$temp2,2);

~~~~if Res\_rho(k) $<$ tol

~~~~~~~~break;

~~~~end

\%~~~~if Res(k) $<$ tol    \% for symmetric cases

\%~~~~~~~~break;

\%~~~~end   

end

time = toc;
\subsection*{Algorithms for comparison}
\subsubsection*{GEAP.m}
\

clearvas;

clc;

\%\% Load Examples

\%\% Example 3 (random 4x4x4x4x4x4 symmetric)

load("A(example 3 random 4x4x4x4x4x4 symmetric).mat");

load("B(example 3 random 4x4x4x4x4x4 symmetric).mat");

\%\% Example 4 (hypergraph 50x50x50x50)

\% load("A(example 4 hypergraph 50x50x50x50).mat");

\% load("B(example 4 hypergraph 50x50x50x50).mat");

\%\%

tic;

m = length(size(A));

N = size(A);

n = N(1);

x = ones(n,1);

x = x / norm(x,2);

temp1 = double(ttsv(A,x,-2));

temp2 = double(ttsv(B,x,-2));

temp3 = temp1 $*$ x;

temp4 = temp2 $*$ x;

temp5 = x' $*$ temp3;

temp6 = x' $*$ temp4;

beta = 1;

tau = 1e-6;

tol = 1e-13;

MaxIter = 50; \% for example 3

\% MaxIter = 200; \% for example 4

Lambda = zeros(MaxIter,1);

H = zeros(n);

Alpha = zeros(MaxIter,1);

Res = zeros(MaxIter,1);

for k = 1 : MaxIter

~~~~Lambda(k) = temp5 / temp6;

~~~~H = m\^\,2$*$temp5/(temp6)\^\,3$*$2$*$(temp4$*$temp4')+m/
    temp6$*$((m-1)$*$temp1+temp5$*$(eye(n)

~~~~~+(m-2)$*$(x$*$x'))+m$*$(temp3$*$x'+x$*$temp3'))-m/(temp6)\^\,2$*$((m-1)$*$temp5$*$temp2

~~~~~+m$*$(temp3$*$temp4'+temp4$*$temp3')+m$*$temp5$*$(x$*$temp4'+temp4$*$x'));
    
~~~~e = eig(beta$*$H);
    
~~~~Alpha(k) = beta $*$ max(0,(tau-min(e))/m);
    
~~~~x = beta $*$ (temp3-Lambda(k)$*$temp4+(Alpha(k)+Lambda(
    k))$*$temp6$*$x);
    
~~~~x = x / norm(x,2);

~~~~temp1 = double(ttsv(A,x,-2));

~~~~temp2 = double(ttsv(B,x,-2));

~~~~temp3 = temp1 $*$ x;

~~~~temp4 = temp2 $*$ x;

~~~~temp5 = x' $*$ temp3;

~~~~temp6 = x' $*$ temp4;

~~~~Res(k) = norm(temp3-Lambda(k)$*$temp4,2);

~~~~if Res(k) $<$ tol 

~~~~~~~~break;

~~~~end

end

lambda = Lambda(k);

time = toc;
\subsubsection*{AG.m}
\

clearvas;

clc;

\%\% Load Examples

\%\% Example 3 (random 4x4x4x4x4x4 symmetric)

load("A(example 3 random 4x4x4x4x4x4 symmetric).mat");

load("B(example 3 random 4x4x4x4x4x4 symmetric).mat");

\%\% Example 4 (hypergraph 50x50x50x50)

\% load("A(example 4 hypergraph 50x50x50x50).mat");

\% load("B(example 4 hypergraph 50x50x50x50).mat");

\%\%

tic;

m = length(size(A));

N = size(A);

n = N(1);

x = ones(n,1);

x = x/norm(x,2);

Temp1 = ttsv(A,x,-1);

temp1 = x' $*$ Temp1;

Temp2 = ttsv(B,x,-1);

temp2 = x' $*$ Temp2;

rho = 0.001; \% for example 3

\% rho = 0.5; \% for example 4

tol = 1e-13;

MaxIter = 50; \% for example 3

\% MaxIter = 100; \% for example 4

Res = zeros(1,MaxIter);

for k = 1:MaxIter

~~~~lambda = temp1 / temp2;

~~~~g = m/temp2 $*$ (Temp1 - lambda$*$Temp2);

~~~~if k == 1

~~~~~~~~alpha = 1/norm(g);

~~~~else

~~~~~~~~alpha = min(1/norm(g),norm(x-x\_old)/norm(g-g\_old));

~~~~end

~~~~x\_new = sqrt(1-alpha\^\,2$*$norm(g)\^\,2)$*$x+alpha$*$g;

~~~~count = 0;

~~~~while ttsv(A,x\_new,0)/ttsv(B,x\_new,0) $<$ lambda+rho$*$alpha$*$norm(g)\^\,2
    
~~~~~~~~alpha = alpha / 2;
    
~~~~~~~~x\_new = sqrt(1-alpha\^\,2$*$norm(g)\^\,2)$*$x+alpha$*$g;
    
~~~~~~~~count = count + 1;
    
\%~~~~~~~~if count $>=$ 2  \% for example 4
    
\%~~~~~~~~~~~~break;
    
\%~~~~~~~~end
    
~~~~end
    
~~~~Res(k) = norm(Temp1 - lambda $*$ Temp2,2);
    
~~~~g\_old = g;
    
~~~~x\_old = x;
    
~~~~x = x\_new;
    
~~~~Temp1 = ttsv(A,x,-1);
    
~~~~temp1 = x' $*$ Temp1;
    
~~~~Temp2 = ttsv(B,x,-1);
    
~~~~temp2 = x' $*$ Temp2;
    
~~~~lambda\_old = lambda;
    
~~~~lambda = temp1 / temp2;
    
~~~~if Res(k) $<$ tol
    
~~~~~~~~break;
    
~~~~end

end

time = toc;
\subsubsection*{NQZ.m}
\

clearvars;

clc;

\%\% Load Example

\%\% Example 4 (hypergraph 50x50x50x50)

load("A(example 4 hypergraph 50x50x50x50).mat");

m = length(size(A));

N = size(A);

n = N(1);

\%\%

tic;

x = ones(n,1);

x = x / norm(x,2);                                     

y = double(ttsv(A,x,-1));

tol = 1e-13;                                                

MaxIter = 300;

Res\_rho = zeros(1,MaxIter);

Res = zeros(1,MaxIter);

for k = 1 : MaxIter

~~~~x = y.\^\,(1/(m-1));

~~~~x = x / norm(x,2);

~~~~y = double(ttsv(A,x,-1));

~~~~lambda\_ub = max(y ./ x.\^\,(m-1));

~~~~lambda\_lb = min(y ./ x.\^\,(m-1));

~~~~lambda = lambda\_ub;

~~~~Res\_rho(k) = abs(lambda\_ub - lambda\_lb)/lambda\_ub;

~~~~Res(k) = norm(y - lambda $*$ x.\^\,(m-1));

~~~~if Res(k) $<$ tol

~~~~~~~~break;

~~~~end

end

time = toc;
\subsubsection*{NNI\_NEPv.m (for Example 5)}
\

clearvars;

clc;

\%\% Load Example

\%\% Example 5 (Nonlinear eigenvalue--2D case)

load("data(example 5 NEPv).mat");

I = eye(n);

\%\% Initialization

tic;

u = ones(n,1);

u = u / norm(u);

AAu = (beta $*$ diag(u.\^\,2) + B) $*$ u;

lambda = min(AAu./u);

MaxIter = 30;

theta = zeros(MaxIter,1);

Res\_NNI = zeros(MaxIter,1);

\%\%

for k = 1:MaxIter

~~~~Temp = [B-lambda$*$I+3$*$beta$*$diag(u.\^\,2),-u;-u',0];

~~~~temp = [AAu-lambda$*$u;1/2$*$(1-u'$*$u)];

~~~~delta = bicgstab(Temp,-temp,1e-9,200);

~~~~theta(k) = 1;

~~~~while 1

~~~~~~~~w = u + theta(k) $*$ delta(1:n);

~~~~~~~~w = w / norm(w);

~~~~~~~~WW = diag(w.\^\,2);

~~~~~~~~h = (beta$*$WW + B - lambda$*$I) $*$ w;

~~~~~~~~if h $>$ 0

~~~~~~~~~~~~break;

~~~~~~~~else

~~~~~~~~~~~~theta(k) = theta(k) / 2;

~~~~~~~~end

~~~~end

~~~~u = w;

~~~~AAu = beta $*$ diag(u.\^\,2) $*$ u + B $*$ u;

~~~~lambda = min(AAu./u);

~~~~Res\_NNI(k) = norm(AAu-lambda$*$u) / ((kappa+abs(lambda))$*$norm(u));
    
~~~~if Res\_NNI(k) $<$ 1e-11
    
~~~~~~~~break;
    
~~~~end

end

time = toc;
\subsubsection*{NNI\_MGPE.m (for Example 6)}
\

clearvars;

clc;

\%\% Load Example

\%\% Example 6 (Nonlinear eigenvalue--MGPE)

load("data(example 6 MGPE).mat");

I = eye(n);

\%\% Initialization

tic;

u = ones(n,1);

u = u / norm(u);

uuu = u.\^\,3;

AAu = beta $*$ uuu + 2 $*$ alpha $*$ A $*$ uuu + B $*$ u;

lambda = min(AAu./u);

MaxIter = 30;

theta = zeros(MaxIter,1);

Res\_NNI = zeros(MaxIter,1);

\%\%

for k = 1:MaxIter

~~~~uu = u.\^\,2;

~~~~Temp = [B-lambda$*$I+diag(3$*$beta$*$uu)+A.$*$repmat(6$*$alpha$*$uu',n,1),-u;-u',0];
    
~~~~temp = [AAu-lambda$*$u;1/2$*$(1-u'$*$u)];

~~~~delta = bicgstab(Temp,-temp,1e-9,200);

~~~~theta(k) = 1;

~~~~while 1

~~~~~~~~w = u + theta(k) $*$ delta(1:n);

~~~~~~~~w = w / norm(w);

~~~~~~~~www = w.\^\,3;

~~~~~~~~h = beta$*$www + 2$*$alpha$*$A$*$www + (B - lambda$*$I)$*$w;

~~~~~~~~if h $>$ 0

~~~~~~~~~~~~break;

~~~~~~~~else

~~~~~~~~~~~~theta(k) = theta(k) / 2;

~~~~~~~~end

~~~~end

~~~~u = w;

~~~~uuu = u.\^\,3;

~~~~AAu = beta $*$ uuu + 2 $*$ alpha $*$ A $*$ uuu + B $*$ u;

~~~~lambda = min(AAu./u);

~~~~Res\_NNI(k) = norm(AAu-lambda$*$u) / ((kappa+abs(lambda))$*$norm(u));

~~~~if Res\_NNI(k) $<$ 1e-10

~~~~~~~~break;

~~~~end

end

time = toc;
\subsubsection*{NRI\_NEPv.m (for Example 5)}
\

clearvars;

clc;

\%\% Load Example

\%\% Example 5 (Nonlinear eigenvalue--2D case)

load("data(example 5 NEPv).mat");

I = eye(n);

\%\% Initialization

tic;

u = ones(n,1);

lambda = 1000;

MaxIter = 30;

Res\_NRI = zeros(MaxIter,1);

count = zeros(MaxIter,1);

\%\%

for k = 1:MaxIter

~~~~u = ones(n,1);

~~~~for l = 1:100

~~~~~~~~u\_old = u;

~~~~~~~~Temp = 3 $*$ beta $*$ diag(u.\^\,2) + B - lambda $*$ I;

~~~~~~~~temp = 2 $*$ beta $*$ u.\^\,3;

~~~~~~~~u = triblocksolve(Temp,temp,N);

~~~~~~~~count(k) = count(k) + 1;

~~~~~~~~if (norm(u-u\_old)+norm(beta$*$u.\^\,3+B$*$u-lambda$*$u))/norm(u) $<$ 1e-10
        
~~~~~~~~~~~~break;
        
~~~~~~~~end        
~~~~end
    
~~~~Temp = 3 $*$ beta $*$ diag(u.\^\,2) + B - lambda $*$ I;
    
~~~~temp = triblocksolve(Temp,u,N);
    
~~~~lambda = lambda - abs(u'$*$u-1) / (2$*$u'$*$temp);
   
~~~~Res\_NRI(k) = norm(beta$*$u.\^\,3+B$*$u-lambda$*$u);
   
~~~~if Res\_NRI(k) $<$ 1e-11
   
~~~~~~~~break;
   
~~~~end

end

time = toc;

\end{document}